\documentclass[reqno,a4paper,twoside,draft]{amsart}
\usepackage{enumitem}
\setenumerate{label=\textnormal{(\arabic*)}}

\allowdisplaybreaks[3]

\usepackage{amsmath,amssymb,dsfont,mathrsfs,verbatim,bm,geometry}
\usepackage{mathtools}
\usepackage[latin1]{inputenc}
\usepackage{array,hhline}
\usepackage[raggedright]{titlesec}
\usepackage{tikz}
\usepackage{float,subfig}
\usepackage{amsrefs}

\titleformat{\section}
{\normalfont\Large\bfseries\center}{\thesection}{1em}{}
\titleformat{\subsection}
{\normalfont\large\bfseries}{\thesubsection}{1em}{}
\titleformat{\subsubsection}
{\normalfont\normalsize\bfseries}{\thesubsubsection}{1em}{}
\titleformat{\paragraph}[runin]
{\normalfont\normalsize\bfseries}{\theparagraph}{1em}{}
\titleformat{\subparagraph}[runin]
{\normalfont\normalsize\bfseries}{\thesubparagraph}{1em}{}

\titlespacing*{\chapter} {0pt}{50pt}{40pt}
\titlespacing*{\section} {0pt}{3.5ex plus 1ex minus .2ex}{2.3ex plus .2ex}
\titlespacing*{\subsection} {0pt}{3.25ex plus 1ex minus .2ex}{1.5ex plus .2ex}
\titlespacing*{\subsubsection}{0pt}{3.25ex plus 1ex minus .2ex}{1.5ex plus .2ex}
\titlespacing*{\paragraph} {0pt}{3.25ex plus 1ex minus .2ex}{1em}
\titlespacing*{\subparagraph} {\parindent}{3.25ex plus 1ex minus .2ex}{1em}

\newtheorem{theorem}{Theorem}[section]
\newtheorem{lemma}[theorem]{Lemma}
\newtheorem{proposition}[theorem]{Proposition}
\newtheorem{corollary}[theorem]{Corollary}

\theoremstyle{definition}
\newtheorem{definition}[theorem]{Definition}

\newtheorem{example}[theorem]{Example}

\theoremstyle{remark}
\newtheorem{remark}[theorem]{Remark}

\DeclareMathOperator{\Aut}{Aut}

\DeclareMathOperator{\coH}{coH}
\DeclareMathOperator{\ide}{id}
\DeclareMathOperator{\ima}{Im}

\DeclareMathOperator{\Prim}{Prim}

\DeclareMathOperator{\Z}{Z}

\DeclareMathOperator{\op}{op}

\DeclareMathOperator{\ord}{ord}

\newcommand{\ov}{\overline}
\newcommand{\ot}{\otimes}
\newcommand{\wh}{\widehat}
\newcommand{\wt}{\widetilde}
\newcommand{\ep}{\epsilon}
\newcommand{\De}{\Delta}
\newcommand{\hs}{\hspace{-0.5pt}}

\newcommand{\xcirc}{\hs \circ \hs}

\begin{document}

\title{Cleft extensions of a family of braided Hopf algebras}

\author{Mauricio Da Rocha}
\address{C\'\i clo B\'asico Com\'un\\ Pabell\'on 3 - Ciudad Universitaria\\ (1428) Buenos
Aires, Argentina.} \curraddr{} \email{}

\thanks{Mauricio Da Rocha was supported by UBACyT 20020110100048 (UBA)}

\author{Jorge A. Guccione}
\address{Departamento de Matem\'atica\\ Facultad de Ciencias Exactas y Naturales-UBA, Pabell\'on~1-Ciudad Universitaria\\ Intendente Guiraldes 2160 (C1428EGA) Buenos Aires, Argentina.}
\address{Instituto de Investigaciones Matem\'aticas ``Luis A. Santal\'o"\\ Pabell\'on~1-Ciudad Universitaria\\ Intendente Guiraldes 2160 (C1428EGA) Buenos Aires, Argentina.}
\email{vander@dm.uba.ar}

\author{Juan J. Guccione}
\address{Departamento de Matem\'atica\\ Facultad de Ciencias Exactas y Naturales-UBA\\ Pabell\'on~1-Ciudad Universitaria\\ Intendente Guiraldes 2160 (C1428EGA) Buenos Aires, Argentina.}
\address{Instituto Argentino de Matem\'atica-CONICET\\ Savedra 15 3er piso\\ (C1083ACA) Buenos Aires, Argentina.}
\email{jjgucci@dm.uba.ar}

\thanks{Jorge A. Guccione and Juan J. Guccione were supported by UBACyT 20020110100048 (UBA) and PIP 11220110100800CO (CONICET)}

\subjclass[2010]{Primary 16S40; Secondary 16T05}
\keywords{Braided Hopf algebras, rank 1 Hopf algebras, Cleft extensions}

\begin{abstract} We introduce a family of braided Hopf algebras that generalizes the rank $1$ Hopf algebras introduced by Krop anad Radford and we study its cleft extensions.
\end{abstract}

\maketitle

\section*{Introduction} Let $k$ be a field. In the paper~\cite{K-R}, the authors associated a Hopf algebra $H_\mathcal{D}$ over $k$, with each data $\mathcal{D}:=(G,\chi,z,\lambda)$, consisting of:

\begin{itemize}

\smallskip

\item[-] a finite group $G$,

\smallskip

\item[-] a character $\chi$ of $G$ with values in $k$,

\smallskip

\item[-] a central element $z$ of $G$,

\smallskip

\item[-] an element $\lambda\in k$,

\smallskip

\end{itemize}
such that $\chi^n=1$ or $\lambda(z^n-1_G)=0$, where $n$ is the order of $\chi(z)$. As an algebra $H_{\mathcal{D}}$ is generated by $G$ and $x$, subject to the group relations of $G$,
$$
xg=\chi(g)gx\quad\text{for all $g\in G$}\qquad\text{and}\qquad x^n=\lambda(z^n-1_G).
$$
The coalgebra structure of $H_{\mathcal{D}}$ is determined by
\begin{align*}
&\Delta(g):=g\ot g\quad\text{for $g\in G$,}&& \Delta(x):=1\ot x + x\ot z,\\
&\epsilon(g):=1\quad \text{for $g\in G$,} &&\epsilon(x):=0,
\end{align*}
and its antipode is given by
$$
S(g)=g^{-1}\quad\text{and}\quad S(x)=-xz^{-1}.
$$
As was point out in~\cite{K-R}, as a vector space $H_{\mathcal{D}}$ has basis $\{gx^m| g \in G, 0 \le m < n\}$. Consequently $\dim H_{\mathcal{D}}=n|G|$.

The algebras $H_{\mathcal{D}}$ are called {\em rank~$1$ Hopf algebras}. The simplest examples are the Taft algebras $H_{n^2}$, which are the rank~$1$ Hopf algebra obtained taking $\mathcal{D}:=(C_n,\chi,g,1)$, where $C_n = \langle g \rangle$ is the cyclic group of order $n>1$ and $\chi(g)$ is a primitive $n$-th root of~$1$.

In~\cite{M} Masuoka studied the cleft extensions of the Taft algebras, giving a very elegant description of its crossed product systems. In \cite{D-T} the Masuoka description was reproduced with simplified proofs, and studying this description were derived several interesting consequences. Motivated by these works, in this paper we introduce a family of braided Hopf algebras that generalize the algebras $H_{\mathcal{D}}$, and we study their cleft extensions.

\smallskip

The paper is organized in the following way:

In Section~1 we recall the well know notions of Gaussian binomial coefficients and braided Hopf algebra, and we make a quick review of some results obtained in \cite{G-G}. The unique new result in this section are the formulas obtained in Theorem~\ref{equiv entre cleft, H-Galois con base normal e isomorfo a un producto cruzado}(5). In Section~2 we associate a braided Hopf algebra with each data $\mathcal{D}=(G,\chi,z,\lambda,q)$ (where $G$ is a finite group, $\chi$ is a character of $G$, $z$ is a central element of $G$ and $q,\lambda$ are scalars) that satisfies suitable conditions. The main result is Corollary~\ref{algebras de Hopf trenzadas de K-R}, in which the algebras $H_{\mathcal{D}}$ are introduced. When $q=1$, we recover the rank~$1$ Hopf algebras defined by Krop y Radford. In Section~3 we describy the $H_{\mathcal{D}}$-space structures, the $H_{\mathcal{D}}$-comodule structures and the $H_{\mathcal{D}}$-comodule algebra structures (in Proposition~\ref{estructuras trenzadas de HsubD}, Corollary~\ref{segunda caracterizacion de HsubD comodules} and Theorem~\ref{caracterizacion de HsubD comodule algebras}, respectively). In Section~4 we characterize the $H_{\mathcal{D}}$ cleft extensions, and, finally, in Section~5 we study two particular examples.

\section{Preliminaries} In this paper $k$ is a field, we work in the category of $k$-vector spaces, and consequently all the maps are $k$-linear maps. Moreover we let $U\ot V$ denote the tensor product $U\ot_k V$ of each pair of vector spaces $U$ and $V$. We assume that the algebras are associative unitary and the coalgebras are coassociative counitary. For an algebra $A$ and a coalgebra $C$, we let $\mu\colon A\ot A\to A$, $\eta\colon k\to A$, $\Delta\colon C\to C\ot C$ and $\eta\colon C\to k$ denote the multiplication map, the unit, the comultiplication map and the counit, respectively, specified with a subscript if necessary.

\subsection{Gaussian binomial coefficients} Let $q$ be a variable. For any $j\in \mathds{N}_0$ set
$$
(j)_q := \sum_{i=0}^{j-1} q^i = \frac{q^j-1}{q-1}\quad\text{and}\quad (j)!_q := (1)_q(2)_q\cdots (j)_q = \frac{(q-1)(q^2-1) \cdots (q^j-1)}{(q-1)^j}.
$$
The Gauss binomials are the rational functions in $q$ defined by
\begin{equation}
\binom{i}{j}_q :=\frac{(i)!_q}{(j)!_q (i-j)!_q}\qquad\text{for $0\le j\le i$.}\label{eq3}
\end{equation}
A direct computation shows that
$$
\binom{r}{0}_q = \binom{i}{i}_q = 1\quad\text{and}\quad \binom{i}{j}_q = q^{i-j}\binom{i-1}{j-1}_q + \binom{i-1}{j}_q \quad\text{for $0<j<i$.}
$$
From these equalities it follows easily that the Gauss binomials are actually polynomials. The Gauss binomials can be evaluated in arbitrary elements of $k$, but the equality~\eqref{eq3} only makes sense for $q=1$ and for $q\ne 1$ such that $q^l\ne 1$ for all $l\le \max(j,i-j)$. We will need the following well known result ($q$-binomial formula). Let $B$ be a $k$-algebra and $q\in k$. If $x,y\in B$ satisfy $yx = qxy$, then
\begin{equation}
(x+y)^i = \sum_{j=0}^i \binom{i}{j}_q x^jy^{i-j}\qquad\text{for each $i\ge 0$.}\label{eq12}
\end{equation}
Let $i,j\ge 0$ and let $0\le l\le i+j$. Using the equality~\eqref{eq12} to compute $(x+y)^i(x+y)^j$ in two different ways and comparing coefficients we obtain that
\begin{equation}
\binom{i+j}{l}_q=\sum_{\substack{ 0\le s \le i\\ 0\le t \le j\\ s+t=l }} q^{(i-s)t}\binom{i}{s}_q\binom{j}{t}_q.\label{eq13}
\end{equation}

\subsection{Braided Hopf algebras} Let $V$, $W$ be vector spaces and let $c\colon V\ot W \to W\ot V$ be a map. Recall that:

\begin{itemize}

\smallskip

\item[-] If $V$ is an algebra, then $c$ is compatible with the algebra structure of $V$ if
$$
\quad c \xcirc(\eta\ot W)= W\ot \eta\quad\text{and}\quad c \xcirc (\mu\ot W)= (W\ot \mu)\xcirc(c\ot V) \xcirc (V\ot c).
$$

\smallskip

\item[-] If $V$ is a coalgebra, then $c$ is compatible with the coalgebra structure of $V$ if
$$
\quad (W\ot \ep)\xcirc c = \ep\ot W\quad\text{and}\quad (W\ot \De) \xcirc c = (c\ot V)\xcirc (V\ot c)\xcirc (\De \ot W).
$$

\smallskip

\end{itemize}
More precisely, the first equality in the first item says that $c$ is {\em compatible with the unit of $V$} and the second one says that it is {\em compatible with the multiplication map of $V$}, while the first equality in the second item says that $c$ is {\em compatible with the counit of $V$} and the second one says that it is {\em compatible with the comultiplication map of $V$}. Of course, there are similar compatibilities when $W$ is an algebra or a coalgebra.

\smallskip

\begin{definition}\label{def: braided bialgebra} A {\em braided bialgebra} is a vector space $H$ endowed with an algebra structure, a coalgebra structure and a braiding operator $c\in \Aut_k(H\ot H)$, called the {\em braid} of $H$, such that $c$ is compatible with the algebra and coalgebra structures of $H$, $\eta$ is a coalgebra morphism, $\ep$ is an algebra morphism and
$$
\De\xcirc\mu = (\mu\ot \mu)\xcirc(H\ot c \ot H)\xcirc(\De \ot \De).
$$
Furthermore, if there exists a map $S\colon H\to H$, which is the inverse of the identity map for the convolution product, then we say that $H$ is a {\em braided Hopf algebra} and we call $S$ the {\em antipode} of~$H$.
\end{definition}

Usually $H$ denotes a braided bialgebra, understanding the structure maps, and $c$ denotes its braid. If necessary, we will write $c_H$ instead of $c$.

\smallskip

Let $A$ and $B$ be algebras. It is well known that if a map $c \colon B\ot A \longrightarrow A\ot B$ is compatible with the algebra structures of $A$ and $B$, then $A\ot B$ endowed with the multiplication map
$$
\mu := (\mu_A\ot \mu_B)\xcirc (A\ot c\ot B),
$$
is an associative algebra with unit $1\ot 1$, which is called {\em the twisted tensor product of $A$ with $B$ associated with $c$} and denoted $A\ot_c B$. Similarly, if $C$ and $D$ are coalgebras and $c\colon C\ot D \longrightarrow D\ot C$ is a map compatible with the coalgebra structures of $C$ and $D$, then $C\ot D$ endowed with the comultiplication map
$$
\Delta := (C\ot c \ot D)\xcirc (\Delta_C\ot \Delta_D),
$$
is a coassociative coalgebra with counit $\epsilon\ot \epsilon$, which is called {\em the twisted tensor product of $C$ with $D$ associated with $c$} and denoted $C\ot^c D$.

\begin{remark} Let $H$ be a vector space which is both an algebra and a coalgebra, and let
$$
c\colon H\ot H\longrightarrow H\ot H
$$
be a braiding operator. Assume that $c$ is compatible with the algebra and the coalgebra structures of $H$. Then $H$ is a braided bialgebra iff its comultiplication map $\Delta \colon H\to H\ot_c H$ and its counit $\epsilon \colon H\to k$ are algebra maps.
\end{remark}

\subsection{Left $\bm{H}$-spaces, left $\bm{H}$-algebras and left $\bm{H}$-coalgebras}

\begin{definition}\label{def: H-braided space} Let $H$ be a braided bialgebra. A {\em left $H$-space} $(V,s)$ is a vector space $V$, endowed with a bijective map $s\colon H\ot V \longrightarrow V\ot H$, called the {\em left transposition of $H$ on $V$}, which is compatible with the bialgebra structure of $H$ and satisfies
$$
(s\ot H)\xcirc (H\ot s)\xcirc (c\ot V) = (V\ot c)\xcirc (s\ot H)\xcirc (H\ot s)
$$
(compatibility of $s$ with the braid). Sometimes, when it is evident, the map $s$ is not explicitly specified. In these cases we will say that $V$ is a left $H$-braided space in order to point out that there is a left transposition involved in the definition. We adopt a similar convention for all the definitions below. Let $(V',s')$ be another left $H$-space. A $k$-linear map $f\colon V \to V'$ is said to be a {\em morphism of left $H$-spaces}, from $(V,s)$ to $(V',s')$, if $(f\ot H)\xcirc s = s' \xcirc (H\ot f)$.
\end{definition}

\begin{remark}\label{basta verificar sobre generadores} Let $s\colon H\ot V \longrightarrow V\ot H$ be a map compatible with the unit, the mul\-tiplication map and the braid of $H$ and let $X\subseteq H$ be a set that generates $H$ as an algebra. In order to show that $s$ is a left transposition it suffices to check the compatibility of $s$ with the counit and the comultiplication map of $H$ on simple tensors $h\ot v$ with $h\in X$ and $v\in V$.
\end{remark}

We let $\mathcal{LHB}$ denote the category of all left $H$-braided spaces. It is easy to check that this is a monoidal category with

\begin{itemize}

\smallskip

\item[-] unit $(k,\tau)$, where $\tau\colon H\ot k\to k\ot H$ is the flip,

\smallskip

\item[-] tensor product
$$
\qquad\quad (U,s_U)\ot(V,s_V) := (U\ot V, s_{U\ot V}),
$$
where $s_{U\ot V}$ is the map $s_{U\ot V}:= (U\ot s_V)\xcirc(s_U\ot V)$,

\smallskip

\item[-] the usual associativity and unit constraints.

\smallskip

\end{itemize}

\begin{definition}\label{def: braided alg} A {\em left $H$-algebra} $(A,s)$ is an algebra in $\mathcal{LHB}$.
\end{definition}

\begin{definition}\label{def: transp alg} A {\em left transposition of $H$ on an algebra $A$} is a bijective map $s\colon H\ot A\longrightarrow A\ot H$, satisfying

\smallskip

\begin{enumerate}

\item $(A,s)$ is a left $H$-space,

\smallskip

\item $s$ is compatible with the algebra structure of $A$.

\end{enumerate}

\end{definition}

\begin{remark}\label{re: alg in LB_H} A left $H$-algebra is nothing but a pair $(A,s)$ consisting of an algebra $A$ and a left transposition $s\colon H\ot A \longrightarrow A\ot H$. Let $(A',s')$ be another left \mbox{$H$-algebra}. A map $f\colon A\to A'$ is a morphism of left $H$-algebras, from $(A,s)$ to $(A',s')$, iff it is a morphism of standard algebras and $(f\ot H) \xcirc s = s' \xcirc (H\ot f)$.
\end{remark}

\begin{definition}\label{def: braided coalg} A {\em left $H$-coalgebra} $(C,s)$ is a coalgebra in $\mathcal{LHB}$.
\end{definition}

\begin{definition}\label{def: transp coalg} A {\em left transposition of $H$ on a coalgebra $C$} is a bijective map $s\colon H\ot C\longrightarrow C\ot H$, satisfying

\begin{enumerate}

\smallskip

\item $(C,s)$ is a left $H$-space,

\smallskip

\item $s$ is compatible with the coalgebra structure of $C$.

\smallskip

\end{enumerate}
\end{definition}

\begin{remark}\label{re: coalg in LB_H} A left $H$-coalgebra is nothing but a pair $(C,s)$ consisting of a coalgebra $C$ and a left transposition $s\colon H\ot C \longrightarrow C\ot H$. Let $(C',s')$ be another left $H$-co\-algebra. A map $f\colon C\to C'$ is a morphism of left $H$-coalgebras, from $(C,s)$ to $(C',s')$, iff it is a morphism of standard coalgebras and $(f\ot H) \xcirc s = s' \xcirc (H\ot f)$.
\end{remark}

Since $(H,c)$ is an algebra and a coalgebra in $\mathcal{LHB}$, it makes sense to consider $(H,c)$-modules and $(H,c)$-comodules in this monoidal category.

\subsection{Left $\bm{H}$-modules and left $\bm{H}$-module algebras}

\begin{definition}\label{def: H-braided module} We will say that $(V,s)$ is a {\em left $H$-module} to mean that it is a left $(H,c)$-module in $\mathcal{LHB}$.  Notice that the classical left $H$-modules can be identified with the left $H$-modules $(V,s)$ in which $s$ is the flip.
\end{definition}

\begin{remark}\label{re: H-braided module} A left $H$-space $(V,s)$ is a left $H$-module iff $V$ is a standard left $H$-module and
$$
s\xcirc (H\ot \rho) = (\rho\ot H) \xcirc (H\ot s) \xcirc (c\ot V),
$$
where $\rho$ denotes the action of $H$ on $V$. Let $(V',s')$ be another left $H$-module. A map $f\colon V\to V'$ is a {\em morphism of left $H$-modules}, from $(V,s)$ to $(V',s')$, iff it is an $H$-linear map and the equality $(f\ot H)\xcirc s = s'\xcirc (H\ot f)$ holds. We let ${}_H\mathcal{LHB}$ denote the category of left $H$-modules in $\mathcal{LHB}$.
\end{remark}

\begin{proposition}[\cite{G-G}*{Proposition~5.6}]\label{prop: _H sub LB H es monoidal} The category ${}_H\mathcal{LHB}$ is monoidal. Its unit is $(k,\tau)$ endowed with the trivial left $H$-module structure, and the tensor product of the left $H$-modules $(U,s_U)$ and $(V,s_V)$, with actions $\rho_U$ and $\rho_V$ respectively, is the left $H$-space $(U,s_U)\ot (V,s_V)$, endowed with the left \mbox{$H$-module} action given by
$$
\rho_{U\ot V}:= (\rho_U\ot \rho_V)\xcirc(H\ot s_U\ot V)\xcirc (\De_H\ot U\ot V).
$$
The associativity and unit constraints are the usual ones.
\end{proposition}

\begin{definition}\label{def: H-braided mod alg} We say that $(A,s)$ is a {\em left $H$-module algebra} if it is an algebra in ${}_H \mathcal{LHB}$.
\end{definition}

\begin{remark}\label{re: car de H-braid mod alg} $(A,s)$ is a left $H$-module algebra iff the following facts hold:

\begin{enumerate}

\smallskip

\item $A$ is an algebra and a standard left $H$-module,

\smallskip

\item $s$ is a left transposition of $H$ on $A$,

\smallskip

\item $s\xcirc (H\ot \rho) = (\rho\ot H) \xcirc (H\ot s) \xcirc (c\ot A)$,

\smallskip

\item $\rho \xcirc (H\ot\mu_A) = \mu_A\xcirc (\rho\ot \rho)\xcirc (H\ot s\ot A)\xcirc (\De_H \ot A\ot A)$,

\smallskip

\item $\rho(h\ot 1_A) = \ep(h)1_A$ for all $h\in H$,

\smallskip

\end{enumerate}
where $\rho$ denotes the action of $H$ on $A$.

\smallskip

Let $(A',s')$ be another left $H$-module algebra. A map $f\colon A\to A'$ is a {\em morphism of left \mbox{$H$-module} algebras}, from $(A,s)$ to $(A',s')$, iff it is an $H$-linear morphism of standard algebras that satisfies $(f\ot H)\xcirc s= s'\xcirc (H\ot f)$.
\end{remark}

\subsection{Right $\bm{H}$-comodules and right $\bm{H}$-comodule algebras}

\begin{definition}\label{def: H-braided comodule} We will say that $(V,s)$ is a {\em right $H$-comodule} if it is a right $(H,c)$-comodule in $\mathcal{LHB}$.
\end{definition}

\begin{remark}\label{re: H-braided comodule} A left $H$-space $(V,s)$ is a right $H$-comodule iff $V$ is a standard right $H$-comodule and
\begin{equation}
(\nu\ot H)\xcirc s = (V\ot c)\xcirc (s\ot H)\xcirc (H\ot \nu),\label{eq6}
\end{equation}
where $\nu$ denotes the coaction of $H$ on $V$. Let $(V',s')$ be another right $H$-comodule. A map $f\colon V\to V'$ is a {\em morphism of right $H$-comodules}, from $(V,s)$ to $(V',s')$, iff it is an $H$-colinear map and $(f\ot H)\xcirc s = s'\xcirc (H\ot f)$. We let $\mathcal{LHB}^H$ denote the category of right $H$-comodules in $\mathcal{LHB}$.
\end{remark}

\begin{definition}\label{coinvariantes} Let $(V,s)$ be a right $H$-comodule. An element $v\in V$ is said to be {\em coinvariant} if $\nu(v)=v\ot 1_H$.
\end{definition}

\begin{remark} For each right $H$-comodule $(V,s)$, the set $V^{\coH}$, of coinvariant elements of $V$, is a vector subspace of $V$. Furthermore, $s(H\ot V^{\coH})=V^{\coH}\ot H$, and the pair $(V^{\coH},s_{V^{\coH}})$, where $s_{V^{\coH}}\colon H\ot V^{\coH}\to V^{\coH}\ot H$ is the restriction of $s$, is a left $H$-space.
\end{remark}

\begin{proposition}[\cite{G-G}*{Proposition 5.2}]\label{prop: LHB spu H es monoidal} The category $\mathcal{LHB}^H$ is monoidal. Its unit is $(k,\tau)$, en\-do\-wed with the trivial right $H$-comodule structure, and the tensor product of the right $H$-comodu\-les $(U,s_U)$ and $(V,s_V)$, with coactions $\nu_U$ and $\nu_V$ respectively, is $(U,s_U)\ot (V,s_V)$, endowed with the right $H$-comodule coaction
$$
\nu_{U\ot V}:= (U\ot V\ot \mu_H)\xcirc(U\ot s_V\ot H) \xcirc (\nu_U\ot \nu_V).
$$
The associativity and unit constraints are the usual ones.
\end{proposition}

\begin{definition}\label{def: H-braided comod alg} We say that $(A,s)$ is a {\em right $H$-comodule algebra} if it is an algebra in $\mathcal{LHB}^H$.
\end{definition}

\begin{remark}\label{re: H-braided comod alg} $(A,s)$ is a right $H$-comodule algebra iff the following facts hold:

\begin{enumerate}

\smallskip

\item $A$ is an algebra and a standard right $H$-comodule,

\smallskip

\item $s$ is a left transposition of $H$ on $A$,

\smallskip

\item $(\nu\ot H)\xcirc s = (A\ot c) \xcirc (s\ot H) \xcirc (H\ot \nu)$,

\smallskip

\item $\nu\xcirc \mu_A = (\mu_A\ot \mu_H)\xcirc (A\ot s\ot H)\xcirc (\nu \ot \nu)$,

\smallskip

\item $\nu(1_A) = 1_A\ot 1_H$,

\smallskip

\end{enumerate}
where $\nu$ denotes the coaction of $H$ on $A$.

\smallskip

Let $(A',s')$ be another right $H$-comodule algebra. A map $f\colon A\to A'$ is a {\em morphism of right $H$-comodule algebras}, from $(A,s)$ to $(A',s')$, iff it is an $H$-colinear morphism of standard algebras that satisfies $(f\ot H) \xcirc s = s' \xcirc (H\ot f)$.
\end{remark}

Recall that $A\ot_s H$ denote the algebra with underlying vector space $A\ot H$, multiplication map
$$
\mu_{A\ot_s H}:= (\mu_A\ot \mu_H)\xcirc (A\ot s\ot H)
$$
and unit $1_A\ot 1_H$. Conditions~(4) and~(5) of Remark~\ref{re: H-braided comod alg} say that $\nu\colon A\to A\ot_s H$ is a morphism of algebras.

\subsection{Hopf crossed products and $\bm{H}$-extensions}

\begin{definition}\label{def: weak H-modules} A left $H$-space $(V,s)$, endowed with a map $\rho\colon H\ot V\to V$, is said to be a {\em weak left $H$-module} if

\begin{enumerate}

\smallskip

\item $\rho(1_H\ot v) = v$, for all $v\in V$,

\smallskip

\item $s\xcirc (H\ot\rho) = (\rho\ot H)\xcirc (H\ot s)\xcirc (c\ot V)$.

\end{enumerate}
\end{definition}

The category ${}_{wH}\mathcal{LHB}$, of weak left $H$-modules in $\mathcal{LHB}$, becomes a monoidal category in the same way that ${}_H\mathcal{LHB}$ does. A {\em weak left $H$-module algebra} $(A,s)$ is, by definition, an algebra in~${}_{wH}\mathcal{LHB}$.

\begin{remark}\label{re: weak H-modulo algebra} $(A,s)$ is a left weak $H$-module algebra iff $A$ is an usual algebra, $s$ is a left transposition of $H$ on $A$ and the structure map $\rho$ satisfies the following conditions:

\begin{enumerate}

\smallskip

\item $\rho(1_H\ot a) = a$, for all $a\in A$,

\smallskip

\item $s\xcirc (H\ot\rho) = (\rho\ot H)\xcirc (H\ot s)\xcirc (c\ot A)$,

\smallskip

\item $\rho\xcirc (H\ot\mu_A) =\mu_A\xcirc(\rho\ot\rho)\xcirc (H\ot s\ot A)\xcirc (\De_H\ot A\ot A)$,

\smallskip

\item $\rho(h\ot 1_A)= \ep(h)1_A$, for all $h\in H$.

\end{enumerate}
\end{remark}

Let $A$ be an algebra and $s\colon H\ot A \longrightarrow A\ot H$ a left transposition. A map $\rho\colon H\ot A\to A$ is said to be a {\em weak action} of $H$ on $(A,s)$ or a {\em weak $s$-action} of $H$ on $A$, if it satisfies the conditions of the above remark.

\begin{definition}\label{def: normal, cociclo, condicion de modulo torcido} Let $A$ be an algebra, $s\colon H\ot A\longrightarrow A\ot H$ a left transposition and $\rho\colon H\ot A\to A$ a weak action of $H$ on $(A,s)$. Let $\sigma\colon H\ot H \to A$ be a map. We say that $\sigma$ is {\em normal} if
$$
\sigma(1_H\ot h)=\sigma(h\ot 1_H)= \ep(h)\qquad\text{for all $h\in H$,}
$$
and that $\sigma$ is a {\em cocycle that satisfies the twisted module condition} if
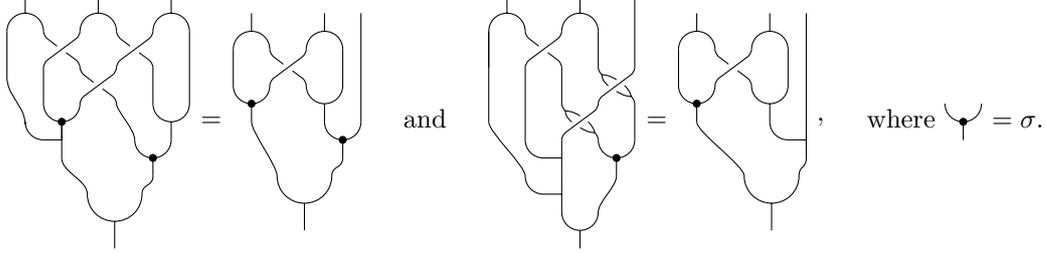
\begin{center}
\begin{tikzpicture}[scale=0.48]
\def\mult(#1,#2)[#3]{\draw (#1,#2) .. controls (#1,#2-0.555*#3/2) and (#1+0.445*#3/2,#2-#3/2) .. (#1+#3/2,#2-#3/2) .. controls (#1+1.555*#3/2,#2-#3/2) and (#1+2*#3/2,#2-0.555*#3/2) .. (#1+2*#3/2,#2) (#1+#3/2,#2-#3/2) -- (#1+#3/2,#2-2*#3/2)}
\def\comult(#1,#2)[#3]{\draw (#1,#2)-- (#1,#2-1*#3/2) (#1,#2-1*#3/2) .. controls (#1+0.555*#3/2,#2-1*#3/2) and (#1+1*#3/2,#2-1.445*#3/2) .. (#1+1*#3/2,#2-2*#3/2) (#1,#2-1*#3/2) .. controls (#1-0.555*#3/2,#2-1*#3/2) and (#1-1*#3/2,#2-1.445*#3/2) .. (#1-1*#3/2,#2-2*#3/2)}
\def\laction(#1,#2)[#3,#4]{\draw (#1,#2) .. controls (#1,#2-0.555*#4/2) and (#1+0.445*#4/2,#2-1*#4/2) .. (#1+1*#4/2,#2-1*#4/2) -- (#1+2*#4/2+#3*#4/2,#2-1*#4/2) (#1+2*#4/2+#3*#4/2,#2)--(#1+2*#4/2+#3*#4/2,#2-2*#4/2)}
\def\cocycle(#1,#2)[#3]{\draw (#1,#2) .. controls (#1,#2-0.555*#3/2) and (#1+0.445*#3/2,#2-#3/2) .. (#1+#3/2,#2-#3/2) .. controls (#1+1.555*#3/2,#2-#3/2) and (#1+2*#3/2,#2-0.555*#3/2) .. (#1+2*#3/2,#2) (#1+#3/2,#2-#3/2) -- (#1+#3/2,#2-2*#3/2) (#1+#3/2,#2-#3/2) node [inner sep=0pt,minimum size=3pt,shape=circle,fill] {}}
\def\braid(#1,#2)[#3]{\draw (#1+1*#3,#2) .. controls (#1+1*#3,#2-0.05*#3) and (#1+0.96*#3,#2-0.15*#3).. (#1+0.9*#3,#2-0.2*#3) (#1+0.1*#3,#2-0.8*#3)--(#1+0.9*#3,#2-0.2*#3) (#1,#2-1*#3) .. controls (#1,#2-0.95*#3) and (#1+0.04*#3,#2-0.85*#3).. (#1+0.1*#3,#2-0.8*#3) (#1,#2) .. controls (#1,#2-0.05*#3) and (#1+0.04*#3,#2-0.15*#3).. (#1+0.1*#3,#2-0.2*#3) (#1+0.1*#3,#2-0.2*#3) -- (#1+0.37*#3,#2-0.41*#3) (#1+0.62*#3,#2-0.59*#3)-- (#1+0.9*#3,#2-0.8*#3) (#1+1*#3,#2-1*#3) .. controls (#1+1*#3,#2-0.95*#3) and (#1+0.96*#3,#2-0.85*#3).. (#1+0.9*#3,#2-0.8*#3)}
\def\transposition(#1,#2)[#3]{\draw (#1+#3,#2) .. controls (#1+#3,#2-0.05*#3) and (#1+0.96*#3,#2-0.15*#3).. (#1+0.9*#3,#2-0.2*#3) (#1+0.1*#3,#2-0.8*#3)--(#1+0.9*#3,#2-0.2*#3) (#1+0.1*#3,#2-0.2*#3) .. controls (#1+0.3*#3,#2-0.2*#3) and (#1+0.46*#3,#2-0.31*#3) .. (#1+0.5*#3,#2-0.34*#3)
(#1,#2-1*#3) .. controls (#1,#2-0.95*#3) and (#1+0.04*#3,#2-0.85*#3).. (#1+0.1*#3,#2-0.8*#3) (#1,#2) .. controls (#1,#2-0.05*#3) and (#1+0.04*#3,#2-0.15*#3).. (#1+0.1*#3,#2-0.2*#3) (#1+0.1*#3,#2-0.2*#3) .. controls (#1+0.1*#3,#2-0.38*#3) and (#1+0.256*#3,#2-0.49*#3) .. (#1+0.275*#3,#2-0.505*#3) (#1+0.50*#3,#2-0.66*#3) .. controls (#1+0.548*#3,#2-0.686*#3) and (#1+0.70*#3,#2-0.8*#3)..(#1+0.9*#3,#2-0.8*#3) (#1+0.72*#3,#2-0.50*#3) .. controls (#1+0.80*#3,#2-0.56*#3) and (#1+0.9*#3,#2-0.73*#3)..(#1+0.9*#3,#2-0.8*#3) (#1+#3,#2-#3) .. controls (#1+#3,#2-0.95*#3) and (#1+0.96*#3,#2-0.85*#3).. (#1+0.9*#3,#2-0.8*#3)}
\begin{scope}[xshift=0cm,yshift=0cm]
\comult(0,0)[1];\comult(2,0)[1];\comult(4,0)[1];\braid(0.5,-1)[1]; \braid(2.5,-1)[1]; \braid(1.5,-2)[1];
\draw[rounded corners=3pt] (-0.5,-1) -- (-0.5,-2.5) -- (0,-3.35) -- (0,-3.5); \draw (0.5,-2) -- (0.5,-3);\draw (3.5,-2) -- (3.5,-3);\draw (4.5,-1) -- (4.5,-3); \cocycle(0.5,-3)[1]; \mult(3.5,-3)[1]; \laction(0,-3.5)[0,1]; \cocycle(3,-4)[1]; \draw (2.5,-3) .. controls (2.5,-3.2) and (3,-3.7) .. (3,-4); \draw (1,-4.5) .. controls (1,-4.9) and (1.7,-5.10) .. (1.7,-5.5); \draw (3.5,-5) .. controls (3.5,-5.3) and (3.2,-5.2) .. (3.2,-5.5); \mult(1.7,-5.5)[1.5];
\end{scope}
\begin{scope}[xshift=0cm,yshift=0cm]
\node at (5.1,-3.5){=};
\end{scope}
\begin{scope}[xshift=1.2cm,yshift=-0.5cm]
\comult(5,0)[1];\comult(7,0)[1]; \braid(5.5,-1)[1]; \draw (8,0) -- (8,-3); \draw (4.5,-1) -- (4.5,-2); \draw (7.5,-1) -- (7.5,-2); \cocycle(4.5,-2)[1]; \mult(6.5,-2)[1]; \cocycle(7,-3)[1]; \mult(5.7,-4.5)[1.5];
\draw (7.5,-4) .. controls (7.5,-4.3) and (7.2,-4.2) .. (7.2,-4.5);
 \draw (5,-3).. controls (5,-3.3) and (5.7,-4.2) .. (5.7,-4.5);
\end{scope}
\begin{scope}[xshift=0cm,yshift=0cm]
\node at (10.95,-3.4){and};
\end{scope}
\begin{scope}[xshift=1.2cm,yshift=0cm]
\comult(12,0)[1];\comult(14,0)[1];\draw (15.5,0) -- (15.5,-2); \braid(12.5,-1)[1]; \draw (11.5,-1) -- (11.5,-2); \draw (14.5,-1) -- (14.5,-2); \transposition(14.5,-2)[1]; \draw (13.5,-2) -- (13.5,-3); \draw (15.5,-3) -- (15.5,-4); \transposition(13.5,-3)[1]; \cocycle(14.5,-4)[1]; \draw (12.5,-2) -- (12.5,-4); \laction(12.5,-4)[0,1]; \draw (11.5,-1) -- (11.5,-3.5); \draw (11.5,-3.5) .. controls (11.5,-3.8) and (12.5,-4.7) .. (12.5,-5); \laction(12.5,-5)[0,1]; \mult(13.5,-6)[1]; \draw (15,-5) .. controls (15,-5.3) and (14.5,-5.7) .. (14.5,-6);
\end{scope}
\begin{scope}[xshift=0.3cm,yshift=0cm]
\node at (17,-3.5){=};
\end{scope}
\begin{scope}[xshift=13.4cm,yshift=-0.5cm]
\comult(5,0)[1];\comult(7,0)[1]; \braid(5.5,-1)[1]; \draw (8,0) -- (8,-4); \draw (4.5,-1) -- (4.5,-2); \draw (7.5,-1) -- (7.5,-2); \cocycle(4.5,-2)[1]; \mult(6.5,-2)[1]; \laction(7,-3)[0,1];
\draw (8,-4) .. controls (8,-4.3) and (7.8,-4.2) .. (7.8,-4.5);
\draw (5,-3).. controls (5,-3.3) and (6.3,-4.2) .. (6.3,-4.5); \mult(6.3,-4.5)[1.5];
\end{scope}
\begin{scope}[xshift=16cm,yshift=0cm]
\node at (5.8,-3.5){,};
\end{scope}
\begin{scope}[xshift=18.2cm,yshift=0cm]
\node at (5.8,-3.4){where};
\end{scope}
\begin{scope}[xshift=18.2cm,yshift=0cm]
\cocycle(7,-3)[1];
\end{scope}
\begin{scope}[xshift=19.7cm,yshift=0cm]
\node at (7.5,-3.5){$= \sigma$.};
\end{scope}
\end{tikzpicture}
\end{center}
More precisely, the first equality is the cocycle condition and the second one is the twisted module condition. Finally we say that $\sigma$ is {\em compatible with $s$} if it is a map in $\mathcal{LHB}$. In other words, if
$$
(\sigma\ot H)\xcirc(H\ot c)\xcirc(c\ot H) = s\xcirc(H\ot \sigma).
$$
\end{definition}

Let $s\colon H\ot A\to A\ot H$ be a left transposition, $\rho\colon H\ot A\to A$ a weak $s$-action and $\sigma\colon H\ot H\to A$ a normal cocycle compatible with $s$, that satisfies the twisted module condition. Consider the maps $\chi\colon H\ot A\longrightarrow A\ot H$ and $\mathcal{F}\colon H\ot H\longrightarrow A\ot H$ defined by
$$
\chi := (\rho\ot H)\xcirc (H\ot s)\xcirc(\Delta\ot A)\quad\text{and}\quad\mathcal{F}:= (\sigma\ot \mu_H) \xcirc (H\ot c \ot H)\xcirc (\Delta\ot \Delta).
$$

\begin{definition}\label{def de producto cruzado} The {\em crossed product associated with $(s,\rho,\sigma)$} is the $k$-algebra $A\#_{\rho,\sigma}^s H$, whose underlying $k$-vector space is $A\ot H$ and whose multiplication map is
$$
\mu:= (\mu_A\ot H)\xcirc(\mu_A\ot \mathcal{F}) \xcirc (A\ot \chi \ot H).
$$
\end{definition}

From now on, a simple tensor $a\ot h$ of $A\#_{\rho,\sigma}^s H$ will usually be written $a\# h$.

\begin{theorem}[\cite{G-G}*{Theorems 2.3, 6.3 and 9.3}]\label{propiedades basicas de productos cruzados1} The algebra $A\#_{\rho,\sigma}^s H$ is associative and has unity $1_A\# 1_H$.
\end{theorem}

\begin{theorem}[\cite{G-G}*{Propositions 10.3 and 10.4}]\label{propiedades basicas de productos cruzados2} The map
$$
\widehat{s}: H\ot A\#_{\rho,\sigma}^s H \longrightarrow A\#_{\rho,\sigma}^s H \ot H,
$$
defined by $\widehat{s}:= (A\ot c)\xcirc (s\ot H)$ is a left transposition of $H$ on $A\#_{\rho,\sigma}^s H$ and the pair $(A\#_{\rho,\sigma}^s H,\widehat{s})$, endowed with the coaction $\nu_{A\#_{\rho,\sigma}^s H}:=A\ot \Delta$, is a right $H$-comodule algebra.
\end{theorem}

\begin{definition}\label{def: H extension} Let $(B,s)$ be a right $H$-comodule algebra and let $i\colon A\hookrightarrow B$ be an algebra inclusion. We say that $(i\colon A\hookrightarrow B,s)$ is an {\em $H$-extension} of $A$ if $i(A) = B^{\coH}$. Let $(i'\colon A\hookrightarrow B',s')$ be another $H$-extension of $A$. We say that $(i\colon A\hookrightarrow B,s)$ and $(i'\colon A\hookrightarrow B',s')$ are {\em equivalent} if there is a right $H$-comodule algebra isomorphism $f\colon (B,s) \to (B',s')$, which is also a left $A$-module homomorphism.
\end{definition}

\begin{remark} For each $H$-extension $(i\colon A\hookrightarrow B,s)$ of $A$, the map $s_A\colon H\ot A\longrightarrow A\ot H$, induced by $s$, is a left transposition (in other words, $(A,s_A)$ is a left $H$-algebra).
\end{remark}

\begin{example}\label{Los productos cruzados son H extensiones} $(i\colon A\hookrightarrow A\#_{\rho,\sigma}^s H,\widehat{s})$, where $i(a):= a\# 1_H$, is an $H$-extension of $A$.
\end{example}

\begin{definition}\label{def: cleft, de Galois y normal} Let $(i\colon A\hookrightarrow B,s)$ be an $H$-extension. We say that:

\begin{enumerate}

\smallskip

\item $(i,s)$ is {\em cleft} if there is a convolution invertible right $H$-comodule map $\gamma\colon (H,c)\to (B,s)$,

\smallskip

\item $(i,s)$ is {\em $H$-Galois} if the map $\beta_B\colon B\ot_A B \longrightarrow B\ot H$, defined by $\beta_B(b\ot b')= (b\ot 1_H)\nu(b')$, where $\nu$ denotes the coaction of $B$, is bijective,

\smallskip

\item $(i,s)$ has the {\em normal basis property} if there exists a left $A$-linear and right $H$-colinear isomorphism $\phi\colon (A\ot H,\wh{s_A}) \longrightarrow (B,s)$, where the coaction of $A\ot H$ is $A\ot \De$ and $\wh{s_A}= (A\ot c)\xcirc (s_A\ot H)$.

\end{enumerate}
\end{definition}

\begin{definition}\label{definicion de aplicacion cleft} Let $(i\colon A\hookrightarrow B,s)$ be an $H$-extension of $A$. If $(i,s)$ is cleft, then each one of the maps $\gamma$ satisfying the conditions required in item~(1) of Definition~\ref{def: cleft, de Galois y normal} is called a {\em cleft map} of $(i,s)$, and if $(i,s)$ has the normal basis property, then each one of the left $A$-linear right and $H$-colinear isomorphism $\phi\colon (A\ot H,\wh{s_A})\longrightarrow (B,s)$ is called a {\em normal basis} of $B$.
\end{definition}

\begin{remark}[\cite{G-G}*{Section 10}] If $\gamma$ is a {\em cleft map} of $(i\colon A\hookrightarrow B,s)$, then $\gamma(1_H)\in B^{\times}$ and the map $\gamma':=\gamma(1_H)^{-1}\gamma$ is a cleft map that satisfies $\gamma'(1_H)=1_B$.
\end{remark}

\begin{lemma}\label{cae en A} Let $H$ be a braided Hopf algebra and let $(i\colon A\hookrightarrow B,s)$ be a cleft $H$-extension, with a cleft map $\gamma$. The map $f\colon H\ot A\to B$, defined by
$$
f:=\mu_B\xcirc (\mu_B\ot B)\xcirc (\gamma\ot i \ot \gamma^{-1})\xcirc(H\ot s_A)\xcirc (\Delta\ot A)
$$
takes its values in $i(A)$.
\end{lemma}

\begin{proof} Let $\lambda_r^X\colon X\to X\ot k$ be the canonical map. We must prove that
$$
\nu\xcirc f = (f\ot\eta )\xcirc (H\ot\lambda_r^A).
$$
A direct computation shows that
\begin{align*}
\nu\xcirc f & = \nu\xcirc \mu_B\xcirc (\mu_B\ot B)\xcirc (\gamma\ot i \ot \gamma^{-1})\xcirc(H\ot s_A)\xcirc (\Delta\ot A)\\
&= (\mu_B\ot\mu_H)\xcirc (B\ot s\ot H)\xcirc (\nu\ot \nu)\xcirc (\mu_B\ot B)\xcirc (\gamma\ot i \ot \gamma^{-1})\xcirc(H\ot s_A)\xcirc (\Delta\ot A)\\
&= (\mu_B\ot \mu_H)\xcirc (\mu_B\ot s \ot H)\xcirc (B\ot s \ot \nu)\xcirc(\nu\ot i \ot \gamma^{-1}) \xcirc(\gamma\ot s_A)\xcirc (\Delta\ot A)\\
%
%
&= (\mu_B\ot \mu_H) \xcirc (\mu_B\ot s \ot H) \xcirc (B\ot i\ot H \ot \nu)\xcirc (\gamma\ot s_A \ot \gamma^{-1})\xcirc (\Delta \ot s_A)\xcirc (\Delta\ot A)\\
&= (\mu_B\ot\mu_H)\xcirc (B\ot s\ot H)\xcirc (\mu_B\ot H\ot \nu\xcirc \gamma^{-1}) \xcirc (B\ot i\ot\Delta)\xcirc (\gamma\ot s_A) \xcirc(\Delta\ot A)\\
& = (\mu_B\ot H)\xcirc (\mu_B\ot L) \xcirc (B\ot i \ot H)\xcirc (\gamma\ot s_A) \xcirc(\Delta\ot A),
\end{align*}
where
$$
L:= (B\ot\mu_H)\xcirc (s\ot H)\xcirc (H\ot\nu\xcirc \gamma^{-1})\xcirc \Delta.
$$
Since, by \cite{G-G}*{Lemma 10.7},
\begin{align*}
L &=(B\ot \mu_H)\xcirc(s\ot H)\xcirc (H\ot\gamma^{-1}\ot S)\xcirc (H\ot c\xcirc \Delta)\xcirc \Delta\\
&=(\gamma^{-1}\ot \mu_H)\xcirc (c\ot S)\xcirc (H\ot c) \xcirc (H\ot\Delta)\xcirc \Delta\\
&=(\gamma^{-1}\ot \mu_H)\xcirc (c\ot S)\xcirc (H\ot c) \xcirc (\Delta\ot H)\xcirc \Delta\\
&=(\gamma^{-1}\ot H)\xcirc c\xcirc (\mu_H\ot H)\xcirc (H\ot S\ot\!H) \xcirc (\Delta\ot H)\xcirc \Delta\\
&=(\gamma^{-1}\ot H)\xcirc c\xcirc (\eta\xcirc\epsilon\ot H)\xcirc \Delta\\
&= \gamma^{-1}\ot \eta,
\end{align*}
we have
$$
\nu\xcirc f = (\mu_B\ot H)\xcirc (B\ot \gamma^{-1}\ot\eta)\xcirc (\mu_B\ot \lambda_r^H)\xcirc (B\ot i\ot H) \xcirc (\gamma\ot s_A) \xcirc(\Delta\ot A) = (f\ot\eta )\xcirc (H\ot\lambda_r^A),
$$
as desired.
\end{proof}

\begin{theorem}\label{equiv entre cleft, H-Galois con base normal e isomorfo a un producto cruzado} Let $H$ be a braided Hopf algebra and let $(i\colon A\hookrightarrow B,s)$ be an $H$-extension. The following assertions are equivalent:

\smallskip

\begin{enumerate}

\item $(i,s)$ is cleft.

\smallskip

\item $(i,s)$ is $H$-Galois with a normal basis.

\smallskip

\item There is a crossed product $A\#_{\rho,\sigma}^{s_A} H$, with convolution invertible cocycle  $\sigma\colon H\ot^c H\to A$, and a right $H$-comodule algebra isomorphism
$$
(B,s)\longrightarrow (A\#_{\rho,\sigma}^{s_A} H,\wh{s_A}),
$$
which is also left $A$-linear.

\smallskip

\end{enumerate}
Furthermore, if $\gamma$ is a cleft map of $(i,s)$ with $\gamma(1_H) = 1_B$, then

\begin{enumerate}[resume]

\smallskip

\item The map $\phi\colon (A\ot H,\wh{s_A})\longrightarrow (B,s)$, defined by $\phi(a\ot h):= i(a)\gamma(h)$, is a normal basis of $B$.

\smallskip

\item The weak action $\rho$ and the cocycle $\sigma$ are given by
\begin{align}
& i\xcirc \rho= \mu_B\xcirc (\mu_B\ot B)\xcirc(\gamma\ot i\ot \gamma^{-1})\xcirc (H\ot s_A)\xcirc(\Delta \ot A)\label{ecua1}
\shortintertext{and}
& i\xcirc \sigma=\mu_B\xcirc (\mu_B\ot \gamma^{-1})\xcirc(\gamma\ot\gamma\ot \mu_H) \xcirc \Delta_{H\ot^c H}.\label{ecua2}
\end{align}
\end{enumerate}
\end{theorem}

\begin{proof} The equivalence between the first three items is \cite{G-G}*{Theorem~10.6}, and the fourth item was proved in the proof of that Theorem. It remains to check the last one. By item~(4), the discussion below~\cite{G-G}*{Definition 10.5} and the proof of Theorem 10.6 of~\cite{G-G}, we know that $\phi$ is bijective, that
$$
(i\ot H)\xcirc\phi^{-1}(b) = b_{(0)}\gamma^{-1}(b_{(1)})\ot b_{(2)},
$$
and that the maps $\rho\colon H\ot A\to A$ and $\sigma\colon H \ot H \to A$, are given by
$$
\rho(h\ot a):=(A\ot \epsilon)\xcirc \phi^{-1}(\gamma(h)i(a))\quad\text{and}\quad \sigma(h\ot l):= (A\ot \epsilon)\xcirc \phi^{-1}(\gamma(h)\gamma(l)).
$$
We must check that $\rho$ and $\sigma$ satisfy~\eqref{ecua1} and~\eqref{ecua2}, respectively. Let $f$ be as in Lemma~\ref{cae en A} and let $i^{-1}$ be the compositional inverse of $i\colon A\to i(A)$. Since
\begin{align*}
\mu_B\xcirc (\gamma\ot i) & = \mu_B\xcirc (\mu_B\ot \eta_B\xcirc \epsilon)\xcirc (B\ot i\ot H)\xcirc (\gamma\ot s_A)\xcirc (\Delta\ot A)\\
& = \mu_B\xcirc (B\ot\mu_B)\xcirc (B\ot \gamma^{-1}\ot \gamma)\xcirc (\mu_B\ot \Delta) \xcirc (B\ot i\ot H) \xcirc (\gamma\ot s_A)\xcirc (\Delta\ot A)\\
& = \mu_B\xcirc (\mu_B\ot\mu_B)\xcirc (\gamma\ot i\ot \gamma^{-1}\ot \gamma)\xcirc (H\ot s_A\ot H)\xcirc (\Delta\ot s_A)\xcirc (\Delta\ot A)\\
& = \mu_B\xcirc (f\ot\gamma)\xcirc (H\ot s_A)\xcirc (\Delta\ot A),
\end{align*}
and, by Lemma~\ref{cae en A},
$$
\mu_B\xcirc (f\ot \gamma) = \phi\xcirc (i^{-1}\xcirc f\ot H),
$$
we have
\begin{align*}
i\xcirc \rho &= (i\ot \epsilon)\xcirc \phi^{-1}\xcirc \mu_B\xcirc (\gamma\ot i)\\
& =(i\ot \epsilon)\xcirc \phi^{-1}\xcirc \mu_B\xcirc (f\ot \gamma)\xcirc (H\ot s_A) \xcirc (\Delta\ot A)\\
& =(i\ot \epsilon)\xcirc (i^{-1}\xcirc f\ot H)\xcirc (H\ot s_A) \xcirc (\Delta\ot A)\\
& = \mu_B\xcirc (\mu_B\ot B)\xcirc(\gamma\ot i\ot \gamma^{-1})\xcirc (H\ot s_A)\xcirc(\Delta \ot A).
\end{align*}
Finally,
\begin{align*}
i\xcirc \sigma & = (i\ot \epsilon)\xcirc \phi^{-1} \xcirc \mu_B\xcirc (\gamma\ot \gamma)\\
& =\mu_B\xcirc(B\ot \gamma^{-1})\xcirc \nu \xcirc \mu_B\xcirc (\gamma\ot \gamma)\\
& =\mu_B\xcirc (B\ot \gamma^{-1})\xcirc(\mu_B\ot \mu_H)\xcirc (B\ot s \ot B)\xcirc(\nu\ot \nu)\xcirc (\gamma\ot \gamma)\\
&= \mu_B\xcirc (\mu_B\ot \gamma^{-1})\xcirc(\gamma\ot\gamma\ot \mu_H) \xcirc \Delta_{H\ot^c H},
\end{align*}
as desired.
\end{proof}

\begin{remark} In the proof of Theorem 10.6 of~\cite{G-G} was also shown that $\phi\colon A\#_{\rho,\sigma}^{s_A} H \to B$ is an algebra isomorphism.
\end{remark}

\section{A family of braided Hopf algebras}
Let $G$ be a finite group, $\chi\colon G\to k^{\times}$ a character, $n>1$ in $\mathds{N}$, and $U\in kG$, where~$kG$ denotes the group algebra of $G$ with coefficients in $k$. Set $\mathcal{E}:=(G,\chi,U,n)$ and write $U:=\sum_{g\in G}\lambda_g g$.

\begin{proposition}\label{estructura de algebra de B_D} There exists an associative algebra $B_{\mathcal{E}}$ such that

\begin{itemize}

\smallskip

\item[-] $B_{\mathcal{E}}$ is generated by $G$ and an element $x\in B_{\mathcal{E}}\setminus kG$,

\smallskip

\item[-] $\mathcal{B}:=\{gx^i: g\in G \text{ and } 0\le i < n\}$ is a basis of $B_{\mathcal{E}}$ as a $k$-vector space,

\smallskip

\item[-] the multiplication of elements of $\mathcal{B}$ is given by:
$$
\quad gx^ihx^j:= \begin{cases}
                \chi^i(h)gh x^{i+j} &\text{if $i+j<n$,}\\
                \chi^i(h)gh U x^{i+j-n} &\text{if $i+j\ge n$,}
         \end{cases}
$$

\smallskip

\end{itemize}
iff $\lambda_{hgh^{-1}}=\chi^n(h)\lambda_g$ for all $h,g\in G$, and $\chi(g)=1$ for all $g\in G$ such that $\lambda_{g}\ne 0$.
\end{proposition}

\begin{proof} Let $V:=kx_0\oplus\cdots\oplus kx_{n-1}$, where $x_0,\dots, x_{n-1}$ are indeterminate. We will prove the result by showing that there is an associative and unitary algebra $kG\# V$, with underlying vector space $kG\ot V$, whose multiplication map satisfies
\begin{align*}
&(g\ot x_i)(g'\ot x_0)=\chi^i(g')gg'\ot x_i \hspace{7pt} \text{ for all $i$ and all $g\in G$}
\shortintertext{and}
&(1_G\ot x_i)(1_G\ot x_j)= \begin{cases} 1_G\ot x_{i+j} & \text{if $i+j<n$,}\\ U \ot x_{i+j-n} &\text{if $i+j\ge n$,}\end{cases}
\end{align*}
iff

\begin{enumerate}

\smallskip

\item $\chi(g)=1$ for all $g\in G$ such that $\lambda_{g}\ne 0$,

\smallskip

\item $\lambda_{hgh^{-1}}=\chi^n(h)\lambda_g$ for all $h,g\in G$.

\smallskip

\end{enumerate}
By the theory of general crossed products developed in~\cite{Br}, for this it suffices to check that the maps
$$
\Phi\colon V\ot k{G}\longrightarrow kG\ot V\quad\text{and}\quad \mathcal{F}\colon V\ot V \longrightarrow kG\ot V,
$$
given by
$$
\Phi(x_i\ot g):= \chi^i(g)g\ot x_i\quad\text{and}\quad \mathcal{F}(x_i\ot x_j):= \begin{cases} 1_G\ot x_{i+j} & \text{if $i+j<n$,}\\ U \ot x_{i+j-n} &\text{if $i+j\ge n$,}\end{cases}
$$
satisfy
$$
\Phi(x_i\ot 1_G) = 1_G\ot x_i,\quad \Phi(x_0\ot g) = g\ot x_0,\quad \mathcal{F}(x_0\ot x_i)=\mathcal{F}(x_i\ot x_0) = 1_G\ot x_i,
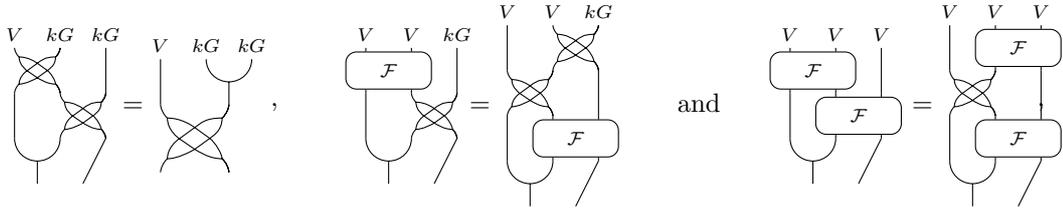
$$
$$
\begin{tikzpicture}[scale=0.6]
\def\mult(#1,#2)[#3]{\draw (#1,#2) .. controls (#1,#2-0.555*#3/2) and (#1+0.445*#3/2,#2-#3/2) .. (#1+#3/2,#2-#3/2) .. controls (#1+1.555*#3/2,#2-#3/2) and (#1+2*#3/2,#2-0.555*#3/2) .. (#1+2*#3/2,#2) (#1+#3/2,#2-#3/2) -- (#1+#3/2,#2-2*#3/2)}
\def\doublemap(#1,#2)[#3]{\draw (#1+0.5,#2-0.5) node [name=doublemapnode,inner xsep=0pt, inner ysep=0pt, minimum height=14pt, minimum width=32pt,shape=rectangle,draw,rounded corners] {$#3$} (#1,#2) .. controls (#1,#2-0.075) .. (doublemapnode) (#1+1,#2) .. controls (#1+1,#2-0.075).. (doublemapnode) (doublemapnode) .. controls (#1,#2-0.925)..(#1,#2-1) (doublemapnode) .. controls (#1+1,#2-0.925).. (#1+1,#2-1)}
\def\doublesinglemap(#1,#2)[#3]{\draw (#1+0.5,#2-0.5) node [name=doublesinglemapnode,inner xsep=0pt, inner ysep=0pt, minimum height=14pt, minimum width=32pt,shape=rectangle,draw,rounded corners] {$#3$} (#1,#2) .. controls (#1,#2-0.075) .. (doublesinglemapnode) (#1+1,#2) .. controls (#1+1,#2-0.075).. (doublesinglemapnode) (doublesinglemapnode)-- (#1+0.5,#2-1)}
\def\twisting(#1,#2)[#3]{\draw (#1+#3,#2) .. controls (#1+#3,#2-0.05*#3) and (#1+0.96*#3,#2-0.15*#3).. (#1+0.9*#3,#2-0.2*#3) (#1,#2-1*#3) .. controls (#1,#2-0.95*#3) and (#1+0.04*#3,#2-0.85*#3).. (#1+0.1*#3,#2-0.8*#3) (#1+0.1*#3,#2-0.8*#3) ..controls (#1+0.25*#3,#2-0.8*#3) and (#1+0.45*#3,#2-0.69*#3) .. (#1+0.50*#3,#2-0.66*#3) (#1+0.1*#3,#2-0.8*#3) ..controls (#1+0.1*#3,#2-0.65*#3) and (#1+0.22*#3,#2-0.54*#3) .. (#1+0.275*#3,#2-0.505*#3) (#1+0.72*#3,#2-0.50*#3) .. controls (#1+0.75*#3,#2-0.47*#3) and (#1+0.9*#3,#2-0.4*#3).. (#1+0.9*#3,#2-0.2*#3) (#1,#2) .. controls (#1,#2-0.05*#3) and (#1+0.04*#3,#2-0.15*#3).. (#1+0.1*#3,#2-0.2*#3) (#1+0.5*#3,#2-0.34*#3) .. controls (#1+0.6*#3,#2-0.27*#3) and (#1+0.65*#3,#2-0.2*#3).. (#1+0.9*#3,#2-0.2*#3) (#1+#3,#2-#3) .. controls (#1+#3,#2-0.95*#3) and (#1+0.96*#3,#2-0.85*#3).. (#1+0.9*#3,#2-0.8*#3) (#1+#3,#2) .. controls (#1+#3,#2-0.05*#3) and (#1+0.96*#3,#2-0.15*#3).. (#1+0.9*#3,#2-0.2*#3) (#1+0.1*#3,#2-0.2*#3) .. controls (#1+0.3*#3,#2-0.2*#3) and (#1+0.46*#3,#2-0.31*#3) .. (#1+0.5*#3,#2-0.34*#3) (#1+0.1*#3,#2-0.2*#3) .. controls (#1+0.1*#3,#2-0.38*#3) and (#1+0.256*#3,#2-0.49*#3) .. (#1+0.275*#3,#2-0.505*#3) (#1+0.50*#3,#2-0.66*#3) .. controls (#1+0.548*#3,#2-0.686*#3) and (#1+0.70*#3,#2-0.8*#3)..(#1+0.9*#3,#2-0.8*#3) (#1+#3,#2-1*#3) .. controls (#1+#3,#2-0.95*#3) and (#1+0.96*#3,#2-0.85*#3).. (#1+0.9*#3,#2-0.8*#3) (#1+0.72*#3,#2-0.50*#3) .. controls (#1+0.80*#3,#2-0.56*#3) and (#1+0.9*#3,#2-0.73*#3)..(#1+0.9*#3,#2-0.8*#3)(#1+0.72*#3,#2-0.50*#3) -- (#1+0.50*#3,#2-0.66*#3) -- (#1+0.275*#3,#2-0.505*#3) -- (#1+0.5*#3,#2-0.34*#3) -- (#1+0.72*#3,#2-0.50*#3)}
\begin{scope}[yshift=-0.5cm]
\node at (0,0.3){$\scriptstyle V$}; \node at (1,0.3){$\scriptstyle kG$}; \node at (2,0.3){$\scriptstyle kG$}; \twisting(0,-0)[1]; \draw (2,0) -- (2,-1); \draw (0,-1) -- (0,-2); \twisting(1,-1)[1]; \mult(0,-2)[1]; \draw (2,-2) -- (1.5,-3);
\end{scope}
\begin{scope}[xshift=0.1cm,yshift=0cm]
\node at (2.5,-1.8){=};
\end{scope}
\begin{scope}[xshift=0.2cm,yshift=-0.75cm]
\node at (3,0.3){$\scriptstyle V$}; \node at (4,0.3){$\scriptstyle kG$}; \node at (5,0.3){$\scriptstyle kG$};\mult(4,0)[1]; \twisting(3,-1)[1.5]; \draw (3,0) -- (3,-1);
\end{scope}
\begin{scope}[xshift=3.2cm,yshift=0cm]
\node at (2.5,-1.8){,};
\end{scope}
\begin{scope}[xshift=7.7cm, yshift=-0.5cm]
\node at (0,0.3){$\scriptstyle V$}; \node at (1,0.3){$\scriptstyle V$}; \node at (2,0.3){$\scriptstyle kG$}; \doublemap(0,0)[\scriptstyle \mathcal{F}]; \draw (2,0) -- (2,-1); \draw (0,-1) -- (0,-2); \twisting(1,-1)[1]; \mult(0,-2)[1]; \draw (2,-2) -- (1.5,-3);
\end{scope}
\begin{scope}[xshift=7.7cm,yshift=0cm]
\node at (2.5,-1.8){=};
\end{scope}
\begin{scope}[xshift=7.8cm]
\node at (3,0.3){$\scriptstyle V$}; \node at (4,0.3){$\scriptstyle V$}; \node at (5,0.3){$\scriptstyle kG$}; \draw (3,0) -- (3,-1);\twisting(4,0)[1]; \twisting(3,-1)[1]; \draw (5,-1) -- (5,-2); \draw (3,-2) -- (3,-3); \doublemap(4,-2)[\scriptstyle \mathcal{F}];\mult(3,-3)[1]; \draw (5,-3) -- (4.5,-4);
\end{scope}
\begin{scope}[xshift=7.7cm,yshift=0cm]
\node at (7.3,-1.7){and};
\end{scope}
\begin{scope}[xshift=17cm, yshift=-0.5cm]
\node at (0,0.3){$\scriptstyle V$}; \node at (1,0.3){$\scriptstyle V$}; \node at (2,0.3){$\scriptstyle V$};\doublemap(0,0)[\scriptstyle \mathcal{F}]; \draw (2,0) -- (2,-1); \draw (0,-1) -- (0,-2); \doublemap(1,-1)[\scriptstyle \mathcal{F}]; \mult(0,-2)[1]; \draw (2,-2) -- (1.5,-3);
\end{scope}
\begin{scope}[xshift=17.4cm,yshift=0cm]
\node at (2.5,-1.8){=};
\end{scope}
\begin{scope}[xshift=17.5cm,yshift=0cm]
\node at (3,0.3){$\scriptstyle V$}; \node at (4,0.3){$\scriptstyle V$}; \node at (5,0.3){$\scriptstyle V$}; \draw (3,0) -- (3,-1);\doublemap(4,0)[\scriptstyle \mathcal{F}]; \twisting(3,-1)[1]; \draw (5,-1) -- (5,-2); \draw (3,-2) -- (3,-3); \doublemap(4,-2)[\scriptstyle \mathcal{F}];\mult(3,-3)[1]; \draw (5,-3) -- (4.5,-4);
\end{scope}
\begin{scope}[xshift=20.cm,yshift=0cm]
\node at (2.5,-1.8){,};
\end{scope}
\end{tikzpicture}
\label{Brtwistmodcond}
$$
where \begin{tikzpicture}[scale=0.25]
\def\twisting(#1,#2)[#3]{\draw (#1+#3,#2) .. controls (#1+#3,#2-0.05*#3) and (#1+0.96*#3,#2-0.15*#3).. (#1+0.9*#3,#2-0.2*#3) (#1,#2-1*#3) .. controls (#1,#2-0.95*#3) and (#1+0.04*#3,#2-0.85*#3).. (#1+0.1*#3,#2-0.8*#3) (#1+0.1*#3,#2-0.8*#3) ..controls (#1+0.25*#3,#2-0.8*#3) and (#1+0.45*#3,#2-0.69*#3) .. (#1+0.50*#3,#2-0.66*#3) (#1+0.1*#3,#2-0.8*#3) ..controls (#1+0.1*#3,#2-0.65*#3) and (#1+0.22*#3,#2-0.54*#3) .. (#1+0.275*#3,#2-0.505*#3) (#1+0.72*#3,#2-0.50*#3) .. controls (#1+0.75*#3,#2-0.47*#3) and (#1+0.9*#3,#2-0.4*#3).. (#1+0.9*#3,#2-0.2*#3) (#1,#2) .. controls (#1,#2-0.05*#3) and (#1+0.04*#3,#2-0.15*#3).. (#1+0.1*#3,#2-0.2*#3) (#1+0.5*#3,#2-0.34*#3) .. controls (#1+0.6*#3,#2-0.27*#3) and (#1+0.65*#3,#2-0.2*#3).. (#1+0.9*#3,#2-0.2*#3) (#1+#3,#2-#3) .. controls (#1+#3,#2-0.95*#3) and (#1+0.96*#3,#2-0.85*#3).. (#1+0.9*#3,#2-0.8*#3) (#1+#3,#2) .. controls (#1+#3,#2-0.05*#3) and (#1+0.96*#3,#2-0.15*#3).. (#1+0.9*#3,#2-0.2*#3) (#1+0.1*#3,#2-0.2*#3) .. controls (#1+0.3*#3,#2-0.2*#3) and (#1+0.46*#3,#2-0.31*#3) .. (#1+0.5*#3,#2-0.34*#3) (#1+0.1*#3,#2-0.2*#3) .. controls (#1+0.1*#3,#2-0.38*#3) and (#1+0.256*#3,#2-0.49*#3) .. (#1+0.275*#3,#2-0.505*#3) (#1+0.50*#3,#2-0.66*#3) .. controls (#1+0.548*#3,#2-0.686*#3) and (#1+0.70*#3,#2-0.8*#3)..(#1+0.9*#3,#2-0.8*#3) (#1+#3,#2-1*#3) .. controls (#1+#3,#2-0.95*#3) and (#1+0.96*#3,#2-0.85*#3).. (#1+0.9*#3,#2-0.8*#3) (#1+0.72*#3,#2-0.50*#3) .. controls (#1+0.80*#3,#2-0.56*#3) and (#1+0.9*#3,#2-0.73*#3)..(#1+0.9*#3,#2-0.8*#3)(#1+0.72*#3,#2-0.50*#3) -- (#1+0.50*#3,#2-0.66*#3) -- (#1+0.275*#3,#2-0.505*#3) -- (#1+0.5*#3,#2-0.34*#3) -- (#1+0.72*#3,#2-0.50*#3)}
\begin{scope}
\twisting(0,0)[1];
\end{scope}
\end{tikzpicture}
stands for $\Phi$, iff conditions~(1) and~(2) are fulfilled.

\smallskip

\noindent By the very definitions of $\Phi$ and $\mathcal{F}$, the first four conditions always hold. Assume that the other ones hold. Evaluating the fifth one in $x_1\ot x_{n-1}\ot h$ we see that
$$
\sum_{g\in G}\lambda_g g h\ot x_0=\sum_{g\in G} \chi^n(h) \lambda_g h g\ot x_0\quad\text{for all $g,h\in G$},
$$
or equivalently,
$$
\lambda_{hgh^{-1}}=\chi^n(h) \lambda_g\quad\text{for all $g,h\in G$,}
$$
and evaluating the sixth one in $x_1\ot x_{n-1}\ot x_1$ we see that
$$
\chi(g)=1\quad\text{for all $g\in G$ with $\lambda_g\ne 0$.}
$$
Conversely, a direct computation proves that if these facts are true, then the equalities in the last two diagrams hold.
\end{proof}

\begin{corollary}\label{si lambda_g ne 0 par un g central, chi^n = 0} If there is an algebra $B_{\mathcal{E}}$ satisfying the conditions required in Propositions~\ref{estructura de algebra de B_D}, and there exists $g$ in the center $\Z G$ of $G$ with $\lambda_{g} \ne 0$, then $\chi^n=1$.
\end{corollary}

\begin{remark} It is clear that if there exists, then $B_{\mathcal{E}}$ is a $k$-algebra unitary with unit $1_Gx^0$, that $kG$ is a subalgebra of $B_{\mathcal{E}}$ and that $B_{\mathcal{E}}$ is unique up to isomorphism.
\end{remark}

\begin{remark}\label{gen y red} Using that $B_{\mathcal{E}}$ has dimension $n|G|$ it is easy to see that it is canonically isomorphic to the algebra generated by the group $G$ and the element $x$ subject to the relations $x^n = U$ and $xg = \chi(g)gx$ for all $g\in G$.
\end{remark}

Given $q\in k^{\times}$, let
$$
c_q\colon B_{\mathcal{E}}\ot B_{\mathcal{E}} \longrightarrow B_{\mathcal{E}}\ot B_{\mathcal{E}}
$$
be the $k$-linear map defined by $c_q(gx^i\ot hx^j):= q^{ij}\, hx^j \ot gx^i$. It is easy to check that $c_q$ is a braiding operator that is compatible with the unit of $B_{\mathcal{E}}$. Furthermore,

\begin{itemize}

\smallskip

\item[-] a direct computation shows that $c_q$ is compatible with the multiplication map of $B_{\mathcal{E}}$ iff $U=0$ or $q^n=1$,

\smallskip

\item[-] by Remark~\ref{gen y red} there exists an algebra map $\epsilon\colon B_{\mathcal{E}}\to k$ such that $\epsilon(x)=0$ and $\epsilon(g)=1$ for all $g\in G$ iff $\sum_{g\in G} \lambda_g=0$. Moreover, in this case, $c_q$ is compatible with $\epsilon$.

\end{itemize}

\begin{proposition}\label{estructura de bialgebra de B sub D} Let $\mathcal{E}$ be as at the beginning of this section, $z\!\in\! G$ and $q\!\in\! k^{\times}$. Assume that $B_{\mathcal{E}}$ exists. Then, the algebra $B_{\mathcal{E}}$ is a braided bialgebra with braid $c_q$ and comultiplication map $\Delta$ defined by
\begin{equation}
\Delta(x):=1\ot x+x\ot z \qquad\text{and}\qquad\Delta(g) := g\ot g\quad\text{for $g\in G$}\label{eqa1}
\end{equation}
iff

\begin{enumerate}

\smallskip

\item $\binom{n}{j}_{q\chi(z)} = 0$ for all $0<j<n$,

\smallskip

\item $z\in \Z G$ and $U=\lambda(z^n-1_G)$ for some $\lambda\in k$, where $\lambda=0$ if $z^n\ne 1_G$ and $q^n\ne 1$.

\end{enumerate}
\end{proposition}

\begin{proof} Since $B_{\mathcal{E}}$ is generated by the group $G$ and the element $x$ subject to the relations $x^n = U$ and $xg = \chi(g)gx$ for all $g\in G$, there exists an algebra map $\Delta\colon B_{\mathcal{E}} \longrightarrow B_{\mathcal{E}} \ot_{c_q} B_{\mathcal{E}}$ such that~\eqref{eqa1} is satisfied iff the equalities
\begin{align}
& (h\ot h)(g\ot g) = hg\ot hg,\notag\\
&(1\ot x+x\ot z)(g\ot g) = \chi(g) (g\ot g)(1\ot x+x\ot z)\notag
\shortintertext{and}
& (1\ot x+x\ot z)^n = \sum_{l\in G} \lambda_l\, l\ot l\label{latercera}
\end{align}
hold in $B_{\mathcal{E}} \ot_{c_q} B_{\mathcal{E}}$ for all $h,g\!\in\! G$. The first equality is trivial, while the second one is e\-qui\-valent to
\begin{equation*}
\chi(g)(g\ot gx + gx\ot zg) = \chi(g)(g\ot gx+gx\ot gz)\qquad\text{for all $g\in G$,}
\end{equation*}
and so it is fulfilled iff $z$ is in the center of $G$. In order to deal with the last one, we note that, in $B_{\mathcal{E}} \ot_{c_q} B_{\mathcal{E}}$,
$$
(1\ot x)(x\ot z)= q\chi(z)\, x\ot zx = q\chi(z)(x\ot z)(1\ot x),
$$
and so, by formula~\eqref{eq12},
\begin{equation*}
(1\ot x+x\ot z)^n = \sum_{j=0}^n \binom{n}{j}_{q\chi(z)}(x\ot z)^j(1\ot x)^{n-j} = \sum_{j=0}^n \binom{n}{j}_{q\chi(z)} x^j\ot z^jx^{n-j}.
\end{equation*}
Hence, equality~\eqref{latercera} holds iff
$$
\sum_{j=0}^n \binom{n}{j}_{q\chi(z)} x^j\ot z^jx^{n-j} = \sum_{l\in G} \lambda_l\, l\ot l,
$$
which is clearly equivalent to
\begin{align*}
& \binom{n}{j}_{q\chi(z)} = 0\qquad\text{for all $0<j<n$,}
\shortintertext{and}
& \sum_{l\in G} \lambda_l\, l\ot l = 1\ot x^n+x^n\ot z^n = 1\ot U + U\ot z^n = \sum_{l\in G} 1\ot \lambda_l l + \sum_{l\in G} \lambda_l l \ot z^n.
\end{align*}
If $z^n = 1_G$ this happens iff $\lambda_l = 0$ for all $l\in G$, while if $z^n \ne 1_G$, this happens iff $\lambda_l = 0$ for all $l\ne 1_G,z^n$ and if $\lambda_{z^n}=-\lambda_{1_G}$. By the way, this computation shows that if $\Delta$ exists, then the augmentation $\epsilon$ introduced above, is well defined. Moreover, by formula~\eqref{eq12},
\begin{equation}
\Delta(gx^i):= \sum_{j=0}^i \binom{i}{j}_{q\chi(z)} (g\ot g)(x\ot z)^j(1\ot x)^{i-j}= \sum_{j=0}^i \binom{i}{j}_{q\chi(z)} gx^j\ot gz^j x^{i-j}\label{comultiplication}
\end{equation}
for all $g\in G$ and $i\ge 0$. Using this it is easy to see that $c_q$ is compatible with $\Delta$. Since we already know that $c_q$ is compatible with $1_{B_{\mathcal{E}}}$, the multiplication map of $B_{\mathcal{E}}$ and $\epsilon$, in order to finish the proof we only must check that $\Delta$ is coassociative and that $\epsilon$ is its counit. But, since $c_q$ is compatible with $\Delta$ and $\Delta$ is an algebra map, it suffices to verify these facts on $x$ and $g\in G$, which is trivial.
\end{proof}

\begin{remark} Let $\mathcal{E}$ be as at the beginning of this section. If $U = \lambda(z^n-1_G)$ with $z\in \Z G$, $z^n\ne 1_G$ and $\lambda\in k^{\times}$, then the hypothesis of Proposition~\ref{estructura de algebra de B_D} are equivalent to $\chi^n = 1$, while if $U=0$, then the hypothesis of Proposition~\ref{estructura de algebra de B_D} are automatically satisfied.
\end{remark}

\begin{remark} It is easy to see that $\binom{n}{1}_{q\chi(z)} = 0$ implies $(q\chi(z))^n = 1$ and that if $q\chi(z)$ is an $n$-th primitive root of unit, then $\binom{n}{j}_{q\chi(z)} = 0$ for all $0<j<n$.
\end{remark}

\begin{corollary}\label{algebras de Hopf trenzadas de K-R} Each data $\mathcal{D}=(G,\chi,z,\lambda,q)$ consisting of:

\begin{itemize}

\smallskip

\item[-] a finite group $G$,

\smallskip

\item[-] a character $\chi$ of $G$ with values in $k$,

\smallskip

\item[-] a central element $z$ of $G$,

\smallskip

\item[-] elements $q\in k^{\times}$ and $\lambda\in k$,

\smallskip

\end{itemize}
such that

\begin{itemize}

\smallskip

\item[-] $q\chi(z)$ is a root of $1$ of order $n$ greater than $1$,

\smallskip

\item[-] if $\lambda(z^n-1_G)\ne 0$, then $\chi^n=1$,

\end{itemize}
has associated a braided Hopf algebra $H_{\mathcal{D}}$. As an algebra, $H_{\mathcal{D}}$ is generated by the group $G$ and the element $x$ subject to the relations $x^n=\lambda(z^n-1_G)$ and $xg=\chi(g)gx$ for all $g\in G$, the coalgebra structure of $H_{\mathcal{D}}$ is determined by
\begin{align*}
&\Delta(g):=g\ot g\,\,\text{ for $g\in G$,}&& \Delta(x):=1\ot x + x\ot z,\\
&\epsilon(g):=1\,\,\text{ for $g\in G$,} &&\epsilon(x):=0,
\end{align*}
the braid $c_q$ of $H_{\mathcal{D}}$ is defined by
\begin{equation}
c_q(gx^i\ot hx^j):= q^{ij}\,hx^j \ot gx^i,\label{def braid}
\end{equation}
and its antipode is given by
\begin{equation}
S(gx^i) := (-1)^i(q\chi(z))^{\frac{i(i-1)}{2}}x^iz^{-i}g^{-1}.\label{eqa4}
\end{equation}
Furthermore, as a vector space $H_{\mathcal{D}}$ has basis
$$
\{gx^i:g \in G\text{ and } 0 \le i < n\},
$$
and consequently, $\dim \bigl(H_{\mathcal{D}}\bigr) = n|G|$.
\end{corollary}

\begin{proof} Let $\mathcal{E}:=\bigl(G,\chi,\lambda(z^n-1_G),n\bigr)$ and let $B_{\mathcal{E}}$ be the algebra obtained applying Proposition~\ref{estructura de algebra de B_D}. Now note that if $\lambda(z^n-1_G)$, then $\chi^n=1$ and so $q^n = q^n\chi(z)^n = 1$. Hence, we can apply Pro\-po\-si\-tion~\ref{estructura de bialgebra de B sub D}, which implies that $B_{\mathcal{E}}$ has a braided bialgebra structure with comultiplication map, counit and braid as in its statement. Let $H_{\mathcal{D}}$ denote this bialgebra. It remains to check that the map $S$ given by~\eqref{eqa4} is the antipode of $H_{\mathcal{D}}$. Since
$$
S\xcirc \mu(gx^i\ot hx^j) = \mu \xcirc (S\ot S)\xcirc c_q(gx^i\ot hx^j),
$$
for this it suffices to verify that
$$
S(x) + xS(z)=S(1)x+S(x)z=0\qquad\text{and}\qquad S(g)g=gS(g)=1\quad\text{for all $g\in G$,}
$$
which is evident.
\end{proof}

\begin{remark} If $\lambda(z^n-1_G)= 0$, then we can assume without lost of generality (and we do it), that $\lambda = 0$.
\end{remark}

\begin{remark} Assume that $n>1$. The previous corollary also holds if the hypothesis that $q\chi(z)$ is a root of $1$ of order $n$ is replaced by $\binom{n}{j}_{q\chi(z)} = 0$ for all $0<j<n$. However, from now on we will consider that $q\chi(z)$ is a root of $1$ of order $n$.
\end{remark}

\section{Right $\bm{H_{\mathcal{D}}}$-comodule algebras}\label{H_D comodule algebras}
Let $G$ be a group, $V$ be a $k$-vector space and $s\colon k[G]\ot V\to V \ot k[G]$ a $k$-linear map. Evidently, there is a unique family of maps $(\alpha_x^y\colon V\to V)_{x,y\in G}$, such that
$$
s(x\ot v)=\sum_{y\in G} \alpha_x^y(v)\ot y.
$$

\begin{proposition}\label{primer resultado sobre estructuras trenzadas de grupos} The pair $(V,s)$ is a left $k[G]$-space iff $s$ is a bijective map and the fo\-llo\-wing conditions hold:

\begin{enumerate}

\smallskip

\item $(\alpha_x^y)_{y\in G}$ is a complete family of orthogonal idempotents, for all $x\in G$,

\smallskip

\item $\alpha_1^1=\ide$,

\smallskip

\item $\alpha_{xy}^z=\sum_{uw=z}\alpha_x^u\xcirc \alpha_y^w$, for all $x,y,z\in G$.

\end{enumerate}
\end{proposition}

\begin{proof} Mimic the proof of \cite{G-G}*{Proposition 4.10}.
\end{proof}

For $x,y\in G$, let $V_x^y:=\{v\in V:s(x\ot v)=v\ot y\}$.

\begin{proposition}\label{segundo resultado sobre estructuras trenzadas de grupos} The pair $(V,s)$ is a left $k[G]$-space iff:

\begin{enumerate}

\smallskip

\item $\bigoplus_{z\in G} V_x^z = V = \bigoplus_{z\in G} V_z^x$, for all $x\in G$,

\smallskip

\item $V_1^1=V$,

\smallskip

\item $V_{xy}^z=\bigoplus_{uw=z} V_x^u\cap V_y^w$, for all $x,y,z\in G$.

\end{enumerate}
\end{proposition}

\begin{proof} Mimic the proof of \cite{G-G}*{Propositions 4.11 and 4.13}.
\end{proof}

\begin{theorem}\label{tercer resultado sobre estructuras trenzadas de grupos} If $G$ is a finitely generated group, then each left $k[G]$-space $(V,s)$ determines an $\Aut(G)$-gradation
$$
V=\bigoplus_{\zeta\in \Aut(G)} V_{\zeta}
$$
on $V$, by
$$
V_{\zeta}:=\bigcap_{x\in G} V_x^{\zeta(x)}=\{v\in V: s(x\ot v)=v\ot \zeta(x) \text{ for all $x\in G$}\}.
$$
Moreover, the correspondence that each left $k[G]$-space $(V,s)$, with underlying vector space $V$, assigns the $\Aut(G)$-gradation of $V$ obtained as above, is bijective.
\end{theorem}

\begin{proof} Mimic the proof of~\cite{G-G}*{Theorem 4.14}.
\end{proof}

In the sequel $\mathcal{D}:=(G,\chi,z,\lambda,q)$ and $H_{\mathcal{D}}$ are as in Corollary~\ref{algebras de Hopf trenzadas de K-R} and we will freely use the notations and properties established there. Furthermore, to abbreviate expressions we set $p:=\chi(z)$. We now begin with the study of the right $H_{\mathcal{D}}$-braided comodule algebras. We let $\Aut_{\chi,z}(G)$ denote the subgroup of $\Aut(G)$ consisting of all the automorphism $\phi$ such that $\phi(z)=z$ and $\chi\xcirc \phi=\chi$.

\begin{proposition}\label{estructuras trenzadas de HsubD (primer resultado)} If $(p,q)\ne (1,-1)$, then for all left $H_{\mathcal{D}}$-space $(V,s)$ it is true that
\begin{align}
& s(kG\ot V)=V\ot kG,\nonumber\\
& s(z\ot v)=v\ot z\quad\text{for all $v\in V$,}\nonumber
\shortintertext{and there exists $\alpha\in \Aut(V)$ such that}
&s(x\ot v)=\alpha(v)\ot x\quad\text{for all $v\in V$.}\label{eqq1}
\end{align}
\end{proposition}

\begin{proof} Write
$$
s(gx^i\ot v) = \sum_{\substack{h\in G\\ 0\le j<n}} \beta^{g,i}_{h,j}(v) \ot hx^j.
$$
Since $S^2(gx^i)=q^{i(i-1)}p^{-i}gx^i$, we have
\begin{align*}
q^{i(i-1)}p^{-i} \sum_{\substack{h\in G\\ 0\le j<n}} \beta^{g,i}_{h,j}(v) \ot hx^j &=q^{i(i-1)}p^{-i} s(gx^i\ot v)\\
&= s\xcirc (S^2\ot V)(gx^i\ot v)\\
&= (V\ot S^2)\xcirc s(gx^i\ot v)\\
&= \sum_{\substack{h\in G\\ 0\le j<n}} \beta^{g,i}_{h,j}(v) \ot S^2(hx^j)\\
&= \sum_{\substack{h\in G\\ 0\le j<n}} q^{j(j-1)}p^{-j}\beta^{g,i}_{h,j}(v) \ot hx^j,
\end{align*}
and consequently,
\begin{equation}\label{equu1}
\beta^{g,i}_{h,j}\ne 0 \Longrightarrow q^{j(j-1)-i(i-1)} = p^{j-i}.
\end{equation}
Using now that $s$ is compatible with $\Delta$, we obtain that
\begin{equation}
\begin{aligned}
\sum_{\substack{h\in G\\ 0\le i<n}}\sum_{j=0}^i\binom{i}{j}_{qp}\beta^{g,0}_{h,i}(v)\ot hx^j\ot hz^jx^{i-j} &= (V\ot \Delta)\xcirc s(g\ot v)\\
&= (s\ot H_{\mathcal{D}}) \xcirc (H_{\mathcal{D}}\ot s)\xcirc (\Delta\ot V)(g\ot v)\\
&= \adjustlimits\sum_{\substack{h\in G\\ 0\le i<n}} \sum_{\substack{h'\in G\\ 0\le i'<n}} \beta^{g,0}_{h,i} \xcirc \beta^{g,0}_{h',i'}(v) \ot hx^i\ot h'x^{i'}.
\end{aligned}\label{eqq2}
\end{equation}
Hence,
\begin{equation}
\beta_{h,i}^{g,0}\xcirc \beta_{h',i'}^{g,0} = \begin{cases} \binom{i+i'}{i}_{qp} \beta_{h,i+i'}^{g,0} &\text{if $h' = hz^i$ and $i+i' < n$,}\\ 0 &\text{otherwise.} \end{cases}\label{ee1}
\end{equation}
Combining this with~\eqref{equu1} we obtain that
$$
\beta_{h,i}^{g,0}\ne 0\Longrightarrow \beta_{h,j}^{g,0}\ne 0\text{ for all $j\le i$ } \Longrightarrow q^{j(j-1)} = p^j \text{ for $j\le i$.}
$$
Consequently, if $\beta_{h,i}^{g,0}\ne 0$ for some $g\in G$ and $i\ge 1$, then $p = q^0 = 1$. Hence if $p\ne 1$, then $\beta_{h,i}^{g,0}=0$ for all $g\in G$ and $i\ge 1$. Assume that $p = 1$. If $\beta_{h,i}^{g,0}\ne 0$ for some $g\in G$ and $i\ge 2$, then $q^2 = p^2 = 1$. But this is impossible, since it implies that $n:=\ord(qp)=\ord(q)\le 2$, which contradicts that $i<n$. Therefore
\begin{equation}
s(g\ot v) = \begin{cases}\sum_{h\in G}\beta^{g,0}_{h,0}(v)\ot h &\text{if $p\ne 1$,}\\ \sum_{h\in G}\beta^{g,0}_{h,0}(v)\ot h + \sum_{h\in G}\beta^{g,0}_{h,1}(v)\ot hx &\text{if $p = 1$.}\end{cases}\label{ee2}
\end{equation}
On the other hand, due to $s$ is compatible with the counit of $H_{\mathcal{D}}$, we get
$$
\sum_{h\in G} \beta_{h,0}^{g,0}=\ide\quad\text{for all $g\in G$,}
$$
which, combined with the particular case of~\eqref{ee1} obtained by taken $i=i'=0$, shows that
\begin{equation}
\bigl(\beta_{h,0}^{g,0}\bigr)_{h\in G}\,\text{ is a complete family of orthogonal idempotents for all $g\in G$.}\label{ee3}
\end{equation}
Equality~\eqref{ee2} shows that if $p\ne 1$, then $s(kG\ot V) \subseteq V\ot kG$. Assume now that $p =1$ and $q\ne -1$ (which implies $n>2$). Using that $s$ is compatible with the multiplication map of $H_{\mathcal{D}}$ we get that
\begin{equation}
\begin{aligned}
v\ot 1&= s(g^{-1}g\ot v)\\
&=(v\ot\mu)\xcirc (s\ot H_{\mathcal{D}})\xcirc (H_{\mathcal{D}}\ot s)(g^{-1}\ot g\ot v)\\
& = \sum_{h\in G}\sum_{l\in G}\beta^{g^{-1},0}_{h,0}\xcirc\beta_{1,0}^{g,0}(v)\ot hl + \sum_{h\in G}\sum_{l\in G}\chi(l)\beta_{h,1}^{g^{-1},0}\xcirc\beta_{l,0}^{g,0}(v)\ot hlx \\
& + \sum_{h\in G}\sum_{l\in G}\beta_{h,0}^{g^{-1},0}\xcirc\beta_{l,1}^{g,0}(v)\ot hlx + \sum_{h\in G}\sum_{l\in G}\chi(l)\beta_{h,1}^{g^{-1},0}\xcirc\beta_{l,1}^{g,0}(v)\ot hl x^2.
\end{aligned}\label{eqq3}
\end{equation}
Consequently,
$$
\sum_{h\in G}\beta_{h^{-1},0}^{g^{-1},0}\xcirc\beta_{g,0}^{h,0} = \ide_V,
$$
which by~\eqref{ee3} implies that
$$
\beta_{g^{-1},0}^{h^{-1},0}(v)=v\quad\text{for all $v\in \ima\bigl(\beta_{h,0}^{g,0} \bigr)$ and $h,g\in G$.}
$$
Since $\bigl(\beta_{g,0}^{h,0}\bigr)_{h\in G}$ and $\bigl(\beta^{g^{-1},0}_{h,0} \bigr)_{h\in G}$ are complete families of orthogonal idempotents, from this it follows that
$$
\beta^{g^{-1},0})_{h^{-1},0}=\beta^{g,0}_{h,0}\quad\text{for all $h,g\in G$.}
$$
Combining this with~\eqref{eqq3}, we conclude that
\begin{equation}
\begin{aligned}
0 &=\sum_{h\in G}\sum_{l\in G}\chi(l)\beta_{h,1}^{g^{-1},0}\xcirc\beta^{g,0}_{l,0}(v)\ot hlx + \sum_{h\in G}\sum_{l\in G}\beta^{g^{-1},0}_{h,0}\xcirc\beta_{l,1}^{g,0}(v)\ot hlx\\
&= \sum_{h\in G}\sum_{l\in G}\chi(l)\beta_{h,1}^{g^{-1},0}\xcirc\beta^{g^{-1},0}_{ l^{-1},0}(v)\ot hlx + \sum_{h\in G}\sum_{l\in G}\beta^{g,0}_{h^{-1},0}\xcirc\beta_{l,1}^{g,0}(v) \ot hlx\\
& = \sum_{h\in G}\chi(z^{-1}h^{-1})\beta_{h,1}^{g^{-1},0}(v)\ot z^{-1}x + \sum_{h\in G} \beta_{h,1}^{g,0}(v)\ot x,
\end{aligned}\label{ecua3}
\end{equation}
where the last equality follows from the fact that by~\eqref{ee1}
\begin{equation}
\beta_{h,1}^{g,0}\xcirc \beta_{h',0}^{g,0} = \begin{cases}\beta_{h,1}^{g,0} &\text{if $h'=hz$,}\\ 0 &\text{otherwise,}\end{cases}\quad\text{and} \quad\beta_{h,0}^{g,0} \xcirc\beta_{h',1}^{g,0} = \begin{cases}\beta_{h,1}^{g,0} &\text{if $h' = h$,}\\ 0 & \text{otherwise.}\end{cases}
\label{ee4}
\end{equation}
Note that by~\eqref{ee3} and the second equality in~\eqref{ee4}, the images of the maps $\beta_{h,1}^{g,0}$ are in direct sum, for each $g\in G$. Hence, from~\eqref{ecua3} it follows that if $z\ne 1$, then $\beta_{h,1}^{g,0}=0$ for all $g,h\in G$, which by equality~\eqref{ee2} implies that $s(kG\ot V) \subseteq V\ot kG$. We now assume additionally that $z=1$. Then $\Prim(H_{\mathcal{D}})= kx$ and so, by~\cite{G-G}*{Proposition 4.4} there exists an automorphism $\alpha$ of $V$ such that $s(x\ot v)=\alpha(v)\ot x$ for all $v\in V$. Furthermore, by the compatibility of $s$ with $c_q$,
\begin{align*}
\sum_{h\in G} \beta_{h,0}^{g,0}\xcirc \alpha(v)\ot x\ot h & + q \sum_{h\in G} \beta_{h,1}^{g,0}\xcirc \alpha(v)\ot x\ot hx \\
&= (V\ot c_q)\xcirc (s\ot H_{\mathcal{D}})\xcirc (H_{\mathcal{D}}\ot s)(g\ot x\ot v)\\
&= (s\ot H_{\mathcal{D}})\xcirc (H_{\mathcal{D}}\ot s)\xcirc (c_q \ot V)(g\ot x\ot v)\\
&= \sum_{h\in G}\alpha \xcirc \beta_{h,0}^{g,0}(v)\ot x\ot h  + \sum_{h\in G} \alpha\xcirc \beta_{h,1}^{g,0} (v)\ot x\ot hx
\end{align*}
for all $g\in G$, and therefore
\begin{equation}
\alpha\xcirc \beta_{h,0}^{g,0}= \beta_{h,0}^{g,0}\xcirc \alpha\quad\text{and}\quad\alpha \xcirc \beta_{h,1}^{g,0}= q\beta_{h,1}^{g,0}\xcirc \alpha\quad \text{for all $h,g\in G$.} \label{eqq7}
\end{equation}
Using now that $s$ is compatible with the multiplication map of $H_\mathcal{D}$, we obtain that
\begin{align*}
\chi(g)\sum_{h\in G}\beta_{h,0}^{g,0}\xcirc\alpha(v)\ot hx & + \chi(g)\sum_{h\in G} \beta_{h,1}^{g,0}\xcirc \alpha(v)\ot hx^2 \\
&=(V\ot\mu)\xcirc(s\ot H_{\mathcal{D}})\xcirc(H_{\mathcal{D}}\ot s)(\chi(g)g\ot x\ot v)\\
&=(s\ot H_{\mathcal{D}})\xcirc(H_{\mathcal{D}}\ot s)\xcirc(\mu\ot V)(\chi(g)g\ot x\ot v)\\
&=(s\ot H_{\mathcal{D}})\xcirc (H_{\mathcal{D}}\ot s)\xcirc (\mu \ot V)(x\ot g\ot v)\\
&=(V\ot \mu)\xcirc (s\ot H_{\mathcal{D}})\xcirc (H_{\mathcal{D}}\ot s)(x\ot g\ot v)\\
&=\sum_{h\in G}\chi(h)\,\alpha\xcirc\beta_{h,0}^{g,0}(v)\ot hx + \sum_{h\in G}\chi(h)\, \alpha\xcirc \beta_{h,1}^{g,0}(v)\ot hx^2
\end{align*}
for all $g\in G$, which combined with~\eqref{eqq7} gives
\begin{align}
&\chi(g)\,\beta_{h,0}^{g,0}\xcirc\alpha = \chi(h)\,\alpha\xcirc\beta_{h,0}^{g,0} = \chi(h)\,\beta_{h,0}^{g,0}\xcirc\alpha\label{eqq8}
\shortintertext{and}
& \chi(g)\,\beta_{h,1}^{g,0}\xcirc \alpha = \chi(h)\,\alpha\xcirc\beta_{h,1}^{g,0} = \chi(h)q\,\beta_{h,1}^{g,0}\xcirc\alpha\label{eqq9}
\end{align}
for all $g,h\in G$. Since $\alpha$ is bijective, from~\eqref{eqq8} it follows that if $\beta_{h,0}^{g,0}\ne 0$, then $\chi(g)=\chi(h)$. Combining this with~\eqref{ee4}, we see that $\beta_{h,1}^{g,0}\ne 0 \Rightarrow \beta_{h,0}^{g,0}\ne 0\Rightarrow \chi(g)=\chi(h)$. Therefore, from~\eqref{eqq9} it follows that if there exist $g,h\in G$ such that $\beta_{h,1}^{g,0}\ne 0$, then $q=1$, which is false. So, $\beta_{h,1}^{g,0}= 0$ for all $g,h\in G$. This concludes the proof that $s(kG\ot V)\subseteq V\ot kG$. But then a similar computation with $s$ replaced by $s^{-1}$ shows that $s(kG\ot V)\supseteq V \ot kG$, and so the equality holds.

\smallskip

We now return to the general case and we claim that
\begin{enumerate}

\smallskip

\item $\beta_{h,j}^{1,1}= 0$ for all $h\in G$ and $j\ge 2$,

\smallskip

\item $\beta_{h,1}^{1,1}= 0$ for $h\ne 1_G$,

\smallskip

\item $\beta_{1,1}^{1,1}$ is bijective,

\smallskip

\item $\beta_{z,0}^{z,0} = \ide$ and $\beta_{h,0}^{z,0} = 0$ for $h\ne z$,

\smallskip

\item $\beta_{h,0}^{1,1} = 0$ for $h\notin\{1_G,z\}$,

\smallskip

\item If $z\ne 1$, then $\beta_{z,0}^{1,1} = -\beta_{1,0}^{1,1}$ while if $z = 1_G$, then $\beta_{1,0}^{1,1} = 0$,

\smallskip

\end{enumerate}
In fact, $s(kG\ot V) \subseteq V\ot kG$ means that $\beta^{g,0}_{h,j}=0$ for all $g,h\in G$ and $j>0$. Hence, by the compatibility of $s$ with $\Delta$,
\begin{equation}
\begin{aligned}
\sum_{\substack{h\in G\\ 0\le i<n}}\sum_{j=0}^i\binom{i}{j}_{qp}\beta^{1,1}_{h,i}(v)\ot hx^j \ot hz^jx^{i-j} & = (V\ot \Delta)\xcirc s (x\ot v)\\
& = (s\ot H_{\mathcal{D}})\xcirc (H_{\mathcal{D}}\ot s)\xcirc (\Delta\ot V)(x\ot v)\\
& = \sum_{\substack{h\in G\\ 0\le i<n}} \beta^{1,1}_{h,i}(v)\ot 1\ot hx^i\\
& + \sum_{\substack{h,l\in G\\ 0\le i<n}}\beta^{1,1}_{l,i}\xcirc \beta^{z,0}_{h,0}(v)\ot lx^i\ot h.
\end{aligned}\label{eqq5}
\end{equation}
This implies that items~(1) and (2) are true and that

\begin{itemize}

\smallskip

\item[(8)] $\beta_{1,1}^{1,1}=\beta_{1,1}^{1,1}\xcirc \beta_{z,0}^{z,0}$ and $\beta_{1,1}^{1,1}\xcirc \beta_{h,0}^{z,0} = 0$ for all $h\in G\setminus \{z\}$,

\smallskip

\item[(9)] $\beta_{h,0}^{1,1}=\beta_{h,0}^{1,1}\xcirc \beta_{h,0}^{z,0}$ and $\beta_{h,0}^{1,1}=-\beta_{1,0}^{1,1}\xcirc \beta_{h,0}^{z,0}$ for all $h \in G\setminus\{1_G\}$,

\smallskip

\item[(10)] $\beta_{h,0}^{1,1}\xcirc \beta_{1,0}^{z,0} = 0$ for all $h\in G$.

\smallskip

\end{itemize}
By items~(1) and~(2) and condition~\eqref{equu1},
\begin{equation}
s(x\ot v) = \begin{cases} \beta_{1,1}^{1,1}(v)\ot x &\text{if $p \ne 1$,}\\ \beta_{1,1}^{1,1}(v)\ot x + \sum_{h\in G} \beta_{h,0}^{1,1}(v)\ot h &\text{if $p = 1$.} \end{cases}\label{ecua4}
\end{equation}
This immediately implies that $\beta_{1,1}^{1,1}$ is injective. In fact, if $\beta_{1,1}^{1,1}(v) = 0$, then $s(x\ot v)\in s(kG\ot V)$, and so $v=0$ since $s\colon H_{\mathcal{D}}\ot V\to V\ot H_{\mathcal{D}}$ is injective. Item~(4) follows from item~(8) and the injectivity of $\beta_{1,1}^{1,1}$. Hence, by item~(9), we have $\beta_{h,0}^{1,1}=\beta_{h,0}^{1,1}\xcirc \beta_{h,0}^{z,0} = 0$ for all $h\in G\setminus \{1_G,z\}$, proving item~(5). Note also that by items~(4), (9) and (10),
$$
\beta_{z,0}^{1,1}=\begin{cases}
                   -\beta_{1,0}^{1,1}\xcirc \beta_{z,0}^{z,0} = -\beta_{1,0}^{1,1} &\text{if $z\ne 1_G$,}\\
                  \beta_{1,0}^{1,1}\xcirc \beta_{1,0}^{1,0} = 0  &\text{if $z = 1_G$,}
                  \end{cases}
$$
which proves item~(6). Combining item~(4) with the fact that $\beta^{z,0}_{h,j}=0$ for all $h\in G$ and $j>0$, we deduce that
$$
s(z\ot v)=v\ot z\qquad\text{for all $v\in V$.}
$$
Furthermore, by items~(5) and~(6), equality~\eqref{ecua4} becomes
\begin{equation}\label{ecua5}
s(x\ot v) = \begin{cases} \beta_{1,1}^{1,1}(v)\ot x  &\text{if $p\ne 1$ or $z = 1_G$,}\\ \beta_{1,1}^{1,1}(v)\ot x + \beta_{1,0}^{1,1}(v)\ot 1_{H_{\mathcal{D}}} - \beta_{1,0}^{1,1}(v)\ot z &\text{otherwise.}\end{cases}
\end{equation}
Next we prove that if $p=1$ and $q\ne -1$, then $\beta_{1,0}^{1,1} = 0$. If $z=1_G$ this was checked above. So we can assume that $z\ne 1_G$. To abbreviate expressions we set $\alpha:=\beta_{1,1}^{1,1}$ and $\beta:=\beta_{1,0}^{1,1}$. Evaluating
$$
(s\ot H_{\mathcal{D}})\xcirc (H_{\mathcal{D}}\ot s)\xcirc (c_q\ot V) \qquad\text{and}\qquad (V\ot c_q)\xcirc(s\ot H_{\mathcal{D}})\xcirc (H_{\mathcal{D}}\ot s)
$$
in $x\ot x\ot v$ for all $v\in V$, and using~\eqref{ecua5} and that these maps coincide, we see that
\begin{equation*}
q \beta\xcirc \alpha = \alpha\xcirc \beta\quad\text{and}\quad  q\alpha\xcirc \beta = \beta \xcirc \alpha.
\end{equation*}
Then $q^2\alpha\circ \beta = \alpha\circ \beta$, and so $\beta = 0$, since $q^2\ne 1$ and $\alpha$ is injective. Hence~\eqref{ecua5} becomes
$$
s(x\ot v) = \alpha (v)\ot v \quad\text{for all $v\in V$.}
$$
Consequently $s(x\ot V)\subseteq V\ot x$ and a similar computation with $s$ replaced by $s^{-1}$ shows that $s(x\ot V)\supseteq V\ot x$, which immediately proves that $\alpha$ is a surjective map.
\end{proof}

In the rest of the paper we assume that $(p,q)\ne (1,-1)$.

\begin{proposition}\label{estructuras trenzadas de HsubD} Let $V$ be a $k$-vector space endowed with an $\Aut_{\chi,z}(G)$-gradation
$$
V = \bigoplus_{\zeta\in \Aut_{\chi,z}(G)} V_{\zeta}
$$
and an automorphism $\alpha\colon V\to V$ fulfilling

\begin{itemize}

\smallskip

\item[-] $\alpha(V_{\zeta})=V_{\zeta}$ for all $\zeta\in\Aut_{\chi,z}(G)$,

\smallskip

\item[-] $\alpha^n=\ide$ if $\lambda(z^n-1_G)\ne 0$.

\smallskip

\end{itemize}
Then the pair $(V,s)$, where $s\colon H_{\mathcal{D}}\ot V\longrightarrow V\ot H_{\mathcal{D}}$ is the map defined by
\begin{equation}
s(gx^i\ot v):=\alpha^i(v)\ot \zeta(g)x^i\qquad \text{for all $v\in V_{\zeta}$,}\label{eqrankone1}
\end{equation}
is a left $H_{\mathcal{D}}$-space. Furthermore, all the left $H_{\mathcal{D}}$-spaces with underlying $k$-vector space $V$ have this form.
\end{proposition}

\begin{proof} It is easy to check that the map $s$ defined by~\eqref{eqrankone1} is compatible with the unit, the counit, the multiplication map and the braid of $H_{\mathcal{D}}$. So, by Remark~\ref{basta verificar sobre generadores}, in order to check that $s$ is a left transposition it suffices to verify that
$$
(s\ot H_{\mathcal{D}})\xcirc(H_{\mathcal{D}}\ot s)\xcirc(\Delta\ot V)(x\ot v)= (V\ot \Delta)\xcirc s(x\ot v)
$$
and
$$
(s\ot H_{\mathcal{D}})\xcirc(H_{\mathcal{D}}\ot s)\xcirc(\Delta\ot V)(g\ot v)= (V\ot \Delta)\xcirc s(g\ot v)\quad\text{for $g\in G$,}
$$
which is clear.

\smallskip

Conversely, assume that $(V,s)$ is a left $H_{\mathcal{D}}$-space. By Proposition~\ref{estructuras trenzadas de HsubD (primer resultado)} and Theorem~\ref{tercer resultado sobre estructuras trenzadas de grupos}, we know that there exist an automorphism $\alpha$ of $V$ and a gradation
$$
V=\bigoplus_{\zeta\in \Aut(G)} V_{\zeta}
$$
of $V$, such that $s(g\ot v)= v\ot \zeta(g)$ and $s(x\ot v)=\alpha(v)\ot x$ for all $g\in G$ and $v\in V_{\zeta}$. Again by Proposition~\ref{estructuras trenzadas de HsubD (primer resultado)}, we also know that $s(z\ot v)=v\ot z$ for all $v\in V$. Therefore,  if $V_{\zeta}\ne 0$, then $\zeta(z)=z$. Now, let $g\in G$ and $v\in V_{\zeta}\setminus \{0\}$. A direct computation shows that
\begin{align*}
\alpha(v) \ot x\zeta(g)& =s(xg\ot v)\\
&=s\bigl(\chi(g)gx \ot v\bigr)\\
&=\sum_{\phi\in \Aut(G)} \alpha(v)_{\phi}\ot \chi(g)\phi(g)x\\
&= \sum_{\phi\in \Aut(G)} \alpha(v)_{\phi}\ot \chi(g)\chi(\phi(g))^{-1}x\phi(g).
\end{align*}
Since $g$ is arbitrary, from this it follows that $\alpha(v)_{\phi}=0$ for $\phi\ne \zeta$ and that $\chi(\zeta(g)) = \chi(g)$. So
$$
\alpha(V_{\zeta})=V_{\zeta}\qquad\text{and}\qquad \chi\xcirc \zeta=\chi.
$$
Lastly, suppose that $\lambda(z^n-1_G)\ne 0$. Then
$$
v\ot \lambda(z^n-1_G)=s\bigl(\lambda(z^n-1_G)\ot v\bigr)=s (x^n\ot v)= \alpha^n(v)\ot x^n= \alpha^n(v)\ot \lambda(z^n-1_G),
$$
for each $v\in V$. This shows that $\alpha^n=\ide$ and finishes the proof.
\end{proof}

Our next aim is to characterize the right $H_{\mathcal{D}}$-braided comodule structures. Let $(V,s)$ be a left $H_{\mathcal{D}}$-space and let
$$
V=\bigoplus_{\zeta\in \Aut_{\chi,z}(G)} V_{\zeta}\qquad\text{and}\qquad \alpha\colon V\longrightarrow V
$$
be the decomposition and the automorphism associated with the left transposition $s$. Each map
$$
\nu\colon V\longrightarrow V\ot H_{\mathcal{D}}
$$
determines and it is determined by a family of maps
\begin{equation}
\bigl(U^g_i\colon V\longrightarrow V\bigr)_{g\in G,\, 0\le i<n}\label{defUgm}
\end{equation}
via
\begin{equation}
\nu(v):= \sum_{\cramped{\substack{g\in G\\ 0\le i<n}}} U^g_i(v)\ot gx^i.\label{defnu}
\end{equation}

\begin{proposition}\label{caracterizacion de Hsub D comodules} The pair $(V,s)$ is a right $H_{\mathcal{D}}$-comodule via $\nu$ iff

\begin{enumerate}

\smallskip

\item $U^g_i(V_{\zeta})\subseteq V_{\zeta}$ for all $g\in G$, $\zeta \in \Aut_{\chi,z}(G)$ and $i\in \{0,1\}$,

\smallskip

\item $(U^g_0)_{g\in G}$ is a complete family of orthogonal idempotents,

\smallskip

\item $U^g_1=U^g_0\xcirc U^g_1=U^g_1\xcirc U^{gz}_0$ for all $g\in G$,

\smallskip

\item $U^g_i=\frac{1}{(i)!_{qp}}\, U^g_1\xcirc U^{gz}_1\xcirc\cdots\xcirc U^{gz^{i-1}}_1$ for all $g\in G$ and $1\le i<n$,

\smallskip

\item $U^g_1\xcirc U^{gz}_1\xcirc\cdots\xcirc U^{gz^{n-1}}_1 =0$ for all $g\in G$,

\smallskip

\item $\alpha\xcirc U^g_0 = U^g_0\xcirc \alpha$ and $q\,\alpha\xcirc U^g_1 = U^g_1\xcirc \alpha$ for all $g\in G$.

\end{enumerate}
\end{proposition}

\begin{proof} For each $v\in V_{\zeta}$, $h\in G$ and $0\le j<n$, write
$$
U^h_j(v)=\sum_{\phi\in \Aut_{\chi,z}(G)} U^h_j(v)_{\phi}\qquad\text{with $U^h_j(v)_{\phi}\in V_{\phi}$}.
$$
Since
$$
(\nu\ot H_{\mathcal{D}})\xcirc s(gx^i\ot v) = \sum_{\cramped{\substack{h\in G\\ 0\le j<n}}} U^h_j\bigl(\alpha^i(v)\bigr)\ot hx^j\ot \zeta(g)x^i
$$
and
$$
(V\ot c_q)\xcirc (s\ot H_{\mathcal{D}})\xcirc (H_{\mathcal{D}}\ot \nu)(gx^i\ot v)= \sum_{\cramped{\substack{h\in G\\ 0\le j<n}}}\sum_{\phi\in \Aut_{\chi,z}(G)}\! q^{ij}\,\alpha^i\bigl(U^h_j(v)_{\phi}\bigr)\ot hx^j \ot \phi(g)x^i
$$
the map $\nu$ satisfies condition~\eqref{eq6} in Remark~\ref{re: H-braided comodule} iff
$$
\sum_{\cramped{\substack{h\in G\\ 0\le j<n}}} U^h_j\bigl(\alpha^i(v)\bigr)\ot hx^j\ot \zeta(g)x^i = \sum_{\cramped{\substack{h\in G\\ 0\le j<n}}}\sum_{\phi\in \Aut_{\chi,z}(G)}\! q^{ij}\,\alpha^i\bigl(U^h_j(v)_{\phi}\bigr)\ot hx^j \ot \phi(g)x^i,
$$
for all $\zeta \in \Aut_{\chi,z}(G)$, $v\in V_{\zeta}$, $g\in G$ and $0\le i<n$. Since $\zeta$, $v$ and $g$ are arbitrary, $\alpha(V_{\phi})=V_{\phi}$ for all $\phi\in\Aut_{\chi,z}(G)$, and $\alpha$ is bijective, this happens iff
\begin{equation}
U^h_j(V_{\zeta})\subseteq V_{\zeta}\quad\text{and}\quad q^j\,\alpha\xcirc U^h_j= U^h_j\xcirc \alpha,\label{eqrankone4}
\end{equation}
for all $h$, $j$, and $\zeta$. On the other hand, since $\epsilon(gx^i)=\delta_{0i}$, the map $\nu$ is counitary iff
\begin{equation}
\sum_{g\in G} U^g_0=\ide,
\label{eqrankone5}
\end{equation}
and since
\begin{align*}
&(V\ot \Delta)\xcirc \nu(v)= \sum_{\cramped{\substack{g\in G\\ 0\le i<n}}} U^g_i(v) \ot \Delta(gx^i) =\sum_{\cramped{\substack{g\in G\\ 0\le i<n}}} \sum_{j=0}^i \binom{i}{j}_{qp} U^g_i(v) \ot gx^j \ot gz^jx^{i-j}
\shortintertext{and}
& (\nu\ot H_{\mathcal{D}})\xcirc \nu(v) = \sum_{\cramped{\substack{h\in G\\ 0\le l<n}}} \nu\bigl(U^h_l(v)\bigr)\ot hx^l = \sum_{\cramped{\substack{h\in G\\ 0\le l<n}}} \sum_{\cramped{\substack{g\in G\\ 0\le j<n}}} U^g_j\bigl(U^h_l(v)\bigr)\ot gx^j \ot hx^l,
\end{align*}
it is coassociative iff
\begin{equation}\label{eqrankone6}
U^g_j\xcirc U^h_l=\begin{cases}
               \binom{j+l}{j}_{qp} U^g_{j+l} &\text{if $h=gz^j$ and $j+l<n$,}\\
               0 &\text{otherwise.}
         \end{cases}
\end{equation}
Thus, in order to prove this proposition we must show items~(1)--(6) are equivalent to conditions~\eqref{eqrankone4}, \eqref{eqrankone5} and~\eqref{eqrankone6}. It is evident that~\eqref{eqrankone4} implies items~(1) and~(6), while items~(2) and~(3) follow from~\eqref{eqrankone5} and~\eqref{eqrankone6}. Finally, using~\eqref{eqrankone6} again it is easy to prove by induction on $j$ that items~(4) and~(5) are also satisfied. Conversely, assume that the maps $U^g_i$ satisfy items~(1)--(6). It is clear that item~(2) implies condition~\eqref{eqrankone5}, and equality~\eqref{eqrankone4} follows from items~(1), (4) and~(6). It remains to prove equality~\eqref{eqrankone6}. We claim that
\begin{equation}
U^f_i\xcirc U^{fz^i}_j = \begin{cases}
                   \binom{i+j}{i}_{qp} U^f_{i+j} &\text{if $i+j<n$,}\\
                   0 &\text{if $i+j\ge n$.}
            \end{cases}\label{eqrankone7}
\end{equation}
By item~(2) this is true if $j=i=0$. In order to check it when $j>0$ and $i=0$ or $j=0$ and $i>0$, it suffices to note that by items~(3) and~(4),
$$
U^f_0\xcirc U^f_i= \frac{1}{(i)!_{qp}}\, U^f_0 \xcirc U^f_1\xcirc\cdots\xcirc U^{fz^{i-1}}_1= \frac{1}{(i)!_{qp}}\, U^f_1\xcirc\cdots\xcirc U^{fz^{i-1}}_1= U^f_i
$$
and
$$
U^f_i\xcirc U^{fz^i}_0= \frac{1}{(i)!_{qp}}\,U^f_1\xcirc\cdots\xcirc U^{fz^{i-1}}_1\xcirc U^{fz^i}_0=\frac{1}{(i)!_{qp}}\,U^f_1\xcirc\cdots\xcirc U^{fz^{i-1}}_1= U^f_i,
$$
respectively. Assume now that $j>0$ and $i>0$. Then, by item~(4),
$$
U^f_i\xcirc U^{fz^i}_j =\frac{1}{(i)!_{qp}}\frac{1}{(j)!_{qp}}\, U^f_1\xcirc\cdots\xcirc U^{fz^{i-1}}_1\xcirc U^{fz^i}_1\xcirc\cdots\xcirc U^{fz^{i+j-1}}_1,
$$
and the claim follows immediately from items~(4) and~(5). Note now that~\eqref{eqrankone7} implies that
$$
U^f_i\xcirc U^h_j=U^f_i\xcirc U^{fz^i}_0\xcirc U^h_0\xcirc U^h_j
$$
which, combined with item~(2), shows that
$$
U^f_i\xcirc U^h_j=0 \quad\text{if $h\ne fz^i$,}
$$
finishing the proof of~\eqref{eqrankone6}.
\end{proof}

\begin{corollary}\label{segunda caracterizacion de HsubD comodules} Let $V$ be a $k$-vector space. Each data consisting of

\begin{itemize}

\smallskip

\item[-] a $G\times \Aut_{\chi,z}(G)$-gradation $\displaystyle{V=\bigoplus_{(g,\zeta)\in G\times \Aut_{\chi,z}(G)} V_{g,\zeta}}$ of $V$,

\smallskip

\item[-] an automorphism $\alpha\colon V\to V$ of $V$ such that
$$
\qquad \alpha^n=\ide\text{ if }\lambda(z^n-1_G)\ne 0\quad\text{and}\quad \alpha(V_{g,\zeta})=V_{g,\zeta}\text{ for all $(g,\zeta)\in G\times \Aut_{\chi,z}(G)$,}
$$

\smallskip

\item[-] a map $U\colon V \to V$, such that
$$
\qquad U\xcirc \alpha = q\,\alpha\xcirc U,\quad U^n = 0\quad\text{and}\quad U(V_{g,\zeta})\subseteq V_{gz^{-1},\zeta} \text{ for all $(g,\zeta)\in G\times \Aut_{\chi,z}(G)$},
$$

\smallskip

\end{itemize}
determines univocally a right $H_{\mathcal{D}}$-comodule $(V,s)$, in which

\begin{itemize}

\smallskip

\item[-] $s\colon H_{\mathcal{D}}\ot V \longrightarrow V\ot H_{\mathcal{D}}$ is the left transposition of $H_{\mathcal{D}}$ on $V$ associated as in~\eqref{eqrankone1} with the map $\alpha$ and the $\Aut_{\chi,z}(G)$-gradation of $V$
$$
V=\bigoplus_{\zeta\in \Aut_{\chi,z}(G)}V_{\zeta},\qquad\text{where } V_{\zeta}:=\bigoplus_{g\in G} V_{g,\zeta},
$$

\smallskip

\item[-]  the coaction $\nu\colon V\to V\ot H_{\mathcal{D}}$ of $(V,s)$ is defined by
$$
\nu(v):=\sum_{j=0}^{n-1} \frac{1}{(j)!_{qp}} U^j(v)\ot z^{-j}gx^j\qquad\text{for all $v\in V_g$,}
$$
where, for all $g\in G$,
$$
V_g:=\bigoplus_{\zeta\in \Aut_{\chi,z}(G)} V_{g,\zeta}.
$$

\smallskip

\end{itemize}
Furthermore, all the right $H_{\mathcal{D}}$-braided comodules with underlying $k$-vector space $V$ have this form.

\end{corollary}

\begin{proof} Assume we have a data as in the statement. Then we define a family of maps as in~\eqref{defUgm}, by

\begin{itemize}

\smallskip

\item[-] $U^g_0(v) := \pi_g(v)$, where $\pi_g\colon V\to V_g$ is the projection onto $V_g$ along $\bigoplus_{h\in G\setminus\{g\}} V_h$,

\smallskip

\item[-] $U^g_1 := U^g_0\xcirc U\xcirc U^{gz}_0$,

\smallskip

\item[-] $U^g_j=\frac{1}{(j)!_{qp}}\, U^g_1\xcirc U^{gz}_1\xcirc\cdots\xcirc U^{gz^{j-1}}_1$ for all $1<j<n$.

\smallskip

\end{itemize}
We must check that these maps satisfy the conditions required in Proposition~\ref{caracterizacion de Hsub D comodules}. Item~(1) is fulfilled since $U(V_{\zeta})\subseteq V_{\zeta}$ and $U_0^g(V_{\zeta})\subseteq V_{\zeta}$ for all $g\in G$, items~(2)--(4) hold by the very definition of the maps $U_i^g$, and item~(6) is fulfilled since $\alpha(V_g)=V_g$ for all $g\in G$ and $U\xcirc \alpha = q\,\alpha\xcirc U$. We next prove item~(5). Since $U^n = 0$, this trivially follows if we prove that, for all $j\ge 1$,
$$
U^g_1\xcirc U^{gz}_1\xcirc\cdots\xcirc U^{gz^{j-1}}_1(v) = \begin{cases} U^j(v) &\text{if $v\in V_{gz^j}$,}\\ 0& \text{if $v\in V_h$ with $h\ne gz^j$.}\end{cases}
$$
Clearly if $v\in V_h$ with $h\ne gz^j$, then $U_0^{gz^j}(v) = 0$, and so
$$
U^g_1\xcirc U^{gz}_1\xcirc\cdots\xcirc U^{gz^{j-1}}_1(v) = U^g_1\xcirc\cdots\xcirc U^{gz^{j-1}}_1\xcirc U_0^{gz^j}(v) = 0.
$$
It remains to consider the case $v\in V_{gz^j,\zeta}$. We proceed by induction on $j$. If $j=1$, then
$$
U_1^g(v) = U_0^g\xcirc U\xcirc U_0^{gz}(v) = U_0^g\xcirc U(v) = U(v),
$$
because $U(v)\in V_g$. Assume now $j>1$ and the result is valid for $j-1$. Then
$$
U^g_1\xcirc U^{gz}_1\xcirc\cdots\xcirc U^{gz^{j-1}}_1(v) = U^g_1\xcirc\cdots\xcirc U^{gz^{j-2}}_1 \xcirc U^{gz^{j-1}}_1(v) = U^g_1\xcirc\cdots\xcirc U^{gz^{j-2}}_1 \xcirc U(v) = U^j(v),
$$
where the last equality follows from the inductive hypothesis and the fact that $U(v)\in V_{gz^{j-1}}$.

\smallskip

Conversely assume that $(V,s)$ is a right $H_{\mathcal{D}}$-comodule via a coaction $\nu\colon V\to V\ot H_{\mathcal{D}}$. Let
$$
V=\bigoplus_{\zeta\in \Aut_{\chi,z}(G)} V_{\zeta}\qquad\text{and}\qquad \alpha\colon V\longrightarrow V
$$
be the decomposition and the automorphism associated with $s$ (see Proposition~\ref{estructuras trenzadas de HsubD}). By items~(1) and~(2) of Proposition~\ref{caracterizacion de Hsub D comodules}, we know that, for each $\zeta\in \Aut_{\chi,z}(G)$, the maps $U^g_0$'s determine by restriction a complete family $(U^g_0\colon V_{\zeta}\to V_{\zeta})_{g\in G}$ of orthogonal idempotents. Let
$$
V_{\zeta} = \bigoplus_{g\in G} V_{g,\zeta}
$$
be the decomposition associated with this family. Clearly
$$
V = \bigoplus_{(g,\zeta)\in G\times \Aut_{\chi,z}(G)} V_{g,\zeta}.
$$
By item~(6) of Proposition~\ref{caracterizacion de Hsub D comodules}, we have $\alpha\xcirc U^g_0 = U^g_0\xcirc \alpha$ for all $g\in G$. Since, by Proposition~\ref{estructuras trenzadas de HsubD} we know that $\alpha(V_{\zeta}) = V_{\zeta}$ for all $\zeta\in \Aut(G)$, this implies that
$$
\alpha(V_{g,\zeta})=V_{g,\zeta}\qquad \text{ for all $(g,\zeta)\in G\times \Aut_{\chi,z}(G)$}.
$$
We now define a map $U\colon V\to V$ by
$$
U(v) = U^g_1(v)\qquad\text{for all $v\in V_{gz,\zeta}$.}
$$
From items~(3) and (5) of Proposition~\ref{caracterizacion de Hsub D comodules}, it follows that $U^n = 0$, and using the second equality in item~(6) of the same proposition, we obtain that $\alpha\xcirc U =q\, U\xcirc\alpha$. Finally,  by items~(1) and~(3) of Proposition~\ref{caracterizacion de Hsub D comodules}, we have $U(V_{g,\zeta})\subseteq V_{z^{-1}g,\zeta}$ for all $(g,\zeta)\in G\times \Aut_{\chi,z}(G)$.

\smallskip

We leave the reader the task to prove that the construction given in the two parts of this proof are reciprocal one of each other.
\end{proof}

\begin{remark} Assume that $q=1$, or equivalently, that $H_{\mathcal{D}}$ is a Krop-Radford Hopf algebra. In this case $(V,s)$ is a standard $H_{\mathcal{D}}$-comodule (that is, $s$ is the flip) iff $V_{g,\zeta}= 0$ for $\zeta \ne \ide$ and $\alpha$ is the identity map. Hence, in order to obtain the standard $H_{\mathcal{D}}$-comodule structures, the conditions that we need verify (given in Corollary~\ref{segunda caracterizacion de HsubD comodules}) are considerably simplified.
\end{remark}

\begin{corollary}\label{coinvarianes de un HsubD comodulo} With the notations of the previous corollary, $V^{\coH}=V_{1_G}\cap \ker(U)$.
\end{corollary}

\begin{proof} This is an immediate consequence of Corollary~\ref{segunda caracterizacion de HsubD comodules}.
\end{proof}

\begin{proposition}\label{estructura de transposiciones de HsubD} Let $B$ be an algebra. If
\begin{equation}
B = \bigoplus_{\zeta\in \Aut_{\chi,z} (G)^{\op}} B_{\zeta}\label{eqrankone8}
\end{equation}
is an $\Aut_{\chi,z}(G)^{\op}$-gradation of $B$ as an algebra and $\alpha\colon B\to B$ an automorphism of algebras that satisfies
\begin{itemize}

\smallskip

\item[-] $\alpha(B_{\zeta}) = B_{\zeta}$ for all $\zeta\in\Aut_{\chi,z}(G)$,

\smallskip

\item[-] $\alpha^n=\ide$ if $\lambda(z^n-1_G)\ne 0$,

\smallskip

\end{itemize}
then, the map $s\colon H_{\mathcal{D}}\ot B\longrightarrow B\ot H_{\mathcal{D}}$, given by
\begin{equation}
s(gx^i\ot b)=\alpha^i(b)\ot \zeta(g)x^i\quad\text{for all $b\in B_{\zeta}$,}\label{eq2}
\end{equation}
is a left transposition of $H_{\mathcal{D}}$ on the algebra $B$. Furthermore, all the left transpositions of $H_{\mathcal{D}}$ on $B$ have this form.
\end{proposition}

\begin{proof} By Proposition~\ref{estructuras trenzadas de HsubD} in order to prove this it suffices to check that the formula~\eqref{eq2} defines a map compatible with the unit and the multiplication map of $B$ iff~\eqref{eqrankone8} is a gradation of $B$ as an algebra and $\alpha$ is an automorphism of algebras. We left this task to the reader.
\end{proof}

The group $\Aut_{\chi,z}(G)$ acts on $G^{\op}$ via $\zeta\cdot g:= \zeta(g)$. So, it makes sense to consider the semidirect product $G(\chi,z):= G^{\op}\rtimes \Aut_{\chi,z}(G)$.

\begin{definition}\label{graduacion compatible con D} Let $\mathcal{D}=(G,\chi,z,\lambda,q)$ be as in Corollary~\ref{algebras de Hopf trenzadas de K-R} and let $B$ be an algebra endowed with an algebra automorphism $\alpha\colon B\to B$, a map $U\colon B\to B$ and a $G(\chi,z)^{\op}$-gradation
\begin{equation}
B=\bigoplus_{(g,\zeta)\in G(\chi,z)^{\op}} B_{g,\zeta},\label{eeqq1}
\end{equation}
of $B$ as a vector space. We say that the decomposition~\eqref{eeqq1} of $B$ is {\em compatible with $\mathcal{D}$} if one of the following conditions is fulfilled:

\begin{enumerate}

\smallskip

\item $\lambda(z^n-1_G)=0$ and $\eqref{eeqq1}$ is a gradation of $B$ as an algebra.

\smallskip

\item $\lambda(z^n-1_G)\ne 0$, $1_B\in B_{1_G,\ide}$,
\begin{equation}\label{compa'}
\qquad\qquad\quad B_{g,\zeta}B_{h,\phi} \subseteq B_{\phi(g)h,\phi\xcirc \zeta}\oplus B_{z^{-n}\phi(g)h,\phi\xcirc \zeta}\quad\text{for all $(g,\zeta),(h,\phi) \in G(\chi,z)^{\op}$,}
\end{equation}
and, for each $b\in B_{g,\zeta}$ and $c\in B_{h,\phi}$, the homogeneous component $(bc)_{z^{-n}\phi(g)h,\phi\xcirc \zeta}$ of $bc$ of degree $(z^{-n}\phi(g)h,\phi\xcirc \zeta)$ is given by
\begin{equation}\label{compa}
\qquad (bc)_{z^{-n}\phi(g)h,\phi\xcirc \zeta}:= -\lambda \sum_{j=1}^{n-1}\frac{p^{j^2}\chi(h)^j}{(j)!_{qp} (n-j)!_{qp}} U^j(b)\alpha^j \bigl(U^{n-j}(c)\bigr).
\end{equation}
\end{enumerate}
\end{definition}

\begin{theorem}\label{caracterizacion de HsubD comodule algebras} Let $B$ be an algebra. Each data consisting of

\begin{itemize}

\smallskip

\item[-] a $G(\chi,z)^{\op}$-gradation
\begin{equation}
\qquad B=\bigoplus_{(g,\zeta)\in G(\chi,z)^{\op}} B_{g,\zeta},\label{eeeq1}
\end{equation}
of $B$ as a vector space,

\smallskip

\item[-] an algebra automorphism $\alpha\colon B\to B$ of $B$ such that
$$
\qquad \alpha^n=\ide\text{ if }\lambda(z^n-1_G)\ne 0\quad\text{and}\quad \alpha(B_{g,\zeta})=B_{g,\zeta}\text{ for all $(g,\zeta)\in G(\chi,z)^{\op}$,}
$$

\smallskip

\item[-] a map $U\colon B \to B$ such that
\begin{align}
&\quad\text{the decomposition~\eqref{eeeq1} is compatible with $\mathcal{D}$,}\notag\\
&\quad U\xcirc \alpha = q\,\alpha\xcirc U,\notag\\
&\quad U^n = 0,\notag\\
&\quad U(B_{g,\zeta})\subseteq B_{z^{-1}g,\zeta} \quad\text{for all $(g,\zeta)\in G(\chi,z)^{\op}$}\notag
\shortintertext{and}
&\quad U(bc) = bU(c) + \chi(h)U(b)\alpha(c)\quad\text{for all $b\in B$ and $c\in B_h$,}\label{eqrankone8*}
\end{align}
where
$$
\qquad B_h:=\bigoplus_{\zeta\in \Aut_{\chi,z}(G)} B_{h,\zeta}\qquad\text{for all $h\in G$,}
$$

\smallskip

\end{itemize}
determines a right $H_{\mathcal{D}}$-comodule algebra $(B,s)$, in which $s\colon H_{\mathcal{D}}\ot B \longrightarrow B\ot H_{\mathcal{D}}$ is the left transposition of $H_{\mathcal{D}}$ on $B$ associated with the map $\alpha$ and the $\Aut_{\chi,z}(G)^{\op}$-gradation of $B$
\begin{equation}
B=\bigoplus_{\zeta\in \Aut_{\chi,z}(G)^{\op}}B_{\zeta},\label{eqqqq}
\end{equation}
where $B_{\zeta}:=\bigoplus_{g\in G} B_{g,\zeta}$. The coaction $\nu\colon B\longrightarrow B\ot H_{\mathcal{D}}$ of $(B,s)$ is given by
\begin{equation}\label{eeeq2}
\nu(b):=\sum_{i=0}^{n-1} \frac{1}{(i)!_{qp}} U^i(b)\ot z^{-i} gx^i\qquad\text{for all $b\in B_g$.}
\end{equation}
Furthermore, all the right $H_{\mathcal{D}}$-braided comodule algebra structures with underlying algebra $B$ are obtained in this way.
\end{theorem}

\begin{proof} Let $(B,s)$ be a right $H_{\mathcal{D}}$-comodule, with $s$ a left transposition of $H_{\mathcal{D}}$ on the algebra $B$. Consider the subspaces $B_{g,\zeta}$ of $B$ and the maps $\alpha$ and $U$ associated with $(B,s)$ as in Corollary~\ref{segunda caracterizacion de HsubD comodules}. By that corollary, Proposition~\ref{estructura de transposiciones de HsubD}, and Remarks~\ref{re: H-braided comodule} and~\ref{re: H-braided comod alg}, in order to finish the proof it suffices to show that the coaction $\nu$ of $(B,s)$ satisfy
\begin{equation}
\nu(1_B)=1_B\ot 1_{H_{\mathcal{D}}}\quad\text{and}\quad\nu\xcirc \mu_B = (\mu_B\ot \mu_{H_{\mathcal{D}}})\xcirc (B\ot s\ot H_{\mathcal{D}})\xcirc (\nu \ot \nu) \label{eqa}
\end{equation}
iff the decomposition
\begin{equation}
B=\bigoplus_{(g,\zeta)\in G(\chi,z)^{\op}} B_{g,\zeta},\label{eeqq4}
\end{equation}
of $B$, is compatible with $\mathcal{D}$ and $U$ satisfies condition~\eqref{eqrankone8*}. First we make some remarks. Let $b\in B_{g,\zeta}$ and $c\in B_{h,\phi}$. By the definition of $\nu$,
\begin{equation}
\nu(bc) =  \sum_{i=0}^{n-1} \sum_{f\in G} \frac{1}{(i)!_{qp}} U^i\bigl((bc)_f\bigr)\ot z^{-i}fx^i.\label{eeqq5}
\end{equation}
On the other hand, a direct computation shows that
\begin{equation}
\begin{aligned}
F(b,c) &=\sum_{j=0}^{n-1} \sum_{i=0}^{n-1} \frac{p^{-ij}\chi(h)^j}{(j)!_{qp}(i)!_{qp}} U^j(b)\alpha^j\bigl(U^i(c)\bigr)\ot z^{-i-j}\phi(g) h x^{i+j}\\
&=\sum_{u=0}^{2n-2}\sum_{\substack{j=0\\0\le u-j<n}}^{n-1} \frac{p^{-(u-j)j}\chi(h)^j}{(j)!_{qp}(u-j)!_{qp}} U^j(b)\alpha^j\bigl(U^{u-j}(c)\bigr)\ot z^{-u}\phi(g) h x^u,
\end{aligned}\label{eeqq6'}
\end{equation}
where to abbreviate expressions we write
$$
F(b,c):=(\mu_B\ot\mu_{H_{\mathcal{D}}})\xcirc (B\ot s \ot H_{\mathcal{D}})(\nu(b)\ot\nu(c)).
$$
Set
$$
A_u^j(b,c):= \frac{p^{(j-u)j}\chi(h)^j}{(j)!_{qp}(u-j)!_{qp}} U^j(b)\alpha^j\bigl(U^{u-j}(c)\bigr).
$$
Since $x^n=\lambda(z^n-1_G)$, equality~\eqref{eeqq6'} becomes
\begin{equation}
\begin{aligned}
F(b,c)&=\sum_{i=0}^{n-1} \sum_{j=0}^i A_i^j(b,c) \ot z^{-i}\phi(g) hx^i +\lambda \sum_{i=0}^{n-1} \sum_{j=i+1}^{n-i} A_{i+n}^j(b,c) \ot z^{-n-i}\phi(g) h (z^n-1_G)x^i\\
&=\sum_{i=0}^{n-1}\biggl(\sum_{j=0}^i A_i^j(b,c)+\lambda \sum_{j=i+1}^{n-1} A_{i+n}^j(b,c)\biggr) \ot z^{-i}\phi(g) hx^i\\
& -\sum_{i=0}^{n-1} \sum_{j=i+1}^{n-1} \lambda A_{i+n}^j(b,c) \ot z^{-n-i}\phi(g) h x^i.
\end{aligned}\label{eeqq6}
\end{equation}

Next we prove the part $\Rightarrow$). We assume that $\lambda(z^n-1_G)\ne0$ and leave the case $\lambda(z^n-1_G)=0$, which is easier, to the reader. To begin with note that by the first equality in~\eqref{eqa},
$$
1_B\in B_{1_G}=\bigoplus_{\zeta\in \Aut_{\chi,z}(G)} B_{1_G,\zeta}.
$$
Since, on the other hand,~\eqref{eqqqq} is a gradation of $B$ as an algebra, necessarily
$$
1_B\in B_{\ide}=\bigoplus_{g\in G} B_{g,\ide},
$$
and so $1_B\in B_{1_G,\ide}$. Recall that $b\in B_{g,\zeta}$ and $c\in B_{h,\phi}$. Since by the second equality in~\eqref{eqa}, equations~\eqref{eeqq5} and~\eqref{eeqq6} coincide, we have
\begin{equation}\label{eqqqq1}
(bc)_f = \begin{cases} A_0^0(b,c)+\lambda \sum_{j=1}^{n-1} A_n^j(b,c)&\text{if $f = \phi(g)h$,}\\- \lambda \sum_{j=1}^{n-1} A_n^j(b,c)&\text{if $f = z^{-n}\phi(g)h$,}\\0&\text{otherwise,}\end{cases}
\end{equation}
and
\begin{equation}\label{eqqqq2}
U\bigl((bc)_f\bigr) = \begin{cases}\sum_{j=0}^1A_1^j(b,c)+\lambda \sum_{j=2}^{n-1} A_{1+n}^j(b,c)&\text{if $f = \phi(g)h$,}\\ - \lambda\sum_{j=2}^{n-1} A_{1+n}^j(b,c)&\text{if $f = z^{-n}\phi(g)h$,}\\ 0&\text{otherwise.}\end{cases}
\end{equation}
Since, by Proposition~\ref{estructura de transposiciones de HsubD}, we know that $bc \in B_{\phi\circ \zeta}$, from equality~\eqref{eqqqq1} it follows easily that the decomposition~\eqref{eeeq1} is compatible with $\mathcal{D}$ (recall that if $\lambda(z^n-1_G)\ne 0$, then $p^n=1$). Finally, by~\eqref{eqqqq2}
$$
U(bc) = \sum_f U\bigl((bc)_f\bigr) = \sum_{j=0}^1 A_1^j(b,c)+\lambda \sum_{j=2}^{n-1} A_{1+n}^j(b,c) -\lambda \sum_{j=2}^{n-1} A_{1+n}^j(b,c) = A_1^0(b,c) + A_1^1(b,c),
$$
and so, \eqref{eqrankone8*} is true.

\smallskip

We now prove the part $\Leftarrow$). So we assume that the decomposition~\eqref{eeqq4} is compatible with $\mathcal{D}$ and that $U$ satisfies~\eqref{eqrankone8*}. Again we consider the case $\lambda(z^n-1_G)\ne0$ and leave the case $\lambda(z^n-1_G)=0$, which is easier, to the reader. To begin with note that $\nu(1_B) = 1_B\ot 1_{H_{\mathcal{D}}}$, because
$$
1_B\in B_{1_G}\qquad\text{and}\qquad U(1_B)=1_BU(1_B)+U(1_B)1_B \Rightarrow U(1_B)=0.
$$
So, we are reduce to prove that the second condition in~\eqref{eqa} is fulfilled. By equalities~\eqref{eeqq5} and~\eqref{eeqq6}, this is equivalent to prove that for all $0\le i < n$ and $f\in G$,
\begin{equation}\label{eqqqq3}
\frac{1}{(i)!_{qp}} U^i\bigl((bc)_f\bigr) = \begin{cases}\sum_{j=0}^i A_i^j(b,c)+\lambda \sum_{j=i+1}^{n-1} A_{i+n}^j(b,c)&\text{if $f = \phi(g)h$,}\\ - \lambda\sum_{j=i+1}^{n-1} A_{i+n}^j(b,c)&\text{if $f = z^{-n}\phi(g)h$,}\\ 0&\text{otherwise.}\end{cases}
\end{equation}
For $i=0$ this follows easily from the fact that equality~\eqref{compa} holds and
$$
A_0^0(b,c) = bc = (bc)_{\phi(g)h,\phi\xcirc \zeta} + (bc)_{z^{-n}\phi(g)h,\phi\xcirc \zeta},
$$
since the decomposition~\eqref{eeqq4} is compatible with $\mathcal{D}$. Assume by inductive hypothesis that equality~\eqref{compa} is true for $i$ and that $i<n-1$. This implies that
\begin{equation*}
\frac{1}{(i)!_{qp}}U^{i+1}\bigl((bc)_f\bigr) = \begin{cases}\sum_{j=0}^i  U\bigl(A_i^j(b,c)\bigr)+\lambda \sum_{j=i+1}^{n-1} U\bigl(A_{i+n}^j(b,c) \bigr)&\text{if $f = \phi(g)h$,}\\ - \lambda\sum_{j=i+1}^{n-1} U\bigl(A_{i+n}^j(b,c)\bigr)&\text{if $f = z^{-n}\phi(g)h$,}\\ 0&\text{otherwise.}\end{cases}
\end{equation*}
So, we must prove that
\begin{equation}\label{pepe1}
\frac{1}{(i+1)_{qp}} \sum_{j=i+1}^{n-1} U\bigl(A_{i+n}^j(b,c)\bigr) = \sum_{j=i+2}^{n-1} A_{i+1+n}^j(b,c)
\end{equation}
and
\begin{equation}\label{pepe2}
\frac{1}{(i+1)_{qp}} \sum_{j=0}^i  U\bigl(A_i^j(b,c)\bigr) = \sum_{j=0}^{i+1} A_{i+1}^j(b,c).
\end{equation}
Recall again that $b\in B_{g,\zeta}$ and $c\in B_{h,\phi}$. Using the equality~\eqref{eqrankone8*} and the facts that $U\xcirc \alpha = q \alpha\xcirc U$, $U^u(c) \in B_{z^{-u}h,\phi}$ for all $u\in \mathds{N}$, and $p^n = 1$, we obtain
\begin{align*}
U\bigl(A_i^j(b,c)\bigr) &= \frac{p^{(j-i)j}\chi(h)^j}{(j)!_{qp}(i-j)!_{qp}} U\bigl(U^j(b)\alpha^j\bigl(U^{i-j}(c)\bigr)\bigr)\\
& = \frac{p^{(j-i)j}\chi(h)^j}{(j)!_{qp}(i-j)!_{qp}} \Bigl(q^j U^j(b)\alpha^j\bigl(U^{i+1-j}(c)\bigr) + p^{j-i} \chi(h) U^{j+1}(b)\alpha^{j+1}\bigl(U^{i-j}(c)\bigr)\Bigr)
\end{align*}
and
\begin{align*}
U\bigl(A_{i+n}^j&(b,c)\bigr) = \frac{p^{(j-i)j}\chi(h)^j}{(j)!_{qp}(i+n-j)!_{qp}} U\bigl(U^j(b)\alpha^j\bigl(U^{i+n-j}(c)\bigr)\bigr)\\
& = \frac{p^{(j-i)j}\chi(h)^j}{(j)!_{qp}(i+n-j)!_{qp}} \Bigl(q^j U^j(b)\alpha^j\bigl(U^{j+1+n-j}(c)\bigr) +
p^{j-i} \chi(h) U^{j+1}(b)\alpha^{j+1}\bigl(U^{i+n-j}(c)\bigr)\Bigr).
\end{align*}
Since $U^n = 0$, this implies that
\begin{align*}
\sum_{j=0}^i U\bigl(A_i^j(b,c)\bigr) & = \sum_{j=0}^i \frac{p^{(j-i)j}\chi(h)^jq^j} {(j)!_{qp}(i-j)!_{qp}}  U^j(b)\alpha^j\bigl(U^{i+1-j}(c)\bigr)\\
& + \sum_{j=0}^i\frac{p^{(j-i)(j+1)}\chi(h)^{j+1}}{(j)!_{qp}(i-j)!_{qp}} U^{j+1}(b)\alpha^{j+1}\bigl(U^{i-j}(c)\bigr)\\
& = \sum_{j=0}^i \frac{p^{(j-i)j}\chi(h)^jq^j} {(j)!_{qp}(i-j)!_{qp}}  U^j(b)\alpha^j\bigl(U^{i+1-j}(c)\bigr)\\
& + \sum_{j=1}^{i+1}\frac{p^{(j-i-1)j}\chi(h)^j}{(j-1)!_{qp}(i+1-j)!_{qp}} U^j(b)\alpha^j\bigl(U^{i+1-j}(c)\bigr)
\end{align*}
and
\begin{align*}
\sum_{j=i+1}^{n-1} U\bigl(A_{i+n}^j(b,c)\bigr) & = \sum_{j=i+1}^{n-1} \frac{p^{(j-i)j}\chi(h)^jq^j} {(j)!_{qp}(i+n-j)!_{qp}}  U^j(b)\alpha^j\bigl(U^{i+1+n-j}(c)\bigr)\\
& + \sum_{j=i+1}^{n-1}\frac{p^{(j-i)(j+1)}\chi(h)^{j+1}}{(j)!_{qp}(i+n-j)!_{qp}} U^{j+1}(b)\alpha^{j+1}\bigl(U^{i+n-j}(c)\bigr)\\
& = \sum_{j=i+2}^{n-1} \frac{p^{(j-i)j}\chi(h)^jq^j} {(j)!_{qp}(i+n-j)!_{qp}}  U^j(b)\alpha^j\bigl(U^{i+1+n-j}(c)\bigr)\\
& + \sum_{j=i+2}^{n-1}\frac{p^{(j-i-1)j}\chi(h)^j}{(j-1)!_{qp}(i+1+n-j)!_{qp}} U^j(b)\alpha^j\bigl(U^{i+1+n-j}(c)\bigr).
\end{align*}
Consequently in order to finish the proof of equalities~\eqref{pepe1} and~\eqref{pepe2} it suffices to see that
\begin{align*}
& \frac{(i+1)_{qp}} {(i+1)!_{qp}} = \frac{1} {(i)!_{qp}},\\
&\frac{(i+1)_{qp} \chi(h)^{i+1}} {(i+1)!_{qp}} = \frac{ \chi(h)^{i+1}}{(i)!_{qp}},\\
&\frac{(i+1)_{qp} p^{(j-i-1)j}\chi(h)^j} {(j)!_{qp}(i+1-j)!_{qp}} = \frac{p^{(j-i)j}\chi(h)^jq^j} {(j)!_{qp}(i-j)!_{qp}} + \frac{p^{(j-i-1)j} \chi(h)^j}{(j-1)!_{qp}(i+1-j)!_{qp}} &&\text{for $1\le j\le i$}
\intertext{and}
& \frac{(i+1)_{qp} p^{(j-i-1)j}\chi(h)^j} {(j)!_{qp}(i+1+n-j)!_{qp}} = \frac{p^{(j-i)j}\chi(h)^jq^j} {(j)!_{qp}(i+n-j)!_{qp}} + \frac{p^{(j-i-1)j} \chi(h)^j}{(j-1)!_{qp}(i+1+n-j)!_{qp}}&&\text{for $i+2\le j< n$.}
\end{align*}
But the first two equalities are trivial and the last ones are equivalent to
\begin{align*}
&(i+1)_{qp} = p^j q^j(i+1-j)_{qp} + (j)_{qp} &&\text{for $1\le j\le i$}
\intertext{and}
&(i+1)_{qp} = p^j q^j(i+1+n-j)_{qp} + (j)_{qp} &&\text{for $i+2\le j< n$,}
\end{align*}
which can be easily checked.
\end{proof}

\begin{remark} Assume that $q=1$, or equivalently, that $H_{\mathcal{D}}$ is a Krop-Radford Hopf algebra. In this case $(B,s)$ is a standard $H_{\mathcal{D}}$-comodule algebra (that is, $s$ is the flip) iff $B_{g,\zeta}= 0$ for $\zeta \ne \ide$ and $\alpha$ is the identity map. Hence, in order to obtain the standard $H_{\mathcal{D}}$-comodule algebra structures, the conditions that we need verify (given in Theorem~\ref{caracterizacion de HsubD comodule algebras}) are considerably simplified.
\end{remark}

\begin{remark}\label{estructura inducida de kG-modulo algebra} Let $(B,s)$ be a right $H_{\mathcal{D}}$-comodule algebra and let $B_{G}:=\{b\in B: \nu(b)\in B\ot kG\}$. Note that
$$
B_G=\ker(U)= \bigoplus_{(g,\zeta)\in G(\chi,z)^{\op}} B_{g,\zeta}\cap \ker(U).
$$
Moreover,
$$
\nu(B_G)\subseteq B_G\ot kG,
$$
because $(\nu\ot H_{\mathcal{D}})\xcirc \nu=B\ot \Delta)\xcirc \nu$. Consequently, since
$$
(B\ot c_q)\xcirc (s\ot H_{\mathcal{D}})\xcirc (H_{\mathcal{D}}\ot \nu)= (\nu\ot H_{\mathcal{D}})\xcirc s,
$$
we have $s(H_{\mathcal{D}}\ot B_G)\subseteq B_G\ot H_{\mathcal{D}}$. Similarly $s^{-1}(B_G\ot H_{\mathcal{D}})\subseteq H_{\mathcal{D}}\ot B_G$, and so,
$$
s(H_{\mathcal{D}}\ot B_G) = B_G\ot H_{\mathcal{D}}.
$$
Furthermore $B_G$ is a subalgebra of $B$ because
$$
\nu\xcirc \mu = (\mu_B\ot \mu_{H_{\mathcal{D}}})\xcirc (B \ot s \ot H_{\mathcal{D}}) \xcirc (\nu \ot \nu).
$$
Clearly $s$ induce by restriction a left transposition $\tilde{s}$  of $kG$ on $B_G$. From the previous discussion it follows that $(B_G,\tilde{s})$ is a right $kG$-comodule algebra.
\end{remark}

\section{$\bm{H_{\mathcal{D}}}$ cleft extensions} Throughout this section we use freely the notations introduced in Section~\ref{H_D comodule algebras} and the characterization of right $H_{\mathcal{D}}$-comodule algebras obtained in Theorem~\ref{caracterizacion de HsubD comodule algebras}. Let $(B,s)$ be a right $H_{\mathcal{D}}$-comodule algebra and let $C:=B^{\coH_{\mathcal{D}}}$. Recall that, by Corollary~\ref{coinvarianes de un HsubD comodulo},
$$
C=B_{1_G}\cap \ker(U)= \bigoplus_{\zeta\in \Aut_{\chi,z}(G)} B_{1_G,\zeta}\cap \ker(U).
$$

\begin{theorem}\label{caracterizacion de cleft} The extension $(C \hookrightarrow B,s)$ is cleft iff there exists $b_x\in B$ and a family $(b_g)_{g\in G}$ of elements of $B^{\times}$, such that
\begin{enumerate}[label=\emph{(\alph*)}]

\smallskip

\item $b_g\in B_{g,\ide}\cap \ker(U)$ for all $g\in G$,

\smallskip

\item $b_x\in B_{z,\ide}\cap U^{-1}(1_B)$,

\smallskip

\item $\alpha(b_x)=qb_x$,

\smallskip

\item $\alpha(b_g)=b_g$ for all $g\in G$.

\smallskip

\end{enumerate}
If this is the case, then the map $\gamma\colon H_{\mathcal{D}} \to B$, defined by $\gamma(gx^i):=b_gb_x^i$, is a cleft map, and its convolution inverse is given by
$$
\qquad \gamma^{-1}(gx^i) = (-1)^i (qp)^{\frac{i(i-1)}{2}}b_x^ib_{gz^i}^{-1}.
$$
\end{theorem}

\begin{proof} Assume that $(C \hookrightarrow B,s)$ is a cleft extension and fix a cleft map $\gamma \colon H_{\mathcal{D}} \to B$ such that $\gamma(1)=1$. For every $g\in G$ and $0\le i< n$, set $b_{gx^i}:= \gamma(gx^i)$. Since $\gamma$ is a right comodule map,
$$
\nu(b_g)=b_g\ot g \quad\text{and}\quad \nu(b_x)=1_B\ot x + b_x\ot z,
$$
which, by formula~\eqref{eeeq2}, is equivalent to
$$
b_g\in B_g \cap \ker(U)\quad\text{and}\qquad b_x\in B_z\cap U^{-1}(1).
$$
Moreover $b_g$ is invertible for each $g\in G$, because $\gamma$ is convolution invertible. On the other hand evaluating the equality $(\gamma\ot H_{\mathcal{D}})\xcirc c_q = s\xcirc (H_{\mathcal{D}}\ot \gamma)$ in $h\ot x$, $x\ot x$, $h\ot g$ and $x\ot g$, where $h\in G$ is arbitrary, we obtain that
$$
b_x\in B_{\ide},\quad \alpha(b_x)=qb_x,\quad  b_g \in B_{\ide}\quad\text{and}\quad \alpha(b_g)=b_g,
$$
for all $g\in G$. Thus, items~\mbox{(a)--(d)} hold. Conversely, assume that there exists $b_x\in B$ and a family $(b_g)_{g\in G}$ of elements of $B^{\times}$ satisfying statements~(a)--(d). We are going to prove that $(C \hookrightarrow B,s)$ is cleft and the map $\gamma\colon H_{\mathcal{D}} \to B$, defined by $\gamma(gx^i):=b_gb_x^i$, is a cleft map. First note that
$$
(\gamma\ot H_{\mathcal{D}})\xcirc c_q (gx^i\ot hx^j) = q^{ij} b_hb_x^j\ot gx^i= s(gx^i\ot b_hb_x^j)= s\xcirc (H_{\mathcal{D}}\ot \gamma) (gx^i\ot hx^j),
$$
for all $h,g\in G$ and $0\le i,j<n$. So we must only check that $\gamma$ is convolution invertible and
\begin{equation}\label{eeeq3}
\nu\xcirc \gamma(gx^i)= (\gamma\ot H_{\mathcal{D}})\xcirc \Delta(gx^i) \quad\text{for all $g\in G$ and $0\le i < n$}.
\end{equation}
For $g=1$ and $i=0$  it is evident that this is true. Assume it is true for $g=1$ and $i=i_0$, and that $i_0<n-1$. Then
\begin{align*}
\nu\xcirc \gamma(x^{i_0+1}) &=\nu\bigl(\gamma(x^{i_0})  b_x\bigr) \\
&=(\mu_{B}\ot \mu_{H_{\mathcal{D}}})\xcirc (B\ot s \ot H_{\mathcal{D}})\bigl(\nu(\gamma(x^{i_0})) \ot \nu(b_x)\bigr)\\
&= \sum_{j=0}^{i_0}\binom{i_0}{j}_{qp}(\mu_{B}\ot \mu_{H_{\mathcal{D}}})\xcirc (B\ot s \ot H_{\mathcal{D}})(b_x^j\ot z^jx^{i_0-j}\ot 1_B\ot x)\\
&+ \sum_{j=0}^{i_0}\binom{i_0}{j}_{qp}(\mu_{B}\ot \mu_{H_{\mathcal{D}}})\xcirc (B\ot s \ot H_{\mathcal{D}})(b_x^j\ot z^jx^{i_0-j}\ot b_x\ot z)\\
&= \sum_{j=0}^{i_0}\binom{i_0}{j}_{qp} b_x^j\ot z^jx^{i_0+1-j} + \sum_{j=0}^{i_0}\binom{i_0}{j}_{qp} b_x^j \alpha^{i_0-j}(b_x)\ot z^jx^{i_0-j}z\\
&= \sum_{j=0}^{i_0}\binom{i_0}{j}_{qp} b_x^j\ot z^jx^{i_0+1-j} + \sum_{j=1}^{i_0+1}\binom{i_0}{j-1}_{qp} q^{i_0+1-j}p^{i_0+1-j}b_x^j \ot z^jx^{i_0+1-j}\\
&=\sum_{j=0}^{i_0}\binom{i_0+1}{j}_{qp} b_x^j\ot z^jx^{i_0+1-j}.
\end{align*}
So, equality~\eqref{eeeq3} holds when $g=1_G$. But then
\begin{align*}
\nu\xcirc \gamma(gx^i) &= \nu\bigl(b_g\gamma(x^i)\bigr)\\
&= \sum_{j=0}^i\binom{i}{j}_{qp} (\mu_B\ot \mu_{H_{\mathcal{D}}})\xcirc (B\ot s\ot H_{\mathcal{D}})(b_g\ot g \ot b_x^j \ot z^jx^{i-j})\\
&= \sum_{j=0}^i\binom{i}{j}_{qp} b_gb_x^j\ot gz^jx^{i-j}.
\end{align*}
It remains to check that $\gamma$ is convolution invertible. As was noted in \cite{D-T}*{Section 3},
\begin{equation}
\sum_{j=0}^i (-1)^j(qp)^{\frac{j(j-1)}{2}}\binom{i}{j}_{qp}=\begin{cases} 1 & \text{if $i=0$,}\\ 0 &\text{if $0<i<n$.}\end{cases} \label{formula}
\end{equation}
Using this it is easy to prove that $\gamma$ is invertible with
$$
\gamma^{-1}(gx^i) = (-1)^i (qp)^{\frac{i(i-1)}{2}}b_x^ib_{gz^i}^{-1},
$$
which finishes the proof.
\end{proof}

In the previous theorem we can assume without loose of generality that $b_1=1_B$.

\begin{theorem}\label{algunas propiedades de las extensiones cleft} Assume that $(C \hookrightarrow B,s)$ is cleft. Take $b_x\in B$ and a family $(b_g)_{g\in G}$ of elements of $B^{\times}$ with $b_1=1_B$, in such a way that conditions~(a)--(d) of Theorem~\ref{caracterizacion de cleft} are fulfilled. Then

\begin{enumerate}

\smallskip

\item $B$ is a free left $C$-module with basis $\{b_gb_x^i : g\in G \text{ and } 0\le i< n\}$.

\smallskip

\item Set $\mathfrak{b}:=b_x^nb_z^{|z|-n}$ and for all $g\in G$ set $\mathfrak{a}_g:=b_g^{|g|}$ and $\mathfrak{c}_g:= (b_xb_g- \chi(g) b_gb_x)b_g^{-1}b_z^{-1}$. Then $\mathfrak{a_g}\in C^{\times}$, $\mathfrak{c}_g\in C$, and if $x^n=0$, then $\mathfrak{b}\in C$.

\smallskip

\item The weak action of $H_{\mathcal{D}}$ on $C$ associated with $\gamma$ according to item~(5) of Theorem~\ref{equiv entre cleft, H-Galois con base normal e isomorfo a un producto cruzado}, is given by
$$
\qquad\quad gx^i \rightharpoonup c=\sum_{j=0}^i (-1)^j(qp)^{\frac{j(j-1)}{2}}\binom{i}{j}_{qp} b_gb_x^{i-j} \alpha^j(c)b_x^j  b_{\zeta(g)z^i}^{-1} \quad\text{for $c\in B_{1_G,\zeta}\cap \ker(U)$.}
$$

\smallskip

\item The two cocycle $\sigma\colon H_{\mathcal{D}}\ot H_{\mathcal{D}} \to C$ associated with $\gamma$ according to item~(5) of Theorem~\ref{equiv entre cleft, H-Galois con base normal e isomorfo a un producto cruzado}, is given by
\begin{align*}
\qquad\quad\,\,\, \sigma(gx^s\ot hx^r) & = \sum_{\substack{0\le i\le s \\ 0\le j\le r \\ \xi_{ij}<n}}\! (-1)^{\xi_{ij}} \binom{s}{i}_{qp}\binom{r}{j}_{qp} (qp)^{\frac{\xi_{ij}(\xi_{ij}-1)}{2}+sj-ij}\chi(h)^{s-i} b_gb_x^i b_h b_x^{s+r-i} b_{ghz^{s+r}}^{-1}\\
& + \lambda \!\!\sum_{\substack{0\le i\le s \\ 0\le j\le r \\ \xi_{ij}\ge n}}\! (-1)^{\xi'_{ij}} \binom{s}{i}_{qp}\binom{r}{j}_{qp} (qp)^{\frac{\xi'_{ij}(\xi'_{ij}-1)}{2}+sj-ij}\chi(h)^{s-i} b_gb_x^i b_h b_x^{\xi_{in}}b_{ghz^{s+r}}^{-1}\\
& - \lambda\!\!  \sum_{\substack{0\le i\le s \\ 0\le j\le r \\ \xi_{ij}\ge n}}\! (-1)^{\xi'_{ij}}\binom{s}{i}_{qp}\binom{r}{j}_{qp} (qp)^{\frac{\xi'_{ij}(\xi'_{ij}-1)}{2}+sj-ij}\chi(h)^{s-i} b_gb_x^i b_h b_x^{\xi_{in}} b_{ghz^{s+r-n}}^{-1},
\end{align*}
where $\xi_{ij}:=s+r-i-j$ and $\xi'_{ij}:=\xi_{ij}-n$.
\end{enumerate}
\end{theorem}

\begin{proof} (1)\enspace By item~(4) of Theorem~\ref{equiv entre cleft, H-Galois con base normal e isomorfo a un producto cruzado}, the map $\phi\colon C\ot H_{\mathcal{D}}\to B$, given by $\phi(c\ot y):= c\gamma(y)$, is a normal basis. Item~(1) is an immediately consequence of this fact.

\smallskip

\noindent (2)\enspace Using item~(4) of Remark~\ref{re: H-braided comod alg} it is easy to check by induction on $i$, that
\begin{equation}
\nu(b_g^i)= b_g^i\ot g^i\quad\text{and}\quad \nu(b_x^i)=\sum_{j=0}^i \binom{i}{j}_{qp} b_x^j\ot z^jx^{i-j}\quad\text{for all $g\in G$ and $i\ge 0$}.\label{pa}
\end{equation}
By the first equality
$$
\nu(\mathfrak{a}_g)=\mathfrak{a}_g \ot 1_{H_{\mathcal{D}}} \quad\text{for all $g\in G$},
$$
and so $\mathfrak{a}_g\in C$. Since $\nu\colon B\to B\ot_s H$ is an algebra map, we have
\begin{equation}
\nu(\mathfrak{a}_g^{-1})=\mathfrak{a}_g^{-1} \ot 1_{H_{\mathcal{D}}}\quad\text{and}\quad \nu(b_g^{-1}) = b_g^{-1}\ot g^{-1} \quad\text{for all $g\in G$,}\label{pa1}
\end{equation}
which implies in particular that $\mathfrak{a}_g\in C^{\times}$. Note also that, by the second equality in~\eqref{pa},
$$
\nu(b_x^n)= 1_B\ot x^n+ b_x^n \ot z^n.
$$
because $\binom{n}{j}_{qp}=0$ for $0<j<n$. Consequently,
\begin{align*}
\nu(\mathfrak{b}) & = \nu(b_x^nb_z^{|z|-n})\\
& =(\mu_B\ot \mu_{H_{\mathcal{D}}})\xcirc (B\ot s\ot  H_{\mathcal{D}})(1_B\ot x^n\ot b_z^{|z|-n}\ot z^{-n} + b_x^n \ot z^n\ot b_z^{|z|-n}\ot z^{-n})\\
& = (\mu_B\ot \mu_{H_{\mathcal{D}}})(1_B\ot b_z^{|z|-n} \ot x^n \ot z^{-n} + b_x^n \ot b_z^{|z|-n}\ot z^n\ot z^{-n})\\
& = b_z^{|z|-n} \ot x^nz^{-n} + \mathfrak{b}\ot 1_{H_{\mathcal{D}}},
\end{align*}
which implies that, if $x^n = 0$, then $\mathfrak{b}\in C$. Furthermore, for all $g\in G$,
\begin{align*}
&\nu(b_xb_g) = (\mu_B\ot \mu_{H_{\mathcal{D}}})\xcirc (B\ot s\ot  H_{\mathcal{D}})(1_B\ot x \ot b_g \ot g + b_x\ot z \ot b_g \ot g)\\
&\phantom{\nu(b_xb_g)} = (\mu_B\ot \mu_{H_{\mathcal{D}}})(1_B\ot b_g \ot x \ot g + b_x\ot b_g \ot z \ot g)\\
&\phantom{\nu(b_xb_g)}= \chi(g)\,b_g\ot gx + b_xb_g\ot zg
\shortintertext{and}
&\nu(\chi(g)\,b_gb_x)= (\mu_B\ot \mu_{H_{\mathcal{D}}})\xcirc (B\ot s\ot  H_{\mathcal{D}})(\chi(g)\,b_g\ot g\ot 1_B\ot x + \chi(g)\,b_g\ot g\ot b_x\ot z)\\
&\phantom{\nu(\chi(g)\,b_xb_g)}=(\mu_B\ot \mu_{H_{\mathcal{D}}})(\chi(g)\,b_g\ot 1_B\ot g\ot x + \chi(g)\,b_g\ot b_x\ot g\ot z)\\
&\phantom{\nu(\chi(g)\,b_xb_g)}=\chi(g)\,b_g\ot gx + \chi(g)\,b_gb_x\ot zg.
\end{align*}
Combining this with the second equality in~\eqref{pa1}, we obtain
\begin{align*}
\nu(\mathfrak{c}_g) & = \nu\bigl((b_xb_g-\chi(g)\,b_gb_x)b_g^{-1}b_z^{-1}\bigr)\\
& = (\mu_B\ot \mu_{H_{\mathcal{D}}})\xcirc (B\ot s\ot  H_{\mathcal{D}})\bigl((b_xb_g -\chi(g)\,b_gb_x)\ot zg \ot b_g^{-1}b_z^{-1}\ot g^{-1}z^{-1})\\
& = (\mu_B\ot \mu_{H_{\mathcal{D}}})\bigl((b_xb_g -\chi(g)\,b_gb_x)\ot b_g^{-1}b_z^{-1}\ot zg \ot g^{-1}z^{-1})\\
& = (b_xb_g -\chi(g)\,b_gb_x)b_g^{-1}b_z^{-1}\ot 1_{H_{\mathcal{D}}},
\end{align*}
and so $\mathfrak{c}_g\in C$, as desired.

\smallskip

\noindent (3)\enspace This follows by a direct computation from item~(5) of Theorem~\ref{equiv entre cleft, H-Galois con base normal e isomorfo a un producto cruzado}, using equalities~\eqref{comultiplication} and~\eqref{eqrankone1}, and the formulas for $\gamma$ and $\gamma^{-1}$ that appears in Theorem~\ref{caracterizacion de cleft}.

\smallskip

\noindent (4)\enspace This follows by a direct computation from item~(5) of Theorem~\ref{equiv entre cleft, H-Galois con base normal e isomorfo a un producto cruzado}, using equalities~\eqref{comultiplication} and~\eqref{def braid}, and the formulas for $\gamma$ and $\gamma^{-1}$ that appears in Theorem~\ref{caracterizacion de cleft}.
\end{proof}

The Proposition~\ref{simplificacion} below is useful to simplify the computation of the first sum in the right side of the equality in Theorem~\ref{algunas propiedades de las extensiones cleft}(4).

\begin{lemma} With the notations of the previous result, we have
$$
\sum_{j=0}^r (-1)^{\xi_{ij}} \binom{r}{j}_{qp} (qp)^{\frac{\xi_{ij}(\xi_{ij}-1)}{2}+sj-ij} = \begin{cases} (-1)^{s-i}(qp)^{\frac{(s-i)(s-i-1)}{2}} & \text{if $r=0$,}\\ 0 &\text{if $0<r<n$.}\end{cases}
$$
\end{lemma}

\begin{proof} Let $a:= s-i$ and $b:=r-j$. Since $\xi_{ij} = a+b$, we have
\begin{align*}
\sum_{j=0}^r (-1)^{\xi_{ij}} \binom{r}{j}_{qp} (qp)^{\frac{\xi_{ij}(\xi_{ij}-1)}{2}+sj-ij}& = \sum_{b=0}^r (-1)^{a+b} \binom{r}{b}_{qp} (qp)^{\frac{(a+b)(a+b-1)}{2}+ar-ab}\\
& = (-1)^a (qp)^{ar + \frac{a(a-1)}{2}} \sum_{b=0}^r (-1)^b \binom{r}{b}_{qp} (qp)^{\frac{b(b-1)}{2}},
\end{align*}
which combined with~\eqref{formula} gives the desired result.
\end{proof}

\begin{proposition}\label{simplificacion} Let $r,s,i\ge 0$ with $0\le i\le s$. The following assertions holds

\begin{enumerate}

\smallskip

\item If $r = 0$, then
$$
\qquad \sum_{\substack{0\le j\le r \\ \xi_{ij}<n}} (-1)^{\xi_{ij}} \binom{r}{j}_{qp} (qp)^{\frac{\xi_{ij}(\xi_{ij}-1)}{2}+sj-ij} = (-1)^{s-i}(qp)^{\frac{(s-i)(s-i-1)}{2}}.
$$

\smallskip

\item If $r > 0$, then
$$
\qquad \sum_{\substack{0\le j\le r \\ \xi_{ij}<n}} (-1)^{\xi_{ij}} \binom{r}{j}_{qp} (qp)^{\frac{\xi_{ij}(\xi_{ij}-1)}{2}+sj-ij} = \sum_{\substack{0\le j\le r \\ \xi_{ij}\ge n}} (-1)^{\xi_{ij}+1} \binom{r}{j}_{qp} (qp)^{\frac{\xi_{ij}(\xi_{ij}-1)}{2}+sj-ij},
$$
where $\xi_{ij}:=s+r-i-j$
\end{enumerate}

\end{proposition}

\begin{proof} By the previous lemma.
\end{proof}

\begin{corollary} Let $r,s,i\ge 0$ with $0\le i\le s$. If $0<r< n-s+i$, then
$$
\sum_{\substack{0\le j\le r \\ \xi_{ij}<n}} (-1)^{\xi_{ij}} \binom{r}{j}_{qp} (qp)^{\frac{\xi_{ij}(\xi_{ij}-1)}{2}+sj-ij} = 0.
$$
\end{corollary}

\begin{proof} By Proposition~\ref{simplificacion}.
\end{proof}

\section{Examples} In this section we consider two examples of the braided Hopf algebras $H_{\mathcal{D}}$ defined in Corollary~\ref{algebras de Hopf trenzadas de K-R} and we apply the results obtained in the previous section in order to determine their cleft extensions.

\subsection{First example} Consider the datum $\mathcal{D}=(C_2\times C_2\times C_2,\chi,z,\lambda,q)$, where:

\begin{itemize}

\smallskip

\item[-] $C_2=\{1,g\}$ is the multiplicative group of order $2$,

\smallskip

\item[-]  $\chi\colon C_2\times C_2\times C_2\longrightarrow \mathds{C}$ is the character given by $\chi(g^{i_1},g^{i_2},g^{i_3}):= (-1)^{i_1+i_2+i_3}$,

\smallskip

\item[-] $z:=(g,g,g)$,

\smallskip

\item[-] $q=1$ and $\lambda=1$.

\smallskip

\end{itemize}
In this case $p:=\chi(z)=-1$, $n=2$ and the Hopf braided $\mathds{C}$-algebra $H_{\mathcal{D}}$ of Corollary~\ref{algebras de Hopf trenzadas de K-R} is the $\mathds{C}$-algebra generated by the group $G:= C_2\times C_2\times C_2$ and an element $x$ subject to the relations
$$
x^2=z^2-1_G=0\quad \text{and}\quad x(g^{i_1},g^{i_2},g^{i_3})=(-1)^{i_1+i_2+i_3}(g^{i_1},g^{i_2},g^{i_3})x,
$$
endowed with the standard Hopf algebra structure with comultiplication map $\Delta$, counit $\epsilon$ and antipode $S$, given by
\begin{align*}
&\Delta(\mathbf{g}):= \mathbf{g}\ot \mathbf{g}, && \Delta(x):= 1\ot x+x\ot z,\\
& \epsilon(\mathbf{g}):=1, && \epsilon(x):=0\\
& S(\mathbf{g}):=\mathbf{g}^{-1}, && S(\mathbf{g}x):=-xz^{-1}\mathbf{g}^{-1},
\end{align*}
where $\mathbf{g}$ denotes an arbitrary element of $G$. Let $S_3$ be the symmetric group in $\{1,2,3\}$. It is easy to check that the map
$$
\theta\colon S_3^{\op}\to \Aut_{\chi,z}(G),
$$
defined by $\theta(\sigma) (g^{i_1}, g^{i_2}, g^{i_3}):= (g^{i_{\sigma(1)}},g^{i_{\sigma(2)}},g^{i_{\sigma(3)}})$, is an isomorphism.

\subsubsection{$\bm{H_{\mathcal{D}}}$-spaces}\label{ex1p1} Let $V$ be a $\mathds{C}$-vector space. By Proposition~\ref{estructuras trenzadas de HsubD} to have a left $H_{\mathcal{D}}$-space structure with underlying vector space $V$ is ``the same'' that to have a gradation
$$
V=\bigoplus_{\sigma\in S_3} V_{\sigma}
$$
and an automorphism $\alpha\colon V\to V$ such that $\alpha(V_\sigma)=V_{\sigma}$ for all $\sigma\in S_3$. The structure map
$$
s\colon H_{\mathcal{D}}\ot V\longrightarrow V\ot H_{\mathcal{D}},
$$
constructed from these data, is given by
\begin{align*}
& s\bigl((g^{i_1},g^{i_2},g^{i_3})\ot v\bigr) := v\ot (g^{i_{\sigma(1)}}, g^{i_{\sigma(2)}}, g^{i_{\sigma(3)}})
\shortintertext{and}
& s\bigl((g^{i_1},g^{i_2},g^{i_3})x\ot v\bigr) := \alpha(v)\ot (g^{i_{\sigma(1)}}, g^{i_{\sigma(2)}}, g^{i_{\sigma(3)}})x,
\end{align*}
for each $v\in V_{\sigma}$.

\subsubsection{$\bm{H_{\mathcal{D}}}$-comodules} Let $V$ be a $\mathds{C}$-vector space. By Corollary~\ref{segunda caracterizacion de HsubD comodules} each right $H_{\mathcal{D}}$-comodule structure $(V,s)$ with underlying vector space $V$ is univocally determined by the following data:

\begin{enumerate}[label=\emph{(\alph*)}]

\smallskip

\item A decomposition
$$
\quad\qquad V=\bigoplus_{(\mathbf{g},\sigma)\in G\times S_3^{\op}} V_{\mathbf{g},\sigma},
$$

\smallskip

\item An automorphism $\alpha\colon V\to V$ that satisfies $\alpha\bigl(V_{\mathbf{g},\sigma}\bigr) = V_{\mathbf{g},\sigma}$ for all $(\mathbf{g},\sigma)\in G\times S_3^{\op}$,

\smallskip

\item A map $U\colon V\to V$ such that
$$
\quad\qquad U\circ \alpha= \alpha\circ U, \qquad U^2=0\qquad\text{and}\qquad U\bigl(V_{\mathbf{g},\sigma}\bigr)\subseteq V_{\mathbf{g}z,\sigma}\quad \text{for all $(\mathbf{g},\sigma)\in G\times S_3^{\op}$.}
$$
\smallskip
\end{enumerate}
The formula for the transposition $s$ of $H_{\mathcal{D}}$ on $V$ is the one obtained in Subsection~\ref{ex1p1} (where we take $V_{\sigma}:=\bigoplus_{\mathbf{g}\in G} V_{\mathbf{g},\sigma}$ for each $\sigma\in S_3$), while the $H_{\mathcal{D}}$-coaction $\nu$ is given by
$$
\nu(v) = v\ot (g^{i_1},g^{i_2},g^{i_3}) + U(v)\ot (g^{i_1+1},g^{i_2+1},g^{i_3+1})x,
$$
for $v\in \bigoplus_{\sigma\in S_3} V_{\mathbf{g},\sigma}$ with $\mathbf{g} = (g^{i_1},g^{i_2},g^{i_3})$.

\smallskip

Next given a decomposition as in item~(a), we give a proceeding to construct an automorphism $\alpha\colon V\to V$ and a map $U\colon V\to V$ satisfying the conditions required in items~(b) and~(c): First we decompose each space $V_{\mathbf{g},\sigma}$ as a direct sum
$$
V_{\mathbf{g},\sigma}=V_{\mathbf{g},\sigma}^0\oplus V_{\mathbf{g},\sigma}^1,
$$
in such a way that $\dim_{\mathds{C}}(V_{\mathbf{g},\sigma}^1)\le \dim_{\mathds{C}}(V_{\mathbf{g}z,\sigma}^0)$, and we fix an injective morphisms
$$
U_{\mathbf{g},\sigma}\colon V_{\mathbf{g},\sigma}^1\longrightarrow V_{\mathbf{g}z,\sigma}^0,
$$
for each $(\mathbf{g},\sigma)\in G\times S_3^{\op}$. Then we define $U$ on $V_{\mathbf{g},\sigma}$, by
$$
U(v) = \begin{cases} U_{\mathbf{g},\sigma}(v) & \text{if $v\in V_{\mathbf{g},\sigma}^1$,}\\ 0 & \text{if $v\in V_{\mathbf{g},\sigma}^0$.}\end{cases}
$$
It remains to construct $\alpha$. Let $(\mathbf{g},\sigma)\in G\times S_3^{\op}$ arbitrary. Since $U\circ \alpha =\alpha\circ U$, $\alpha(V_{\mathbf{g},\sigma})\subseteq V_{\mathbf{g},\sigma}$ and $U_{\mathbf{g},\sigma}$ is injective, there exists morphisms
$$
\alpha^0_{\mathbf{g},\sigma}\colon V_{\mathbf{g},\sigma}^0 \longrightarrow V_{\mathbf{g},\sigma}^0,\qquad \alpha^1_{\mathbf{g},\sigma}\colon V_{\mathbf{g},\sigma}^1 \longrightarrow V_{\mathbf{g},\sigma}^1 \qquad\text{and}\qquad \alpha^{10}_{\mathbf{g},\sigma}\colon V_{\mathbf{g},\sigma}^1 \longrightarrow V_{\mathbf{g},\sigma}^0,
$$
such that
$$
\alpha(v_0,v_1)= \bigl(\alpha^0_{\mathbf{g},\sigma}(v_0)+\alpha^{10}_{\mathbf{g},\sigma}(v_1),\alpha^1_{\mathbf{g},\sigma}(v_1)\bigr)\qquad\text{for all $(v_0,v_1)\in V_{\mathbf{g},\sigma}^0\oplus V_{\mathbf{g},\sigma}^1$.}
$$
Moreover, since $\alpha$ is an automorphism, the maps $\alpha^0_{\mathbf{g},\sigma}$ and $\alpha^1_{\mathbf{g},\sigma}$ are also automorphisms. All these maps can be constructed as follows: For each $(\mathbf{g},\sigma)\in G\times S_3^{\op}$ we take an arbitrary automorphism $\alpha^1_{\mathbf{g},\sigma}$ of $V_{\mathbf{g},\sigma}^1$. Then, for each $(\mathbf{g},\sigma)\in G\times S_3^{\op}$, we choose $\alpha^0_{\mathbf{g},\sigma}$ as an automorphism of $V_{\mathbf{g},\sigma}^0$ such that
$$
\alpha^0_{\mathbf{g},\sigma}\bigl(U_{\mathbf{g}z,\sigma}(v)\bigr) = U_{\mathbf{g},\sigma}\bigl(\alpha^1_{\mathbf{g}z,\sigma}(v)\bigr)\qquad\text{for all $v\in V^1_{\mathbf{g}z,\sigma}$}
$$
(which is forced by the condition $U\circ \alpha = \alpha \circ U$). Finally, we take $\alpha^{10}_{\mathbf{g},\sigma}$ as an arbitrary automorphism.

\begin{remark}\label{coinvariante ejemplo 1} By Corollary~\ref{coinvarianes de un HsubD comodulo} we know that $V^{\coH_{\mathcal{D}}}= V_{1_G}\cap \ker (U)$, where $V_{1_G}=\bigoplus_{\sigma\in S_3} V_{1_G,\sigma}$.
\end{remark}

\begin{remark}\label{caso clasico comodulos ejemplo 1} We are in the classical case (i.e. $s$ is the flip) iff $V_{\mathbf{g},\sigma}=0$ for $\sigma \ne \ide$ and $\alpha$ is the identity map. So, in this case the decomposition in item a) above have at most eight nonzero summands, item b) becomes trivial and the first condition in item c) also becomes trivial.
\end{remark}

\subsubsection{Transpositions of $\bm{H_{\mathcal{D}}}$ on an algebra} By Proposition~\ref{estructura de transposiciones de HsubD}, for each $\mathds{C}$-algebra $B$, to have a transposition $s\colon H_{\mathcal{D}}\ot B \longrightarrow B\ot H_{\mathcal{D}}$ is equivalent to have an algebra gradation
$$
B=\bigoplus_{\sigma\in S_3^{\op}} B_{\sigma}
$$
and an automorphism of algebras $\alpha \colon B\to B$ such that $\alpha(B_{\sigma})= B_{\sigma}$ for all $\sigma\in S_3$. The structure map $s\colon H_{\mathcal{D}}\ot B\longrightarrow B\ot H_{\mathcal{D}}$, constructed from these data, is the same as in Subsection~\ref{ex1p1}.

\subsubsection{Right $\bm{H_{\mathcal{D}}}$-comodule algebras} By the discussion above Definition~\ref{graduacion compatible con D} we know that the group $S_3^{\op}$ acts on $G$ via
$$
\sigma\cdot (g^{i_1},g^{i_2},g^{i_3}):= (g^{i_{\sigma(1)}},g^{i_{\sigma(2)}},g^{i_{\sigma(3)}}).
$$
Consider the semidirect product $G(\chi,z):= G\rtimes S_3^{\op}$. We are going to work with $G(\chi,z)^{\op}$. Its underlying set is $C_2\times C_2 \times C_2\times S_3$ and its product is given by
$$
(g^{i_1},g^{i_2},g^{i_3},\sigma) (g^{j_1},g^{j_2},g^{j_3},\tau)= (g^{j_1+i_{\tau(1)}},g^{j_2+i_{\tau(2)}},g^{j_3+i_{\tau(3)}},\sigma\circ \tau).
$$
Let $B$ be a $\mathds{C}$-algebra. By Theorem~\ref{caracterizacion de HsubD comodule algebras} to have a right $H_{\mathcal{D}}$-comodule algebra $(B,s)$ is equivalent to have

\begin{enumerate}[label=\emph{(\alph*)}]

\smallskip

\item a $G(\chi,z)^{\op}$-gradation
$$
\qquad\quad B= \bigoplus_{(\mathbf{g},\sigma)\in G(\chi,z)^{\op}} B_{\mathbf{g},\sigma}
$$
of $B$ as an algebra,

\smallskip

\item an automorphism of algebras $\alpha\colon B\to B$ such that
$$
\qquad\quad \alpha\bigl(B_{\mathbf{g},\sigma}\bigr) \subseteq B_{\mathbf{g},\sigma}\quad \text{for all $(\mathbf{g},\sigma)\in G(\chi,z)^{\op}$,}
$$

\smallskip

\item a map $U\colon B\to B$ such that

\begin{itemize}

\smallskip

\item[-] $U\circ \alpha=\alpha\circ U$,

\smallskip

\item[-] $U^2=0$,

\smallskip

\item[-] $U\bigl(B_{\mathbf{g},\sigma}\bigr)\subseteq B_{\mathbf{g}z,\sigma}$\quad \text{for all $(\mathbf{g},\sigma)\in G(\chi,z)^{\op}$,}

\smallskip

\item[-] the equality
$$
\quad\qquad\qquad U(bc)=bU(c) + (-1)^{i_1+i_2+i_3}U(b)\alpha(c)
$$
holds for all $b\in B$ and $c\in B_{(g^{i_1},g^{i_2},g^{i_3})}:= \bigoplus_{\sigma\in S_3} B_{(g^{i_1},g^{i_2},g^{i_3}),\sigma}$.

\end{itemize}

\end{enumerate}

\begin{remark}\label{caso clasico comodulo algebras ejemplo 1} We are in the classical case (i.e. $s$ is the flip) iff $B_{\mathbf{g},\sigma}=0$ for $\sigma \ne \ide$ and $\alpha$ is the identity map. So, in this case the gradation in item~(a) is a $G$-gradation, item~(b) is trivial, and item~(c) is considerably simplified.
\end{remark}

\subsubsection{Right $\bm{H_{\mathcal{D}}}$-cleft extensions} Let $C:=B^{\coH_{\mathcal{D}}}$. By Corollary~\ref{coinvarianes de un HsubD comodulo} we know that
$$
C=B_{1_G}\cap \ker(U)= \bigoplus_{\sigma\in S_3} B_{1_G,\sigma}\cap \ker(U).
$$
By Theorem~\ref{caracterizacion de cleft} and the comment bellow that result, the extension $(C\hookrightarrow B,s)$ is cleft iff there exists $b_x\in B$ and  a family $(b_\mathbf{g})_{\mathbf{g}\in G}$ of elements of $B^{\times}$, such that

\begin{enumerate}[label=\emph{(\alph*)}]

\smallskip

\item $b_{1_G}=1$,

\smallskip

\item $b_{\mathbf{g}}\in B_{\mathbf{g},\ide}\cap \ker(U)$ for all $\mathbf{g}\in G$,

\smallskip

\item $b_x\in B_{(g,g,g),\ide}\cap U^{-1}(1)$,

\smallskip

\item $\alpha(b_x)=b_x$,

\smallskip

\item $\alpha(b_{\mathbf{g}})=b_{\mathbf{g}}$ for all $\mathbf{g}\in G$.

\smallskip

\end{enumerate}
By Theorem~\ref{algunas propiedades de las extensiones cleft} we know that

\begin{enumerate}

\smallskip

\item $B$ is a free left $C$-module with basis $\{b_gb_x^i : g\in G \text{ and } 0\le i\le 1\}$.

\smallskip

\item By Theorem~\ref{algunas propiedades de las extensiones cleft}(3), the weak action of $H_{\mathcal{D}}$ on $C$ associated with $\gamma$ according to item~(5) of Theorem~\ref{equiv entre cleft, H-Galois con base normal e isomorfo a un producto cruzado}, is given by
\begin{align*}
&\quad (g^{i_1}, g^{i_2}, g^{i_3})\rightharpoonup c = b_{(g^{i_1}, g^{i_2}, g^{i_3})} c b^{-1}_{(g^{i_{\sigma(1)}}, g^{i_{\sigma(2)}} g^{i_{\sigma(3)}})}
\shortintertext{and}
&\quad (g^{i_1}, g^{i_2}, g^{i_3})x\rightharpoonup c = b_{(g^{i_1}, g^{i_2}, g^{i_3})}\bigl(b_x c - \alpha(c) b_x\bigr) b^{-1}_{(g^{i_{\sigma(1)}+1}, g^{i_{\sigma(2)}+1}, g^{i_{\sigma(3)}+1})}
\end{align*}
for $c\in B_{1_G,\sigma}\cap \ker(U)$.
\smallskip

\item By Theorem~\ref{algunas propiedades de las extensiones cleft}(3), the two cocycle $\sigma\colon H_{\mathcal{D}}\ot H_{\mathcal{D}} \to C$, associated with $\gamma$ according to item~(5) of Theorem~\ref{equiv entre cleft, H-Galois con base normal e isomorfo a un producto cruzado}, is given by
\begin{align*}
&\quad\qquad \sigma(\mathbf{g}\ot \mathbf{h}) = b_{\mathbf{g}}b_{\mathbf{h}} b_{\mathbf{g}\mathbf{h}}^{-1},\\
&\quad\qquad \sigma(\mathbf{g}x \ot \mathbf{h}) = - \chi(\mathbf{h}) b_{\mathbf{g}}b_{\mathbf{h}}  b_x b_{\mathbf{g}\mathbf{h}(g,g,g)}^{-1} +  b_{\mathbf{g}} b_x b_{\mathbf{h}} b_{\mathbf{g}\mathbf{h}(g,g,g)}^{-1},\\
&\quad\qquad \sigma(\mathbf{g} \ot \mathbf{h}x) = 0
\shortintertext{and}
&\quad\qquad \sigma(\mathbf{g}x \ot \mathbf{h}x) = \chi(h) b_{\mathbf{g}}b_{\mathbf{h}}  b_x^2 b_{\mathbf{g}\mathbf{h}}^{-1},
\end{align*}
for $\mathbf{g},\mathbf{h}\in G$.
\end{enumerate}

\begin{remark} It is clear that once choiced $b^{(1)}_g:=b_{(g,1,1)}$, $b^{(2)}_g:=b_{(1,g,1)}$ and $b^{(3)}_g:=b_{(1,1,g)}$, one can take $b_{(g,g,1)}:=b^{(1)}_gb^{(2)}_g$, $b_{(g,1,g)}:=b^{(1)}_gb^{(3)}_g$, $b_{(1,g,g)}:=b^{(2)}_gb^{(3)}_g$ and $b_{(g,g,g)}:=b^{(1)}_gb^{(2)}_gb^{(3)}_g$.
\end{remark}

\subsection{Second example} Consider the datum $\mathcal{D}=(C_6,\chi,z,\lambda,q)$, where:

\begin{itemize}

\smallskip

\item[-] $C_6=\{1,g,g^2,g^3,g^4,g^5\}$ is the multiplicative cyclic group of order~$6$,

\smallskip

\item[-] $\chi\colon C_6\longrightarrow \mathds{C}$ is the character given by $\chi(g^i):=\xi^i$, where $\xi$ is a root of order~$3$ of~$1$,

\smallskip

\item[-] $z:=g$,

\smallskip

\item[-] $q=\xi$ and $\lambda=1$.

\smallskip

\end{itemize}
In this case $p:=\chi(z)=\xi$, $n=3$ and the Hopf braided $\mathds{C}$-algebra $H_{\mathcal{D}}$ of Corollary~\ref{algebras de Hopf trenzadas de K-R} is the $\mathds{C}$-algebra generated by the group $C_6$ and an element $x$ subject to the relations
$$
x^3=g^3-1 = -2\quad \text{and}\quad xg=\xi gx,
$$
endowed with the braided Hopf algebra structure with comultiplication map $\Delta$, counit $\epsilon$, antipode $S$ and braid $c_{\xi}$, given by
\begin{align*}
&\Delta(g^i):= g^i\ot g^i, && \Delta(x):= 1\ot x+x\ot g,\\
& \epsilon(g^i):=1, && \epsilon(x):=0\\
& S(g^ix^j):=(-1)^j\xi^{j(j-1)}x^j g^{-j-i}, \\
&c_{\xi}(g^ix^j\ot g^kx^l)=\xi^{jl} g^kx^l\ot g^ix^j.
\end{align*}
It is clear that $\Aut_{\chi,z}(C_6)=\{\ide\}$.

\subsubsection{$\bm{H_{\mathcal{D}}}$-spaces}\label{ex2p1}  Let $V$ be a $\mathds{C}$-vector space. By Proposition~\ref{estructuras trenzadas de HsubD} we know that to have an $H_{\mathcal{D}}$-space structure with underlying vector space $V$ is equivalent to have an automorphism~$\alpha\colon V\to V$ such that $\alpha^3=\ide$. The structure map $s\colon H_{\mathcal{D}}\ot V\longrightarrow V\ot H_{\mathcal{D}}$ construct from these data is given by
$$
s\bigl(g^ix^j\ot v\bigr) := \alpha^j(v)\ot g^ix^j.
$$

\subsubsection{$\bm{H_{\mathcal{D}}}$-comodules} Let $V$ be a $\mathds{C}$-vector space. By Corollary~\ref{segunda caracterizacion de HsubD comodules} the right $H_{\mathcal{D}}$-comodule structures $(V,s)$ with underlying vector space $V$ are univocally determined by the following data:

\begin{enumerate}[label=\emph{(\alph*)}]

\smallskip

\item a decomposition
$$
\quad\qquad V=\bigoplus_{g^i\in C_6} V_{g^i}=V_1\oplus V_g \oplus V_{g^2}\oplus V_{g^3} \oplus V_{g^4}\oplus V_{g^5},
$$

\smallskip

\item an automorphism $\alpha\colon V\to V$ that satisfies $\alpha^3=\ide$ and $\alpha\bigl(V_{g^i}\bigr) = V_{g^i}$ for all $i$,

\smallskip

\item a map $U\colon V\to V$ such that
$$
\quad\qquad U\circ \alpha= \xi\alpha\circ U, \quad U^3=0\quad\text{and}\quad U\bigl(V_{g^i}\bigr)\subseteq V_{g^{i-1}}\quad\text{for all $i$.}
$$

\smallskip
\end{enumerate}
The formula for the transposition $s$ of $H_{\mathcal{D}}$ on $V$ is the one obtained in Subsection~\ref{ex2p1}, while the $H_{\mathcal{D}}$-coaction $\nu$ is given by
$$
\nu(v) = v\ot g^i + U(v)\ot g^{i-1}x - \xi U^2(v)\ot g^{i-2}x^2 \qquad\text{for all $v\in V_{g^i}$.}
$$

\subsubsection{Transpositions of $\bm{H_{\mathcal{D}}}$ on an algebra} By Proposition~\ref{estructura de transposiciones de HsubD}, for each $\mathds{C}$-algebra $B$, to have a transposition $s\colon H_{\mathcal{D}}\ot B \longrightarrow B\ot H_{\mathcal{D}}$ is equivalent to have an automorphism of algebras $\alpha\colon B\to B$ such that $\alpha^3=\ide$. The structure map $s\colon H_{\mathcal{D}}\ot B\longrightarrow B\ot H_{\mathcal{D}}$, constructed from these data, is the same as in Subsection~\ref{ex2p1}.

\subsubsection{Right $\bm{H_{\mathcal{D}}}$-comodule algebras} Let $B$ be a $\mathds{C}$-algebra. By Theorem~\ref{caracterizacion de HsubD comodule algebras} to have a right $H_{\mathcal{D}}$-comodule algebra $(B,s)$ is equivalent to have

\begin{enumerate}[label=\emph{(\alph*)}]

\smallskip

\item a $C_6$-gradation
$$
\qquad\quad B= B_1\oplus B_g \oplus B_{g^2}\oplus B_{g^3} \oplus B_{g^4}\oplus B_{g^5},
$$
of $B$ as a vector space such that $1_B\in B_1$ and $B_{g^i}B_{g^j}\subseteq B_{g^{i+j}}\oplus B_{g^{i+j-3}}$ for all $i,j$.

\smallskip

\item an automorphism of algebras $\alpha\colon B\to B$ such that
$$
\qquad\quad \alpha^3=\ide\quad\text{and}\quad\alpha\bigl(B_{g^i}\bigr) \subseteq B_{g^i}\quad \text{for all $i$,}
$$

\smallskip

\item a map $U\colon B\to B$ such that

\begin{itemize}

\smallskip

\item[-] $U\circ \alpha=\xi \alpha\circ U$,

\smallskip

\item[-] $U^3=0$,

\smallskip

\item[-] $U\bigl(B_{g^i}\bigr)\subseteq B_{g^{i-1}}$\quad \text{for all $i$,}

\smallskip

\item[-] the equality
$$
\quad\qquad\qquad U(bc)=bU(c) + \xi^i U(b)\alpha(c)
$$
holds for all $b\in B$ and $c\in B_{g^i}$,

\smallskip

\item[-] For $b\in B_{g^i}$ and $c\in B_{g^j}$, the component $(bc)_{g^{i+j-3}}\in B_{g^{i+j-3}}$ of $bc$ is given by
$$
\quad\qquad (bc)_{g^{i+j-3}}= \xi^j U(b) \alpha(U^2(c)) + \xi^{2j} U^2(b) \alpha^2(U(c)).
$$
\end{itemize}
\end{enumerate}

\subsubsection{Right $\bm{H_{\mathcal{D}}}$-cleft extensions}
Let $C:=B^{\coH_{\mathcal{D}}}$. By Corollary~\ref{coinvarianes de un HsubD comodulo} we know that
$$
C=B_{1}\cap \ker(U).
$$
By Theorem~\ref{caracterizacion de cleft} and the comment bellow that result, the extension $(C\hookrightarrow B,s)$ is cleft iff there exists $b_x\in B$ and  a family $(b_{g^i})_{g^i\in C_6}$ of elements of $B^{\times}$, such that
\begin{enumerate}[label=\emph{(\alph*)}]

\smallskip

\item $b_1=1$,

\smallskip

\item $b_{g^i}\in B_{g^i}\cap \ker(U)$ for all $g^i\in C_6$,

\smallskip

\item $b_x\in B_{g}\cap U^{-1}(1)$,

\smallskip

\item $\alpha(b_x)=\xi b_x$,

\smallskip

\item $\alpha(b_{g^i})=b_{g^i}$ for all $g^i\in C_6$.

\end{enumerate}
By Theorem~\ref{algunas propiedades de las extensiones cleft} we know that

\begin{enumerate}

\smallskip

\item $B$ is a free left $C$-module with basis $\{b_{g^i}b_x^j : g^i\in C_6 \text{ and } 0\le j\le 2\}$.

\smallskip

\item The weak action of $H_{\mathcal{D}}$ on $C$ associated with $\gamma$ according to item~(5) of Theorem~\ref{equiv entre cleft, H-Galois con base normal e isomorfo a un producto cruzado}, is given by
\begin{align*}
&\quad  g^i\rightharpoonup c = b_{g^i} c b^{-1}_{g^i},\\
& \quad g^i x \rightharpoonup c = b_{g^i} \bigl(b_x c - \alpha(c)b_x\bigr) b^{-1}_{g^{i+1}}
\shortintertext{and}
&\quad g^i x^2 \rightharpoonup c 
%
= b_{g^i} \bigl( b^2_x c + \xi b_x\alpha(c)b_x + \alpha^2(c)b^2_x\bigr) b^{-1}_{g^{i+2}},
\end{align*}
for $c\in C$.

\smallskip

\item The two cocycle $\sigma\colon H_{\mathcal{D}}\ot H_{\mathcal{D}} \to C$, associated with $\gamma$ according to item~(5) of Theorem~\ref{equiv entre cleft, H-Galois con base normal e isomorfo a un producto cruzado}, is given by
\begin{align*}
&\quad\qquad \sigma(g^i\ot g^j) = b_{g^i}b_{g^j} b_{g^{i+j}}^{-1},\\
&\quad\qquad \sigma(g^i x\ot g^j) = -\xi^j b_{g^i}b_{g^j} b_x b_{g^{i+j+1}}^{-1} + b_{g^i}b_x b_{g^j}  b_{g^{i+j+1}}^{-1},\\
&\quad\qquad \sigma(g^i x^2\ot g^j) = \xi^{2j+2} b_{g^i}b_{g^j} b^2_x b_{g^{i+j+2}}^{-1} + \xi^{j+1}b_{g^i}b_xb_{g^j} b_x b_{g^{i+j+2}}^{-1}+ b_{g^i}b^2_xb_{g^j} b_{g^{i+j+2}}^{-1},\\
&\quad\qquad \sigma(g^i \ot g^j x) = 0,\\
&\quad\qquad \sigma(g^i x \ot g^j x) = 0,\\
&\quad\qquad \sigma(g^i x^2\ot g^j x) = -\xi^{2j} b_{g^i}b_{g^j}  b_{g^{i+j+3}}^{-1} -\xi^{2j} b_{g^i}b_{g^j} b_{g^{i+j}}^{-1} + \xi^{2j} b_{g^i}b_{g^j}  b_x^3 b_{g^{i+j+3}}^{-1},\\
&\quad\qquad \sigma(g^i  \ot g^j x^2) = 0,\\
&\quad\qquad \sigma(g^i x\ot g^j x^2) = -\xi^{j} b_{g^i}b_{g^j}b_{g^{i+j+3}}^{-1} -\xi^{j} b_{g^i}b_{g^j} b_{g^{i+j}}^{-1} + \xi^{j} b_{g^i}b_{g^j}b_x^3 b_{g^{i+j+3}}^{-1}
\shortintertext{and}
&\quad\qquad \sigma(g^i x^2\ot g^j x^2)  = -\xi^{2j+1} b_{g^i}b_{g^j}b_x b_{g^{i+j+4}}^{-1} -\xi^{2j+1} b_{g^i}b_{g^j}b_x b_{g^{i+j+1}}^{-1} + \xi^{j+1} b_{g^i}b_xb_{g^j}b_{g^{i+j+4}}^{-1} \\
&\phantom{\quad\qquad \sigma(g^i x^2\ot g^j x^2)} + \xi^{j+1} b_{g^i}b_xb_{g^j}b_{g^{i+j+1}}^{-1} + \xi^{2j+1} b_{g^i}b_{g^j}b^4_x b_{g^{i+j+4}}^{-1} - \xi^{j+1} b_{g^i}b_xb_{g^j}b^3_x b_{g^{i+j+4}}^{-1}.
\end{align*}
\end{enumerate}

\end{document}